\documentclass[11pt]{amsart}
\usepackage{amscd}      
\usepackage{amssymb}
\usepackage{amsmath, amsthm, graphics}
\usepackage[normalem]{ulem}
\usepackage{color}
\usepackage{xypic}      
\LaTeXdiagrams          
\usepackage[all,v2]{xy}
\xyoption{2cell} \UseAllTwocells \xyoption{frame} \CompileMatrices
\allowdisplaybreaks[4]

\usepackage[margin=1.05in]{geometry}

\usepackage{latexsym}
\usepackage{epsfig}
\usepackage{amsfonts}
\usepackage{enumerate}
\usepackage{times}

\newtheorem{prop}{Proposition}[section]

\newtheorem{question}[prop]{Question}

\newtheorem{numcon}[prop]{Construction}

\newtheorem{claim}{Claim}

\theoremstyle{plain}
        \newtheorem{theorem}{Theorem}[section]
        \newtheorem{lemma}[theorem]{Lemma}
        \newtheorem{remark}[theorem]{Remark}
        \newtheorem{proposition}[theorem]{Proposition}
        \newtheorem{corollary}[theorem]{Corollary}
        \newtheorem{conjecture}[theorem]{Conjecture}
        
\theoremstyle{definition}
        \newtheorem{definition}[theorem]{Definition}

\newtheorem{assump}[theorem]{Assumption}

\theoremstyle{remark}

\theoremstyle{remark}

\numberwithin{equation}{section}

\newcommand{\Mbar}{\overline{\M}}

\newcommand{\com}{\mathbb{C}}

\newcommand{\X}{\mathcal{X}}
\newcommand{\Y}{\mathcal{Y}}

\newcommand{\M}{\mathcal{M}}

\newcommand{\B}{\mathcal{B}}

\newcommand{\D}{\mathcal{D}}
\newcommand{\F}{\mathcal{F}}

\newcommand{\sL}{\mathcal{L}}
\newcommand{\sI}{\mathcal{I}}

\def\<{\left\langle}
\def\>{\right\rangle}

\newcommand{\complex}{{\mathbb C}}
\newcommand{\rational}{{\mathbb Q}}

\newcommand{\reals}{{\mathbb R}}

\newcommand{\integers}{{\mathbb Z}}

\newcommand{\cala}{{\mathcal A}}
\newcommand{\calb}{{\mathcal B}}
\newcommand{\calc}{{\mathcal C}}

\newcommand{\calg}{{\mathcal G}}
\newcommand{\calh}{{\mathcal H}}
\newcommand{\cali}{{\mathcal I}}

\newcommand{\calk}{{\mathcal K}}
\newcommand{\call}{{\mathcal L}}
\newcommand{\calm}{{\mathcal M}}
\newcommand{\caln}{{\mathcal N}}

\newcommand{\calq}{{\mathcal Q}}

\newcommand{\cals}{{\mathcal S}}

\newcommand{\calv}{{\mathcal V}}
\newcommand{\calu}{{\mathcal U}}
\newcommand{\calw}{{\mathcal W}}

\newcommand{\Sh}{\operatorname{Sh}}
\newcommand{\Ad}{\operatorname{Ad}}
\newcommand{\End}{\operatorname{End}}
\newcommand{\tr}{\operatorname{tr}}
\newcommand{\orb}{\operatorname{Orb}}
\newcommand{\stab}{\operatorname{Stab}}
\newcommand{\age}{\operatorname{age}}
\newcommand{\frakQ}{{\mathfrak {Q}}}
\newcommand{\frakH}{{\mathfrak {H}}}
\newcommand{\frakG}{{\mathfrak {G}}}
\newcommand{\frakc}{{c}}
\newcommand{\wfrakQ}{\widehat{\mathfrak{Q}}}
\newcommand{\wfrakH}{\widehat{\mathfrak{H}}}

\newcommand{\bi}{\mathbf{i}}
\newcommand{\bj}{\mathbf{j}}

\title{Duality theorems for \'etale gerbes on orbifolds}
\author{Xiang Tang}
\thanks{(X. T.) Department of Mathematics, Washington University, St. Louis, MO 63130, USA, xtang@math.wustl.edu.}
\author{Hsian-Hua Tseng}
\thanks{(H.-H. T.) Department of Mathematics, Ohio State University, Columbus, OH 43210, USA, hhtseng@math.ohio-state.edu.}
\begin{document}
\date{\today}
\begin{abstract}
Let $G$ be a finite group and $\Y$ a $G$-gerbe over an orbifold
$\B$. A disconnected orbifold $\widehat{\Y}$ and a flat $U(1)$-gerbe
$c$ on $\widehat{\Y}$ is canonically constructed from $\Y$.
Motivated by a proposal in physics, we study a mathematical duality
between the geometry of the $G$-gerbe $\Y$ and the geometry of
$\widehat{\Y}$ {\em twisted by} $c$. We prove several results verifying this duality in the contexts of non-commutative geometry and
symplectic topology. In particular, we prove that the category of
sheaves on $\Y$ is equivalent to the category of $c$-twisted sheaves
on $\widehat{\Y}$. When $\Y$ is symplectic, we show, by a
combination of techniques from non-commutative geometry and
symplectic topology, that the Chen-Ruan orbifold cohomology of $\Y$
is isomorphic to the $c$-twisted orbifold cohomology of
$\widehat{\Y}$ as graded algebras.
\end{abstract}

\maketitle \tableofcontents

\section{Introduction}
\subsection{Background}
The notion of an orbifold\footnote{Throughout this paper we consider
orbifolds which are not necessarily reduced.} was first introduced
in \cite{sa} under the name ``$V$-manifold,'' and was introduced in
algebraic geometry in \cite{deligne-mumford},  and is now called a
Deligne-Mumford stack. The term ``orbifold" was coined by Thurston
\cite{th} during his study of 3-dimensional manifolds. Orbifolds are
geometric objects that are locally modeled on quotients of manifolds
by actions of finite groups. Introductory accounts about orbifolds
can be found in \cite{Adem_Leida_Ruan}, \cite{lm-b}, and \cite{mo-pr}.

Besides being interesting in its own right, the theory of orbifolds
can be applied in numerous areas, such as the study of moduli
problems and quotient singularities. Moreover, there has been an
increase of activities in the study of the stringy geometry of
orbifolds. See \cite{Adem_Leida_Ruan}, \cite{ruan_0011149}, and \cite{ruan_0201123} for expository accounts.

In this paper, we study a special kind of orbifolds called {\em
gerbes}. Let $G$ be a finite group and $BG=[pt/G]$ the
classifying orbifold of $G$. Roughly speaking, one can think of a
$G$-gerbe over an orbifold $\B$ as a $BG$-bundle over
$\B$. Then in order to define a $G$-gerbe $\Y$ over $\B$, one starts
with an open cover $\{U_i\}$ of $\B$ and specifies the following
data:
\begin{equation}\label{intro:gerbe_data}
\begin{split}
&\varphi_{ij}\in Aut(G) \quad \text{ for each double overlap } U_{ij}:=U_i\cap U_j, \text{ and }\\
&g_{ijk}\in G\quad \text{ for each triple overlap } U_{ijk}:=U_i\cap
U_j\cap U_k,
\end{split}
\end{equation}
so that the following constraints are satisfied:
\begin{equation}\label{intro:gerbe_constraints}
\begin{split}
&\varphi_{jk}\circ\varphi_{ij}=\text{Ad}_{g_{ijk}}\circ \varphi_{ik}, \quad \text{ on } U_{ijk},\\
&g_{jkl}g_{ijl}=\varphi_{kl}(g_{ijk})g_{ikl},\quad \text{ on }
U_i\cap U_j\cap U_k\cap U_l.
\end{split}
\end{equation}
Here, $\text{Ad}_{g}: G\to G$ denotes the map of conjugation by $g$.
The data in (\ref{intro:gerbe_data}) are then used to glue $U_i\times BG$
together to form a $G$-gerbe $\Y$ together with an associated map
$\Y\to \B$.

We easily see that $BG$ is the unique $G$-gerbe over a point. Gerbes
arise naturally from ineffective group actions. For example, let $M$
be a manifold and $H$ a compact group that acts on $M$ with
finite stabilizers at every point. The quotient space $[M/H]$ is an
orbifold. Suppose that $G$ is a finite normal subgroup of $H$ such
that the induced action of $G$ on $M$ is trivial. Then there is an
induced action of the quotient group $Q:=H/G$ on $M$. The orbifold
$[M/H]$ defines a $G$-gerbe
$[M/H]\to [M/Q]$
over the orbifold $[M/Q]$. In general, gerbes play an important role
in the structure theory of orbifolds. For example, given an orbifold
$\X$ there is a finite group $G$ and a {\em reduced} orbifold $\X'$
such that $\X$ is a $G$-gerbe over $\X'$. See \cite[Proposition
4.6]{bn}. Introductory accounts about gerbes can be found in
 \cite{EHKV}, \cite{gi}, and \cite{la-st-xu}.

The purpose of this paper is to study the geometry and topology of
$G$-gerbes. Our study is motivated and inspired by results in the
physics paper \cite{hel-hen-pan-sh}. Given a $G$-gerbe $\Y\to \B$,
the authors of \cite{hel-hen-pan-sh} construct a disconnected space
$\widehat{\Y}$ with a map $\widehat{\Y}\to \B$ and a flat
$U(1)$-gerbe $c$ on $\widehat{\Y}$. This construction is reviewed in
Sec.  \ref{intro:dual_construction} below. The main point of
\cite{hel-hen-pan-sh} is the conjecture which asserts that the
conformal field theories on the $G$-gerbe $\Y$ are equivalent to the
corresponding conformal field theories on $\widehat{\Y}$ {\em
twisted by the B-field} $c$. This conjecture 
suggests the existence of a certain duality between the
$G$-gerbe $\Y$ and the pair $(\widehat{\Y},c)$. Our viewpoint toward
this conjecture is the following claim:\\

{\bf ($\star$) {\em The geometry/topology of the $G$-gerbe $\Y$ is equivalent to the geometry/topology of $\widehat{\Y}$ twisted by $c$.}}\\

The claim ($\star$) reveals a deep and highly nontrivial connection
between different geometric spaces. Let us look at the simplest $G$-gerbe, namely, a $G$-gerbe over a point, (i.e.,
$\B=pt$ and $\Y=[pt/G]=BG$). The dual orbifold $\widehat{\Y}$
is the discrete set $\widehat{G}$, the space of isomorphism classes
of irreducible unitary $G$-representations, with a trivial $U(1)$-gerbe $c$ on
$\widehat{\Y}$. In this case, the claim ($\star$) states
that the geometry/topology of the classifying space $BG$ is
equivalent to the geometry/topology of the discrete set
$\widehat{G}$. Such a relationship is not clear at all at
the level of spaces. For example, when $G=\integers_2$, there does
not seem to be any obvious geometric connection between the space
$B\mathbb{Z}_2$ (interpreted either as an orbifold
$[pt/\mathbb{Z}_2]$ or as the space $\mathbb{RP}^\infty$) and the
space $\widehat{\integers}_2$. In general, to our best knowledge
there is no known geometric relation at the level of spaces between
a $G$-gerbe $\Y$ and the orbifold $\widehat{\Y}$ with the
$U(1)$-gerbe $c$.

One observes that a natural place where both $BG$ and $\widehat{G}$
appear is representation theory, since $BG$ encodes information
about principal $G$-bundles, and $\widehat{G}$ is defined to be the
set of isomorphism classes of irreducible $G$-representations.
Noncommutative geometry is a powerful modern approach to
representation theory. Thus, it makes sense to consider possible
relations between $BG$ and $\widehat{G}$ in noncommutative geometry.
In noncommutative geometry, $BG$ is represented by the group algebra
$\complex G$, and $\widehat{G}$ is represented by the commutative
algebra, $C(\widehat{G})$, of functions on $\widehat{G}$. By a
classical result, the group algebra $\complex G$ is Morita
equivalent to the algebra $C(\widehat{G})$. We can interpret this as
saying that the two spaces $BG$ and $\widehat{G}$ are ``equivalent"
from the viewpoint of noncommutative geometry. This observation
strongly suggests that noncommutative geometry naturally relates the
two geometries of $\Y$ and $(\widehat{\Y}, c)$, which appear to be
very different in the classical geometric/topological viewpoints.

Indeed, the main theme of this paper is  using tools from
noncommutative geometry to build a bridge connecting the $G$-gerbe
$\Y$ and the dual $\widehat{\Y}$ with the $U(1)$-gerbe $c$. Such a
bridge turns out to be very useful. We will prove a number of
results that show ($\star$) is true in the contexts of both
noncommutative geometry and symplectic topology, and in particular,
in Gromov-Witten theory. 
\subsection{The dual of an \'etale gerbe}\label{intro:dual_construction}
Given a $G$-gerbe over an orbifold $\B$,  $\Y\to\B,$ we describe the dual of the gerbe $\Y$
following \cite{hel-hen-pan-sh}. Consider the group of outer
automorphisms of $G$, $Out(G)=Aut(G)/Inn(G),$ i.e., the quotient of the
group $Aut(G)$ of automorphisms of $G$ by the normal subgroup
$Inn(G)$ of inner automorphisms. Associated to $\Y\to
\B$, there is a naturally defined $Out(G)$-bundle,
$\overline{\Y}\to \B,$ called the {\em band} of  $\Y$. In the
description (\ref{intro:gerbe_data}) of the $G$-gerbe $\Y$, let
$\phi_{ij}\in Out(G)$ be the image of $\varphi_{ij}$ under the
quotient map $Aut(G)\to Out(G)$. By (\ref{intro:gerbe_constraints})
we have $\phi_{jk}\circ\phi_{ij}=\phi_{ik}$ on $U_{ijk}$. Thus, the
collection $\{\phi_{ij}\}$ defines an $Out(G)$-bundle
$\overline{\Y}$ over $\B$.

We view
$\widehat{G}$ as a disjoint union of points, whose cardinality is equal to the number of conjugacy classes of $G$. A right action of
$Out(G)$ on $\widehat{G}$ is defined as follows. Given an
irreducible $G$-representation $\rho: G\to End(V_\rho)$ and $\phi\in
Aut(G)$, the composite map
$\rho\circ \phi: G\to End(V_\rho)$
is an irreducible representation of $G$. The action of $\phi$ on the
class $[\rho]$ is defined to be $[\rho\circ\phi]$. This defines a
right action of $Out(G)$ on $\widehat{G}$ because inner
automorphisms preserve isomorphism classes of irreducible
representations of $G$. Note that the isomorphism class of the trivial representation of $G$ is fixed by this
$Out(G)$ action.

Following \cite{hel-hen-pan-sh}, we define the dual space to be the
associated bundle
\begin{equation}\label{defn:dual_space}
\widehat{\Y}:=[(\overline{\Y} \times\widehat{G})/Out(G)].
\end{equation}
There is a natural map, $\widehat{\Y}\to \B$, induced from the map
$\overline{\Y}\to \B$. It is easy to see that $\widehat{G}$
decomposes into a union of $Out(G)$ orbits, and $\widehat{\Y}$
decomposes into a union of components with respect to the $Out(G)$
orbits.

Next, we define a $U(1)$-gerbe, $c$, on $\widehat{\Y}$. For each
isomorphism class $[\rho]\in \widehat{G}$, we fix a representation
$\rho: G\to End(V_\rho)$. To a point $(x, [\rho])\in
\overline{\Y}\times \widehat{G}$, we assign the vector space $V_\rho$. This defines a
family of vector spaces over $\widehat{\Y}$, which is in general
{\em not} a vector bundle over $\widehat{\Y}$. The obstruction to
forming a vector bundle over $\widehat{\Y}$ with the fiber over $(x,
[\rho])$ being $V_\rho$ is a $U(1)$-gerbe, $c^{-1}$, on
$\widehat{\Y}$. The $U(1)$-gerbe $c$ over $\widehat{\Y}$ is obtained
from $c^{-1}$ by applying the group homomorphism $U(1)\to U(1),
a\mapsto a^{-1}$. As observed in \cite{hel-hen-pan-sh}, the
$U(1)$-gerbe $c$ is flat, and the isomorphism class of $c$ is a {\em
torsion} class in the cohomology $H^2(\widehat{\Y},U(1))$.

\subsubsection*{Remark on terminology}
It is well-known that the isomorphism classes of flat $U(1)$-gerbes
over a space $\X$ are in bijective correspondence with the torsion
classes of the cohomology group $H^2(\X, U(1))$. If we fix a
sufficiently fine open cover, $\{V_i\}$, of $\X$, then the flat
$U(1)$-gerbes on $\X$ ({\em not} their isomorphism classes) are in
bijective correspondence with the $U(1)$-valued \v{C}ech
$2$-cocycles with respect to the cover $\{V_i\}$, which represent
the torsion classes in $H^2(\X, U(1))$.

In our study of the $G$-gerbe $\Y\to \B$ and the dual pair
$(\widehat{\Y}, c)$, we often choose a presentation of the orbifold
$\B$ arising from a sufficiently fine open cover of $\B$. Such a
presentation yields an open cover of $\widehat{\Y}$, and we often
represent the $U(1)$-gerbe $c$ by a $U(1)$-valued \v{C}ech
$2$-cocycle with respect to this cover. In view of the
aforementioned correspondence, in what follows, we abuse the
notation and let $c$ denote either the $U(1)$-gerbe on
$\widehat{\Y}$ or the $U(1)$-valued \v{C}ech $2$-cocycle on
$\widehat{\Y}$. We will call $c$ a $U(1)$-gerbe if a presentation of
$\widehat{\Y}$ is not chosen and call $c$ a $U(1)$-valued
$2$-cocycle if a presentation of $\widehat{\Y}$ is chosen.

\subsection{Gerbes arising from group extensions}
The simplest examples of $G$-gerbes other than $BG$ itself are of
the form $BH\to BQ$, where $H$ and $Q$ are finite groups, and $BH$
(respectively, $BQ$) is the classifying orbifold of $H$ (respectively,
$Q$). Such an example arises from an extension of a finite group $Q$
by $G$, namely, an exact sequence
$1\rightarrow G\rightarrow H \rightarrow Q\rightarrow 1.$
Our study of the gerbe $BH\to BQ$ uses knowledge about finite group
extensions substantially.

By construction, the dual space of $BH$ is given by
$\widehat{BH}=[\widehat{G}/Q]$. Here, the (right) $Q$-action on
$\widehat{G}$ is defined by a group homomorphism\footnote{The homomorphism $Q\to Out(G)$ defines an $Out(G)$-bundle over $BQ$,
which is the band of the $G$-gerbe $BH\to BQ$.} $Q\to Out(G)$, which
is constructed as follows. Choose\footnote{It is easy to see that
the resulting homomorphism, $Q\to Out(G)$, does not depend on the
choice of such a section.} a section, $s: Q\to H$, of the group
extension such that $s(1)=1$. Given $q\in Q$, we define an
automorphism of $G$ by $G\ni g\mapsto \text{Ad}_{s(q)}(g).$
The homomorphism $Q\to Out(G)$ sends $q$ to the image of the above
automorphism of $G$. The $U(1)$-gerbe on $\widehat{BH}$ can be
represented by a function $c: \widehat{G}\times Q\times Q\to U(1)$.
See Proposition \ref{prop:u(1)-cocycle} for more details.

One object naturally associated with the group $H$ is its
group algebra $\com H$. Given the $Q$-action on $\widehat{G}$ and
the function $c$, one can define the {\em twisted groupoid algebra}
$C(\widehat{G}\rtimes Q, c).$
The construction, a special case of a construction in
\cite{tu-la-xu}, is explained in Sec.  \ref{sec:mackey}. Our first result for the gerbe $BH\to BQ$ is:
\begin{theorem}[=Theorem \ref{thm:local-mackey}]\label{Intro:thm:local-mackey}
The group algebra $\complex H$ is Morita equivalent to the twisted
groupoid algebra $C(\widehat{G}\rtimes Q, c)$.
\end{theorem}

Theorem \ref{Intro:thm:local-mackey} can be interpreted as {\em Mackey's machine} in the
case of finite group extensions, which is well-studied. However, our
formulation using the language of Morita equivalence seems to be
new. The generalization of this theorem serves as a crucial step in
our study of $G$-gerbes. We prove Theorem \ref{Intro:thm:local-mackey} by explicit
constructions of bimodules that realize the Morita equivalence.
These constructions also allow us to analyze the structure of the
induced isomorphism $I: Z(\complex H)\to Z(C(\widehat{G}\rtimes Q,
c))$ between centers. See Proposition\ \ref{prop:conjugacy}.
\subsection{Non-commutative geometry}
Theorem \ref{Intro:thm:local-mackey} is best understood in the
context of non-commutative geometry. According to the viewpoint of
non-commutative differential geometry \`a la A. Connes, the geometry
of an orbifold is encoded in the Morita equivalence class of the
groupoid algebras of its groupoid presentations\footnote{It is known
that groupoid algebras arising from different groupoid presentations
of the same orbifold are Morita equivalent.}, and the geometry of an
orbifold with a $U(1)$-gerbe is represented by the Morita
equivalence class of the associated twisted groupoid algebras.
Therefore, Theorem \ref{Intro:thm:local-mackey} can be interpreted
as saying that the non-commutative geometry of $BH$ is equivalent to
the non-commutative geometry of the dual $(\widehat{BH}, c)$. This
provides an example of the claim {\bf ($\star$)}.

One of the main results of this paper is a generalization of Theorem
\ref{Intro:thm:local-mackey}. Let $\Y\to
\B$ be a $G$-gerbe over an orbifold $\B$. Denote by
$\widehat{\Y}$ the dual space of the $G$-gerbe and by $c$ the
$U(1)$-gerbe on $\widehat{\Y}$. As explained in Sec. 
\ref{subsec:general}, we pick groupoid presentations $\frakH$
(respectively, $\frakQ$) of $\Y$ (respectively, $\B$) so that the
dual space $\widehat{\Y}$ is represented by a transformation
groupoid $\widehat{G}\rtimes \frakQ.$ 
The $U(1)$-gerbe on $\widehat{\Y}$ can be presented by a locally
constant groupoid 2-cocycle $c$, see Propositions
\ref{prop:u(1)-cocycle-gpd}-\ref{prop:local-const}. The Mackey machine provides a
beautiful bridge between $\Y$ and $(\widehat{\Y},c)$ in the context
of noncommutative geometry. Let $\cala$ be an
$\frakH$-sheaf\footnote{We refer the readers to Sec. \ref{subsec:algebra} for the definition of an $\frakH$-sheaf.} of unital algebras, and let $\widetilde{\cala}$ be the
corresponding $\widehat{G}\rtimes \frakQ$-sheaf defined by $\cala$.
We consider the cross-product algebra $\cala\rtimes \frakH$
associated with the groupoid $\frakH$ and the {\em twisted}
cross-product algebra $\widetilde{\cala}\rtimes_c(\widehat{G}\rtimes
\frakQ)$ associated with $\widehat{G}\rtimes \frakQ$ and $c$, as $c$
is locally constant. We prove the following:

\begin{theorem}[=Theorem \ref{thm:global-mackey}]\label{Intro:thm:global-mackey}
The crossed product algebra $\cala\rtimes \frakH$  is Morita equivalent to the twisted crossed product algebra $\widetilde{\cala}\rtimes_c(\widehat{G}\rtimes \frakQ)$.
\end{theorem}

When $\cala$ is the sheaf $\calc^\infty$ of smooth functions on
$\frakH_0$ (the unit space of the groupoid $\frakH\rightrightarrows \frakH_0$), Theorem \ref{Intro:thm:global-mackey} shows that the
groupoid algebra $C^\infty_c(\frakH)$ is Morita equivalent to the
$c$-twisted groupoid algebra
$\calc^\infty\rtimes_c(\widehat{G}\rtimes \frakQ)$.

Consider the symplectic case, namely, the base $\B$ is assumed to be
symplectic. Then both the gerbe $\Y$ and its dual $\widehat{\Y}$ can
be equipped with symplectic structures pulled back from the one on
$\B$. In this case, $\frakH_0$ is equipped with an $\frakH$-invariant
symplectic form, and an $\frakH$-sheaf $\cala^{((\hbar))}$ of
deformation quantization on $\frakH_0$ is constructed in \cite{ta}
via Fedosov's construction. The crossed product algebra
$\cala^{((\hbar))}\rtimes \frakH$ is a deformation quantization of
the groupoid algebra $C_c^\infty(\frakH)$, and
$\widetilde{\cala}^{((\hbar))}\rtimes_c(\widehat{G}\rtimes \frakQ)$
is a deformation quantization of the $c$-twisted groupoid algebra.
Because our cocycle $c$ is locally constant, the  algebra
$\widetilde{\cala}^{((\hbar))}\rtimes_c(\widehat{G}\rtimes \frakQ)$
can be constructed by following exactly the same method described in \cite{ta} (and recalled in Sec. 
\ref{subsec:gpd-algebra}). We refer  to \cite{bgnt} for a
general discussion of deformation quantizations of gerbes. Theorem
\ref{Intro:thm:global-mackey} shows that
\begin{corollary}\label{Intro:cor_global_mackey_quantized}
The two deformation quantizations $\cala^{((\hbar))}\rtimes \frakH$
and $\widetilde{\cala}^{((\hbar))}\rtimes_c(\widehat{G}\rtimes
\frakQ)$ are Morita equivalent.
\end{corollary}

We can interpret Theorem \ref{Intro:thm:global-mackey} as saying
that the non-commutative differential geometry of the gerbe $\Y\to
\B$ is equivalent to the non-commutative differential geometry of
the dual pair $(\widehat{\Y}, c)$, proving the claim {\bf ($\star$)}
in full generality in the context of non-commutative geometry.

Categories of sheaves can be used to provide an algebro-geometric
approach to non-commutative spaces. More precisely, in this
approach, one studies geometric properties of spaces by considering
properties of their categories of sheaves. For the $G$-gerbe $\Y$,
we consider sheaves on $\Y$. For the dual pair $(\widehat{\Y},c)$, we
consider the $c$-twisted sheaves on $\widehat{\Y}$. See
\cite{cal} and \cite{Lieblich} for detailed introductions to
twisted sheaves.

\begin{theorem}[=Theorem \ref{equivalence3_general}]\label{thm:category}
The abelian category of sheaves on $\Y$ is equivalent to
the abelian category of $c$-twisted sheaves on
$\widehat{\Y}$.
\end{theorem}

In the algebraic context, i.e., when both $\Y$ and $\B$ are
Deligne-Mumford stacks, this theorem is also valid for categories of
(quasi-)coherent sheaves. And Theorem \ref{thm:category} can be
interpreted as saying that the non-commutative algebraic geometry of
the gerbe $\Y\to \B$ is equivalent to the non-commutative algebraic
geometry of the dual $(\widehat{\Y}, c)$, proving the claim {\bf
($\star$)} in full generality in the context of non-commutative
algebraic geometry. This theorem suggests that in algebraic geometry
a suitably defined theory of counting invariants of (semi)stable
sheaves on $\Y$ should be equivalent to an analogous theory of
counting invariants of (semi)stable $c$-twisted sheaves on
$\widehat{\Y}$. See \cite{gt} for results in this direction.

\subsection{Hochschild cohomology}
The Hochschild cohomology of an (associative) algebra is an
important invariant of the algebra that depends only on the Morita
equivalence class of the algebra. It also plays an important role in
non-commutative geometry. See, for example, \cite{c} and
\cite{loday}. Motivated by this, we study the Hochschild cohomology
of the algebras $\cala^{((\hbar))}\rtimes \frakH$ and
$\widetilde{\cala}^{((\hbar))}\rtimes_c(\widehat{G}\rtimes \frakQ)$
in Corollary \ref{Intro:cor_global_mackey_quantized}.


In our joint work \cite{pptt} with M. Pflaum and H. Posthuma, we
found a beautiful connection between Hochschild cohomology and
symplectic topology. Namely, we proved that the Hochschild
cohomology of the deformation quantization of the groupoid algebra
arising from a presentation of a symplectic orbifold $\X$ is
additively isomorphic to the shifted de Rham cohomology
$H^{\bullet-\ell}(I\X)((\hbar))$ (with coefficients in the field
$\complex ((\hbar))$) of the {\em inertia orbifold} $$I\X:=\{(x,
(g))| x\in \X, (g)\subset \text{Iso}(x) \text{ conjugacy class of
the isotropy subgroup }\text{Iso}(x)\}.$$ Here, the grading shift
$\ell$ is given by the codimensions of the embeddings of components
of $I\X$ into $\X$. Applying this result to the algebra
$\cala^{((\hbar))}\rtimes \frakH$ arising from the $G$-gerbe $\Y$,
we obtain
$$H^{\bullet-\ell}(I\Y)((\hbar))\cong HH^\bullet(\cala^{((\hbar))}\rtimes \frakH,{\cala}^{((\hbar))}\rtimes \frakH).$$

In this paper, we generalize the calculation of \cite{pptt} to treat
the Hochschild cohomology of the deformation quantization of a
twisted groupoid algebra. This yields the following:

\begin{theorem}[=Theorem \ref{thm:hoch-coh-twisted-alg}]\label{Intro:thm:hoch-coh-twisted-alg}
The Hochschild cohomology of the algebra
$\widetilde{\cala}^{((\hbar))}\rtimes_c (\widehat{G}\rtimes \frakQ)$
is equal to the {\em $\frakc$-twisted de Rham cohomology}
$H^{\bullet-\ell}(I\widehat{\Y},\frakc)((\hbar))$ of the orbifold
$I\widehat{\Y}$ with the grading shift defined by the codimensions,
$\ell$, of the embeddings of components of $I\widehat{\Y}$ into
$\widehat{\Y}$,
$$
HH^\bullet(\widetilde{\cala}^{((\hbar))}\rtimes_c
(\widehat{G}\rtimes \frakQ), \widetilde{\cala}^{((\hbar))}\rtimes_c
(\widehat{G}\rtimes \frakQ))\cong
H^{\bullet-\ell}(I\widehat{\Y},\frakc)((\hbar)),
$$
where $H^\bullet(I\widehat{\Y}, c)$ is the de Rham cohomology of
$I\widehat{\Y}$ with coefficients in a line bundle ${\mathcal L}_c$,
which is naturally defined (Def. \ref{dfn:L-c}) by the $U(1)$-gerbe $c$.
\end{theorem}

Since Morita equivalent algebras have isomorphic Hochschild
cohomologies, Theorem \ref{Intro:thm:hoch-coh-twisted-alg} and
Corollary \ref{Intro:cor_global_mackey_quantized} yield the
following result:
\begin{theorem}[=Theorem \ref{thm:cohomology}]\label{Intro:thm:hochschild}
There are isomorphisms of cohomologies,
\[
\begin{split}
H^{\bullet-\ell}(I\Y)((\hbar))&\cong
HH^\bullet(\cala^{((\hbar))}\rtimes \frakH,{\cala}^{((\hbar))}\rtimes \frakH)\\
&\cong HH^\bullet(\widetilde{\cala}^{((\hbar))}\rtimes_c
(\widehat{G}\rtimes \frakQ), \widetilde{\cala}^{((\hbar))}\rtimes_c
(\widehat{G}\rtimes \frakQ))\cong
H^{\bullet-\ell}(I\widehat{\Y},\frakc)((\hbar)).
\end{split}
\]
Moreover, the above isomorphisms yield an isomorphism of graded
$\complex ((\hbar))$-vector spaces,
\[
H^{\bullet-2\age}(I\Y)((\hbar))\cong
H^{\bullet-2\age}(I\widehat{\Y},\frakc)((\hbar)),
\]
where the vector spaces are equipped with the {\em age grading} as
defined in Eq.  (\ref{eq:age}) (see also Definition \ref{Horb_groups}).
\end{theorem}
In fact, this isomorphism is valid over $\complex$, yielding an
isomorphism  of graded $\complex$-vector spaces,
\begin{equation}\label{add_isom_over_C}
H^{\bullet-2\text{age}}(I\Y, \complex)\cong H^{\bullet-2\text{age}}(I\widehat{\Y},\frakc, \complex).
\end{equation}
\subsection{Chen-Ruan orbifold cohomology}
In this and the next sections we discuss important applications of Theorem \ref{Intro:thm:hochschild} to the symplectic geometry/topology of $G$-gerbes.

Chen-Ruan orbifold cohomology, first introduced in
\cite{cr}, is a very important object in the theory of orbifolds,
and has generated numerous exciting research in recent years. For an almost
complex orbifold $\X$, its Chen-Ruan cohomology, denoted by
$H_{{\rm CR}}^\bullet(\X, \com),$ is additively the cohomology $H^\bullet(I\X,
\com)$ of the inertia orbifold, $I\X$, of $\X$, equipped with a
shifted grading, a non-degenerate pairing called {\em orbifold
Poincar\'e pairing}, and an associative product structure called the
{\em Chen-Ruan orbifold cup product}. Structure constants of the
Chen-Ruan orbifold cup product are defined to be certain integrals
over the {\em $2$-multi-sector}, $\X_{(2)}$, of $\X$ involving the
{\em obstruction bundle} $Ob_\X$. See Sec.  \ref{subsec:cr} for a quick review of the construction of the Chen-Ruan orbifold
cohomology. 

Given an almost complex orbifold $\X$ with a flat $U(1)$-gerbe $c$, the $c$-twisted orbifold cohomology of $\X$, introduced by Y. Ruan \cite{ru1} and 
denoted by $H_{orb}^\bullet(\X,c,\com),$ is additively the
cohomology $H^\bullet(I\X, c, \complex):=H^\bullet(I\X, \sL_c)$ of
$I\X$ with coefficients in the line bundle $\sL_c$. The line bundle
$\sL_c$ is naturally defined from the $U(1)$-gerbe $c$ and is an
example of an {\em inner local system} \cite{ru1}. The groups
$H_{orb}^\bullet(\X,c,\com)$ are equipped with a shifted grading, a
non-degenerate pairing, and an associative product structure defined
in ways similar to their counterparts in the Chen-Ruan orbifold
cohomology. In particular, the structure constants of the product
are also defined as certain integrals over the $2$-multi-sector
involving the obstruction bundle. See Sec.  \ref{subsec:cr} for a review of this construction.

Consider a $G$-gerbe $\Y$ over a compact symplectic orbifold $\B$
and its dual pair $(\widehat{\Y}, c)$. Both $\Y$ and $\widehat{\Y}$
are equipped with symplectic structures coming from the one of
$\B$. Equip $\Y$ and $\widehat{\Y}$ with compatible almost
complex structures and consider the two cohomology groups
$H_{{\rm CR}}^\bullet(\Y, \com)$ and
$H_{orb}^\bullet(\widehat{\Y},c,\com)$. We improve the  isomorphism (\ref{add_isom_over_C})  of graded $\complex$-vector spaces, i.e.,
$H_{{\rm CR}}^\bullet(\Y, \com)\simeq
H_{orb}^\bullet(\widehat{\Y},c,\com)$ in the
following:

\begin{theorem} [see Theorem \ref{thm:iso-coh-ring}]\label{Intro:thm:iso-coh-ring}
There is an isomorphism of {\em graded $\complex$-algebras},
\begin{equation}\label{eqn:coh_ring_isom}
H_{{\rm CR}}^\bullet(\Y, \com )\simeq
H_{orb}^\bullet(\widehat{\Y},c,\com ).
\end{equation}
\end{theorem}

We view this theorem as the realization of {\bf ($\star$)} at the
level of Chen-Ruan orbifold cohomology rings. The isomorphism (\ref{eqn:coh_ring_isom}) is also compatible with (twisted) orbifold Poincar\'e
pairings. See Corollary \ref{cor:poin-pairing}.

The proof of this theorem, given in Sec. 
\ref{subsec:cohom-gerbe}-\ref{subsec:pairing}, amounts to showing
that the additive isomorphism 
$H^{\bullet-2\text{age}}(I\Y,
\complex)\cong H^{\bullet-2\text{age}}(I\widehat{\Y},\frakc,
\complex),$ 
obtained in Theorem \ref{Intro:thm:hochschild}, is in
fact a ring isomorphism. The Morita equivalence bimodule in the
proof of Theorem \ref{Intro:thm:global-mackey} gives an explicit
formula of this isomorphism. To prove that this isomorphism preserves the ring structure, we
need to compare the structure constants of the orbifold cup
products. We first establish, in Proposition \ref{prop:obstruction},
a comparison result between the obstruction bundles $Ob_\Y,
Ob_{\widehat{\Y}}$, and the obstruction bundle, $Ob_\B$, of the base
$\B$. This reduces the question to comparing certain cohomology
classes on the $2$-multi-sector $\B_{(2)}$ of $\B$. See
(\ref{3pt_inv_3}). We prove (\ref{3pt_inv_3})  in Sec. 
\ref{subsec:pairing} by carefully examining the
isomorphism in Theorem \ref{Intro:thm:hochschild} via representation
theory of finite groups.

We should point out that Theorem \ref{Intro:thm:iso-coh-ring} is
somewhat surprising. Although one can construct the isomorphism
(\ref{eqn:coh_ring_isom}) directly in some special examples (such as the
toric case \cite{ajt3}), in general it is not clear {\em
at all} that the two cohomologies $H_{{\rm CR}}^\bullet(\Y, \com)$ and
$H_{orb}^\bullet(\widehat{\Y},c,\com)$ are equal even as vector
spaces, since the $G$-gerbe $\Y$ and the $U(1)$-gerbe $c$ on
$\widehat{\Y}$ are related through representation theory and their
geometric connections are somewhat obscure. Noncommutative geometry
(and, in particular, Morita equivalence) provides us the right tool
to extract the geometric information from representation theory.
From this perspective, our construction of the isomorphism
(\ref{eqn:coh_ring_isom}) via Morita equivalence and the connection
between Hochschild cohomology and orbifold cohomology is very
natural and is so far the only known construction that works in
general. Furthermore, the result that the two orbifold cohomologies
are isomorphic as rings is not a formal consequence of our results
on Hochschild cohomology. In \cite{pptt}, it is shown that the ring
structure on $H_{{\rm CR}}^\bullet(\Y, \com)$ (and also
$H_{orb}^\bullet(\widehat{\Y},c,\com)$) defined by the product on
the Hochschild cohomology of the orbifold groupoid algebra is
closely related, but {\em not} isomorphic, to the Chen-Ruan cup
product. There are certain subtle but crucial differences between
the Hochschild and Chen-Ruan cup products. Our proof of the fact
that (\ref{eqn:coh_ring_isom}) is indeed a ring isomorphism with
respect to Chen-Ruan cup products uses some delicate and important
properties of the isomorphism (\ref{add_isom_over_C}). We view our
proof of Theorem \ref{Intro:thm:iso-coh-ring} as a successful
application of non-commutative geometric techniques to the study of
symplectic topology.

\subsection{Gromov-Witten theory}
Chen-Ruan orbifold cohomology and twisted orbifold cohomology can be
considered as part of a bigger and richer theory called
Gromov-Witten theory. Let $\X$ be a compact symplectic orbifold. The
Gromov-Witten theory of $\X$ concerns Gromov-Witten
invariants of $\X$, which are integrals of certain naturally defined
cohomology classes over moduli spaces of orbifold stable maps to
$\X$. These invariants may be organized into a generating function,
$\D_\X$, called the total descendant potential of $\X$, whose
properties reflect the structures of Gromov-Witten invariants.

The Gromov-Witten theory of orbifolds is constructed in the work
\cite{cr2} in symplectic geometry and in the works \cite{agv1, agv2}
in algebraic geometry. It has been a very active research area in
the past few years. Expository accounts of this theory can also be
found in \cite{abramovich} and \cite{tseng}.

Given a flat $U(1)$-gerbe, $c$, on a compact symplectic orbifold
$\X$, a ``twist'' of the Gromov-Witten theory of $\X$ by $c$ is
constructed in the work \cite{pry}. The main ingredients here are
the $c$-twisted Gromov-Witten invariants of $\X$. These invariants
are integrals of certain naturally defined cohomology classes over
the moduli spaces of orbifold stable maps to $\X$, and they can be
organized into a generating function, $\D_{\X,c}$, called the
$c$-twisted total descendant potential of $\X$.

It is known that for Calabi-Yau target spaces the Gromov-Witten
theory may be understood as the mathematical version of a
topological twist of a conformal field theory called non-linear
sigma model. Since the original physics conjecture on the duality of
\'etale gerbes concerns the equivalence of conformal field theories,
it is very natural to consider the claim {\bf ($\star$)} in the
context of Gromov-Witten theory.

Let $\Y$ be a $G$-gerbe over a compact symplectic orbifold $\B$ and
$(\widehat{\Y},c)$ its dual pair. The claim {\bf ($\star$)} in
the context of Gromov-Witten theory is naturally formulated in the
following:
\begin{conjecture}\label{intro:dual_GW}
The generating functions $\D_\Y$ and $\D_{\widehat{\Y}, c}$ are equal after suitable changes of variables.
\end{conjecture}
One can also formulate an analogue of Conjecture \ref{intro:dual_GW}
for the {\em ancestor potentials} (see, e.g, \cite[Sec. 
5]{givental_quantization} for the definition of ancestor potential).
In order for Conjecture \ref{intro:dual_GW} to possibly be true, it
is necessary that the cohomology groups $H_{{\rm CR}}^\bullet(\Y,\com)$
and $H_{orb}^\bullet(\widehat{\Y},c,\complex)$ be isomorphic.
Therefore, Theorem \ref{Intro:thm:iso-coh-ring} provides the first
step towards an approach to Conjecture \ref{intro:dual_GW} in
general.

So far, progress on Conjecture \ref{intro:dual_GW} has been focused
on explicit classes of gerbes, all of which have trivial bands\footnote{We say that a $G$-gerbe $\Y\to \B$ has a {\em trivial band} if the
$Out(G)$-bundle $\overline{\Y}\to \B$ admits a section
(hence is trivialized by this section).}.
Conjecture \ref{intro:dual_GW} has been verified for trivial gerbes
\cite{ajt1}, for certain gerbes over $\mathbb{P}^1$-orbifolds
\cite{johnson}, and for root gerbes over complex projective manifolds \cite{ajt2}, \cite{ajt2.5}. The case of toric
gerbes is treated in \cite{ajt3}. All these works use very different
methods.

In this paper, we provide more supporting evidence to Conjecture
\ref{intro:dual_GW} by proving it for the $G$-gerbe $BH\to BQ$
obtained from a finite group extension. This is done in Sec. 
\ref{sec:cqft}. Our approach starts with an isomorphism, which we
deduce from Theorem \ref{Intro:thm:local-mackey}, between the
quantum cohomology ring of $BH$ and the $c$-twisted quantum
cohomology ring of $\widehat{BH}$. The isomorphism between quantum
cohomology rings yields an equality between $3$-point genus $0$
Gromov-Witten invariants of $\Y$ and $3$-point genus $0$ $c$-twisted
Gromov-Witten invariants of $\widehat{\Y}$. We then apply a
reconstruction-type argument, similar to the one used in
\cite[Sec.  4]{jk} to prove an equality between all Gromov-Witten
invariants. See (\ref{decomp_vanishing})-(\ref{decomp_formula}).
Conjecture \ref{intro:dual_GW} then follows easily. See Sec. 
\ref{sec:cqft} for details.

To the best of our knowledge, our result is the first verification of Conjecture
\ref{intro:dual_GW} for a class of gerbes with non-trivial bands. In general, we believe that the ring isomorphism in Theorem
\ref{Intro:thm:iso-coh-ring} should yield the changes of variables
needed in Conjecture \ref{intro:dual_GW}.
\subsection{Outlook}
A natural question arising from our work is whether the claim
($\star$) can be generalized to $G$-gerbes for groups $G$ which are
not necessarily finite. On one hand, Mackey's machine on Lie groups
is well studied in representation theory \cite{fe-do}. This suggests
that the duality theorems proved in this paper may admit
generalizations to more general $G$-gerbes. On the other hand,
Mackey's machine on Lie groups is more involved than the simple case
of finite groups presented in this paper. One can imagine that the
complete picture of the duality theory for general $G$-gerbes will
be more complicated, even for the trivial gerbe $BG$. We plan to
study this more general theory in the near future.

\subsection{Structure of the paper}
The rest of this paper is organized as follows. In Sec. 
\ref{sec:prelim}, we review some background material on groupoids,
their extensions, gerbes, twisted sheaves, and Hochschild
cohomology. In Sec.  \ref{sec:gps_extenstion_mackey_machine}, we
study the $G$-gerbe, $BH\to BQ$, arising from an extension of finite
groups. In Sec.  \ref{sec:hochschild}, we study cross-product
algebras of a general $G$-gerbe and Hochschild cohomology. Sec. 
\ref{sec:ring} is devoted to the study of the Chen-Ruan orbifold
cohomology of a $G$-gerbe. The Gromov-Witten theory of a gerbe
$BH\to BQ$ is treated in Sec.  \ref{sec:cqft}. In Sec. 
\ref{sec:sheaf}, we consider the category of sheaves on a $G$-gerbe.
\subsection*{Acknowledgment}
We thank N. Higson and M. Rieffel for discussions on Mackey's
machine. We thank Y. Ruan for his interests in this work and for his
encouragement.
X. T. is supported in part by NSF grant DMS-0703775 and 0900985.
H.-H. T. is supported in part by NSF grant DMS-0757722.

\section{Preliminary material}\label{sec:prelim}
In this section, we explain the basic concepts and tools that
will be used in this paper.

\subsection{Orbifolds and groupoids}\label{subsec:orb}
An orbifold is a separable Hausdorff topological space which is
locally modeled on the quotient of $\reals^n$ by a linear action of
a finite group. 
In this paper, we study orbifolds in the
viewpoint developed by Haefliger \cite{ha} and Moerdijk-Pronk
\cite{mo-pr}. Namely, we represent an orbifold as the quotient of a
proper \'etale groupoid.

A groupoid is a small category all of whose morphisms are invertible. A groupoid is usually denoted by $\calg\rightrightarrows \calg_0$,
where $\calg_0$ is the set of objects, and
$\calg$ is the set of arrows. 
Let $m$ be the groupoid multiplication map
$\calg\times_{\calg_0}\calg\rightarrow \calg$, $i$  the inverse on
$\calg$, $s$ and $t$  the source and target maps $\calg\rightarrow
\calg_0$, and $u$  the unit map $\calg_0\rightarrow \calg$. An
arrow $g\in\calg$ is sometimes denoted by $x\to y$, which means that
$s(g)=x$ and $t(g)=y$. A groupoid $\calg\rightrightarrows\calg_0$ is
called a Lie groupoid if $\calg$ and $\calg_0$ are smooth
manifolds,\footnote{$\calg$ may not be Hausdorff.} all the structure
maps
\[ \calg \times_{\calg_0}
\calg  \overset{m}{\rightarrow} \calg \overset{i}{\rightarrow}
   \calg \overset{s}{\underset{t}{\rightrightarrows}}
   \calg_0 \overset{u}{\rightarrow} \calg,
\]
are smooth, and the $s$ and $t$ maps are submersions. An \'etale
groupoid is a special type of Lie groupoid where the maps
$s$ and $t$ are local diffeomorphisms. A groupoid
$\calg$ is proper if the map $(s,t):\calg\rightarrow \calg_0\times
\calg_0$ is a proper map. A groupoid $\calg$ naturally defines an
equivalence relation on the set, $\calg_0$, of objects. Two points
$x$ and $y$ in $\calg_0$ are equivalent if there is an arrow $g:x\to
y$ in $\calg$ whose source is $x$ and target is $y$. Let $[\calg_0/\calg]$ denote the quotient space with respect to the
above-defined equivalence relation. Moerdijk and Pronk \cite{mo-pr}
proved that any orbifold $\X$ can be represented by the quotient
space of a proper \'etale Lie groupoid $\calg$. This leads to the
following definition. More detailed discussions can be found in
\cite{Adem_Leida_Ruan} and \cite{mo-pr}.

\begin{definition}
An orbifold groupoid is a proper \'etale groupoid $\calg$, and an
orbifold is the quotient space $[\calg_0/\calg]$ of an orbifold
groupoid.
\end{definition}

Locally, an orbifold $\X$ can be represented by the quotient space of the transformation
groupoid $\reals^n \rtimes \Gamma \rightrightarrows \reals^n$.
Gluing the local charts of an orbifold together yields a proper
\'etale groupoid, $\calg\rightrightarrows \calg_0$, representing
$\X$. Observe that the choice of charts on an orbifold
is not unique, and an orbifold can be represented by
different proper \'etale groupoids. However, a careful study shows that these different \'etale groupoids are all
Morita equivalent. The definition of Morita equivalence will be
recalled in the next subsection.  A crucial property of Morita
equivalent groupoids is that the corresponding quotient spaces are
isomorphic. In general, an orbifold can be uniquely represented by a
Morita equivalence class of  proper \'etale groupoids.

Given a groupoid $\calg$, we can consider the space of ``loops" in
$\calg$, which is defined to be $\calg^{(0)}:=\{g\in \calg:
s(g)=t(g)\}.$ There is a natural $\calg$-action on $\calg^{(0)}$. Let $p: \calg^{(0)}\rightarrow \calg_0$ be the map defined by taking the source
(=target) of an element $g\in \calg^{(0)}$. The $\calg$ action on
$\calg^{(0)}$ is a map $\rho:\calg\times_{t, \calg_0, p}\calg^{(0)}\rightarrow \calg^{(0)}$ such
that $\rho(h,g):=hgh^{-1}$. If $\calg$ is a proper \'etale groupoid,
one can easily check that the action groupoid $\calg\ltimes
\calg^{(0)}\rightrightarrows \calg^{(0)}$ is also proper \'etale.
This groupoid is called the inertia groupoid associated to the groupoid $\calg$.
If $\X$ is the orbifold represented by $\calg$, its inertia orbifold $I\X$ is represented by the inertia groupoid $\calg\ltimes \calg^{(0)}$. The inertia groupoid has a natural cyclic\footnote{See \cite[Def. 3.15]{crainic} for more details.} structure
$\theta:\calg^{(0)}\rightarrow \calg\ltimes \calg^{(0)}$ defined by
$\theta(g):=(g,g)$. This is very useful in the study of the cyclic
theory \cite{crainic} of the groupoid $\calg$.

As $\calg$ is an \'etale groupoid, every element $g$ in the loop space $\calg^{(0)}\subset \calg$ acts on the
tangent space $T_{p(g)}\calg_0$ by composing $(s_*)^{-1}:T_{s(g)}\calg_0\to T_g\calg$ with $t_*:T_g\calg \to T_{t(g)}\calg_0$. Assume that $\calg_0$ is
equipped with a $\calg$-invariant almost complex structure. This
makes $T_{p(g)}\calg_0$ a complex vector space. Since $g$ is of
finite order $r$, $T_{p(g)}\calg_0$ splits into a sum of eigenspaces of
the $g$ action, i.e.,
$T_{p(g)}\calg_0=\bigoplus_{k=0}^{r-1} V_k,$
where $g$ acts on $V_k$ with eigenvalue $\exp(\frac{2\pi
\sqrt{-1}k}{r})$. Define the $\age$ function on $\calg^{(0)}$ by
\begin{equation}\label{eq:age}
\age(g):=\sum_{k=0}^{r-1}\frac{k}{r}\dim(V_k)\in \rational.
\end{equation}
One easily checks that $\age$ is a locally constant function on
$\calg^{(0)}$ and invariant under the $\calg$ action. Therefore the
$\age$ function descends to a locally constant function on the inertia orbifold
$I\X$. It is an important part of the definition of the
Chen-Ruan orbifold cohomology \cite{cr}.
\subsection{Gerbes on orbifolds and groupoid extensions}\label{subsec:gerbe}
The notion of a gerbe was introduced by Giraud \cite{gi} in
algebraic geometry during his study of nonabelian cohomology. Let
$G$ be a topological group. In most of the cases of this paper, $G$
is a finite group equipped with the discrete topology. The notion of
a $G$-gerbe is a generalization of a principal $G$-bundle. Let $BG$
be the classifying orbifold of the group $G$. A $G$-gerbe over a
topological space $X$ is a $BG$ bundle over $X$. See
\cite{br} for related discussions. In this paper, we follow the
groupoid approach to stacks and gerbes developed in 
\cite{be-xu}, \cite{bry}, and \cite{la-st-xu}.

A $G$-gerbe $\X$ over an orbifold $\calb$ can
be represented by a {\em groupoid $G$-extension}, which is a diagram
\[
\begin{xy}
(15,0)*+{\calg_{\,\,}}="g1";(35,0)*+{\calh_{\,\,}}="h1";(55,0)*+{\calq_{\,\,}}="q1";%
(15,-10)*+{\calg_0}="g0";(35,-10)*+{\calh_0}="h0";(55,-10)*+{\calq_0.}="q0";%
{\ar "g1";"h1"}?*!/_2mm/{i};%
{\ar "h1";"q1"}?*!/_2mm/{j};%
{\ar@{=} "g0";"h0"};%
{\ar@{=} "h0";"q0"};%
{\ar@<0.ex> "g1";"g0"};%
{\ar@<-1.ex> "g1";"g0"};%
{\ar@<-1.ex> "h1";"h0"};%
{\ar@<0.ex> "h1";"h0"};%
{\ar@<-1.ex> "q1";"q0"};%
{\ar@<0.ex> "q1";"q0"};
\end{xy}
\]
In the above diagram,
\begin{enumerate}
\item
$\calg$ is a locally trivial bundle of groups over $\calg_0$ with fibers isomorphic to $G$. 
\item
$\calg_0$, $\calh_0$, and $\calq_0$ are identical smooth manifolds.
\item
$i$ and $j$ are smooth morphisms of Lie groupoids.
\item
$i$ is injective, $j$ is surjective, and the above sequence is exact.
\item
the groupoid $\calq\rightrightarrows \calq_0$ is a proper \'etale groupoid representing the orbifold
$\calb$.
\end{enumerate}
We follow \cite{la-st-xu} and use
$\calh\rightarrow \calq\rightrightarrows \calq_0$
to denote a groupoid extension of $\calq\rightrightarrows
\calq_0$.
\begin{definition}
A $G$-gerbe groupoid over $\calq\rightrightarrows
\calq_0$ is a groupoid $G$-extension of 
$\calq\rightrightarrows \calq_0$.
\end{definition}
Just like an orbifold has many different representations by proper
\'etale groupoids, the above groupoid extension representation of a
$G$-gerbe is, in general, not unique. A notion of Morita equivalence
between groupoid extensions, which we now recall, was introduced by
Laurent-Gengoux, Stienon, and Xu \cite{la-st-xu}.

Let $\calg\rightrightarrows \calg_0$ and $\calh\rightrightarrows
\calh_0$ be two groupoids. A Morita morphism from
$\calg\rightrightarrows \calg_0$ to $\calh\rightrightarrows \calh_0$
is a smooth morphism $(\phi_1,\phi_0)$ of Lie groupoids from
$\calg\rightrightarrows \calg_0$ to $\calh\rightrightarrows \calh_0$,
\[
\begin{xy}
(15,0)*+{\calg_{\,\,}}="g1"; (35,0)*+{\calh_{\,\,}}="h1";%
(15,-15)*+{\calg_0}="g0"; (35,-15)*+{\calh_0}="h0";%
{\ar@<0.ex> "g1";"g0"};%
{\ar@<-1.ex> "g1";"g0"};%
{\ar@<-1.ex> "h1";"h0"};%
{\ar@<0.ex> "h1";"h0"};%
{\ar "g1";"h1"}?*!/_2mm/{\phi_1};%
{\ar "g0";"h0"}?*!/_2mm/{\phi_0};%
\end{xy}
\]
such that $\phi_0$ is a surjective submersion and the pull-back of
the groupoid $\calh\rightrightarrows \calh_0$ along the map $\phi_0$
is isomorphic to $\calg\rightrightarrows \calg_0$. Two groupoids
$\calg\rightrightarrows \calg_0$ and $\calh\rightrightarrows
\calh_0$ are Morita equivalent if there is a third groupoid
$\calk\rightrightarrows \calk_0$ together with Morita morphisms from
$\calk\rightrightarrows \calk_0$ to both $\calg\rightrightarrows
\calg_0$ and $\calh\rightrightarrows \calh_0$. A Morita morphism
from the groupoid extension $Y_1\rightarrow X_1\rightrightarrows
M_1$ to $Y_2\rightarrow X_2\rightrightarrows M_2$ consists of Morita
morphisms $F_X:(X_1\rightrightarrows M_1)\rightarrow
(X_2\rightrightarrows M_2)$ and $F_Y:(Y_1\rightrightarrows
M_1)\rightarrow (Y_2\rightrightarrows M_2)$ such that the diagram
\[
\begin{xy}
(15,0)*+{Y_1}="y1"; (35,0)*+{X_1}="x1"; (55,0)*+{M_1}="m1",%
(15,-15)*+{Y_2}="y2"; (35,-15)*+{X_2}="x2"; (55,-15)*+{M_2}="m2";%
{\ar "y1";"x1"}; {\ar "y2";"x2"};%
{\ar^{F_Y} "y1";"y2"}; {\ar^{F_X} "x1";"x2"};%
{\ar^{F_X=F_Y} "m1";"m2"};
{\ar@<0.5ex> "x1";"m1"}; {\ar@<-0.5ex> "x1";"m1"};%
{\ar@<0.5ex> "x2";"m2"}; {\ar@<-0.5ex> "x2";"m2"};%
\end{xy}
\]
commutes. Two groupoid extensions $Y_i\rightarrow
X_i\rightrightarrows M_i$, $i=1,2$, are Morita equivalent if there
is a groupoid extension $Y\rightarrow X\rightrightarrows M$ 
with Morita morphisms from $Y\rightarrow X\rightrightarrows M$ to
both groupoid extensions.

An isomorphism class of $G$-gerbes over an orbifold $\B$ is
determined by a Morita equivalence class of $G$-groupoid extensions.
For example, if $Q$ is a finite group, a $G$-gerbe over an orbifold
$[pt/Q]$ is represented by a group extension
$1\rightarrow G\rightarrow H\rightarrow Q\rightarrow 1.$ 
One easily checks that two group extensions are Morita equivalent,
as defined above, if and only if they are isomorphic group
extensions (see, e.g., \cite{ro} for more discussions on group
extensions). So isomorphism classes of $G$-gerbes over $[pt/Q]$ are
in one-to-one correspondence with isomorphism classes of group
extensions of $Q$ by $G$.

\subsection{$\calg$-sheaves and modules of the groupoid algebra}\label{subsec:algebra}
\label{subsec:gpd-algebra} An approach to sheaf theory on orbifolds
via orbifold groupoids is explained in \cite{mo}. Let $\calg$ be a
proper \'etale groupoid representing an orbifold
$\X=[\calg_0/\calg]$. Denote by $\pi: \calg_0\to \X$ the projection
from $\calg_0$ to $\X$. A $\calg$-sheaf is a sheaf $\cals$ on
$\calg_0$ with a $\calg$ action, i.e., for any $g\in\calg$, $g:x\to
y$ induces a morphism $\hat{g}$ on stalks from $\cals_y$ to
$\cals_x$ satisfying $\widehat{gh}=\hat{g}\circ \hat{h}$. $\calg$-sheaves of abelian groups form an abelian
category $\operatorname{Sh}(\calg)$. A section $\xi$ of a
$\calg$-sheaf is called invariant if $\hat{g}(\xi_y) =\xi_x$ for any
$g:x\to y$. The functor
\[
\Gamma_{\text{inv}}: \Sh(\calg)\ni \cals\mapsto \{\text{invariant
sections of }\cals\}
\]
is a left exact functor from $\Sh(\calg)$ to the category
$\operatorname{Ab}$ of abelian groups. The right derived functor of
$\Gamma_{\text{inv}}$ defines the groupoid cohomology groups
$H^\bullet(\calg, \cals)$. Compactly supported sections of a $\calg$-sheaf,
\[
\Gamma_{\text{\rm cpt}}: \Sh(\calg)\ni \cals\mapsto \{ \xi \in
\Gamma_{\text{inv}}(\cals): \text{supp}(\xi)\ \text{is compact}\}
\]
also defines a left exact functor from $\Sh(\calg)$ to the category
$\operatorname{Ab}$. Its right derived functor defines compactly
supported cohomology groups $H^\bullet_{\text{\rm cpt}}(\calg, \cals)$.

If $\widetilde{\cals}$ is a sheaf on the orbifold $\X$, then the
pullback $\pi^*(\widetilde{\cals})$ along the projection map $\pi$
defines a $\calg$-sheaf over $\calg_0$. Furthermore, if $\cals$ is a
$\calg$-sheaf, $\pi_!(\cals)$ defines a sheaf on $\X$, where
$\pi_!(\cals)_x:=H^\bullet_{\text{\rm cpt}}(x/\pi, \pi_x^{-1}(\cals))$
and $x/\pi$ is the subgroupoid of $\calg$ over $x$. These two maps
define natural functors between the category of $\calg$-sheaves and
the category of sheaves on $\X$. Let $\cala$ be a $\calg$-sheaf of
unital algebras. For $a\in \Gamma_{\text{\rm cpt}}(\calg, s^*\cala)$ and $g\in \calg$, denote the value of $a$ at $g$ by $[a](g)$. Define the convolution algebra $\cala\rtimes \calg$
to be the vector space $\Gamma_{\text{\rm cpt}}(\calg, s^*\cala)$ with
the product
\[
[a_1\ast a_2](g)=\sum_{g_1g_2=g}[a_1](g_1)g_1([a_2](g_2)),
\]
for any $a_1, a_2\in \Gamma_{\text{\rm cpt}}(\calg, s^*\cala)$ and
$g_1,g_2, g\in \calg$.  For example, when $\cala$ is the sheaf
$\calc^\infty$ of smooth functions on $\calg_0$, we recover the
standard groupoid algebra, which is $C_c^\infty(\calg)$ with the
multiplication
\[
(a_1\ast a_2)(g)=\sum_{g_1g_2=g}a_1(g_1)a_2(g_2),
\]
for any $a_1,a_2\in C^\infty_c(\calg)$ and $g_1, g_2, g\in \calg$.
We will always identify the groupoid algebra $C_c^\infty(\calg)$
with the crossed product algebra $\calc^\infty\rtimes \calg$.

If $\cals$ is a $\calg$-sheaf of vector spaces, then
$\Gamma_{\text{\rm cpt}}(\cals)$ is a module over the groupoid algebra
$C_c^\infty(\calg)$ via
\[
a\xi(x)=\sum_{t(g)=x}a(g)g(\xi(x)), \qquad a\in C^\infty_c(\calg),\
\xi\in \Gamma_{\text{\rm cpt}}(\cals),\ x\in \calg_0.
\]
We obtain an additive functor $\Gamma_{\text{\rm cpt}}$ from the
category of $\calg$-sheaves to the category of modules of
$C_c^\infty(\calg)$.

Let $c$ be a $U(1)$-valued groupoid 2-cocycle on $\calg$, i.e.,
$c:\calg\times_{\calg_0}\calg\to U(1)$ such that
\[
c(g_1, g_2)c(g_1g_2, g_3)=c(g_1, g_2g_3)c(g_2,g_3).
\]
A $c$-twisted $\calg$-sheaf is a sheaf over $\calg_0$ together with
a $\calg$ action such that for $g_1, g_2\in \calg$, the actions
$\hat{g}_1:\cals_{x_1}\to \cals_{x_2}$ and $\hat{g}_2:\cals_{x_2}\to
\cals_{x_3}$ satisfy $\hat{g}_1\circ
\hat{g}_2=c(g_1,g_2)\widehat{g_1g_2}.$ $c$-twisted $\calg$-sheaves form an additive category. Let $\X$ be
the orbifold represented by $\calg$. Then the cocycle $c$ defines a
$U(1)$-gerbe, which we still denote by $c$, on the orbifold $\X$. We
can consider $c$-twisted sheaves on $\X$ as in \cite{cal} and
\cite{Lieblich}. One can easily check that the functors $\pi^*$ and
$\pi_!$ define natural functors between the category of $c$-twisted
$\calg$-sheaves and the category of $c$-twisted sheaves on $\X$.

Given a locally constant $U(1)$-valued 2-cocycle $c$ on $\calg$ and
a $\calg$-sheaf $\cala$ of unital algebras, we can define a
$c$-twisted crossed product algebra $\cala\rtimes_c \calg$ to be the
space $\Gamma_{\text{\rm cpt}}(\calg, s^*\cala)$ with the product
defined by
\begin{equation}\label{eq:product}
[a_1\ast_c a_2]_g=\sum_{g_1g_2=g}c(g_1,
g_2)[a_1]_{g_1}g_1([a_2]_{g_2}), \qquad a_1,a_2\in
\Gamma_{\text{\rm cpt}}(\calg, s^*\cala),\ g_1, g_2, g\in \calg.
\end{equation}
When we take $\cala$ to be the sheaf $\calc^\infty$ of smooth
functions, we have defined the $c$-twisted groupoid algebra
$\calc^\infty\rtimes_c \calg$.  Like the case without the twist, we
have natural additive functors between the category of $c$-twisted
$\calg$-sheaves and category of modules of the $c$-twisted groupoid
$\calc^\infty\rtimes_c\calg$.
\begin{remark}
For $\calc^\infty$, the same formula in Eq.  (\ref{eq:product})
defines an associative algebra $\calc^\infty\rtimes_c\calg$ for any
general $U(1)$-valued 2-cocycle on $\calg$ even without the
``locally constant" assumption. If we change the cocycle $c$ by a
coboundary, then a direct computation shows that the $c$-twisted
crossed product algebras $ \calc^\infty\rtimes_c\calg$ is changed by
an isomorphism constructed with the coboundary. This shows that the
isomorphism class of the $c$-twisted groupoid algebras only depends
on the cohomology class of ${c}$ in $H^2(\calg, U(1))=H^2(\X,
U(1))=H^3(\X, \integers)$.
\end{remark}
\begin{remark}
A $U(1)$-valued 2-cocycle on $\calg$ can be used to define an $S^1$-extension $\Gamma$ of the groupoid $\calg$,
\[
S^1\rightarrow \Gamma \rightarrow \calg\rightrightarrows \calg_0.
\]
The group $U(1)=S^1$ acts on the groupoid algebra of $\Gamma$ by
algebra automorphisms. Following \cite{tu-la-xu}, we consider the
$S^1$-invariant part $C_c^\infty(\Gamma)^{S^1}$, which is a subalgebra
of the groupoid algebra $C_c^\infty(\Gamma)$. By direct
computations, we can easily identify $C_c^\infty(\Gamma)^{S^1}$ with
the $c$-twisted groupoid algebra $\calc^\infty\rtimes_c \calg$.
\end{remark}

Assume that there is a $\calg$-invariant symplectic form on
$\calg_0$. Then via Fedosov's method, the first author in \cite{ta}
constructed a $\calg$-sheaf $\cala^{((\hbar))}$ of deformation
quantization of $\calg_0$. The crossed product
$\cala^{((\hbar))}\rtimes \calg$ is a deformation of the groupoid
algebra $\calc^\infty\rtimes \calg$. When $c$ is locally constant,
one can use Eqn. (\ref{eq:product}) to define a $c$-twisted
$\cala^{((\hbar))}\rtimes _c \calg$ deformed groupoid algebra. More
concretely, when $c$ is locally constant, one has a natural flat
connection on the associated $S^1$ extension $\Gamma\rightarrow
\calg\rightrightarrows \calg_0$. Such a flat connection is
sufficient to define a $U(1)$-invariant deformation
$\cala^{((\hbar))}(\Gamma)$ of the groupoid algebra
$C^{\infty}_c(\Gamma)$ by \cite[Sec.  3]{ta}. It is not hard to
check that the $S^1$-invariant component
$\cala^{((\hbar))}(\Gamma)^{S^1}$ defines a deformation of the
twisted groupoid algebra $C_c^\infty(\Gamma)^{S^1}\cong
\calc^\infty\rtimes_c \calg$, whose product is written out like Eq.
(\ref{eq:product}).

\subsection{Hochschild and cyclic cohomology}\label{subsec:hoch}

Let $A$ be a unital algebra over the field $\complex$ or
$\complex((\hbar))$. In this paper, following \cite{pptt}, we will
always work with bornological vector spaces, bornological tensor
products between bornological vector spaces, bornological algebras,
and modules of bornological algebras. As is explained in \cite{me}
and \cite[Appendix A]{pptt}, the category of modules of a
bornological algebra has enough projectives, and one can apply
standard homological tools to define Hochschild and cyclic
(co)homology of a bornological algebra. Let $A^{e}$ be the
enveloping algebra of $A$ defined\footnote{In this paper, without
causing extra confusion, we will use $\hat{\otimes}$ to denote the
bornological tensor product between bornological spaces.} by
$A^e:=A\hat{\otimes} A^{op}$. For an $A$-bimodule $M$, define
the Hochschild homology $HH_\bullet(A, M)$ and cohomology
$HH^\bullet(A, M)$ by
\[
HH_\bullet(A,M):=\operatorname{Tor}_\bullet ^{A^e}(A, M),\qquad
HH^\bullet(A,M):=\operatorname{Ext}^\bullet _{A^e}(A,M).
\]
We point out that the Yoneda product defines a natural product
structure on $HH^\bullet(A,A)$. Therefore, $HH^\bullet(A,A)$ is an
associative algebra.

If we consider the Bar-resolution of $A$ as an $A$-bimodule, then we
can write down explicit complexes $(C_\bullet(A, M), b)$ and
$(C^\bullet(A, M), d)$ computing $HH_\bullet(A,M)$ and
$HH^\bullet(A, M)$. Define $C_\bullet(A, M)=M\hat{\otimes}
A^{\hat{\otimes} \bullet}$ together with $b:C_\bullet(A,
M)\rightarrow C_{\bullet-1}(A,M)$ by
\[
b(m\otimes a_{1}\otimes\cdots\otimes a_k)=ma_1\otimes a_2\otimes
\cdots a_k-m\otimes a_1a_2\otimes \cdots
a_k+\cdots+(-1)^ka_km\otimes a_1\otimes \cdots a_{k-1}.
\]
Define $C^\bullet(A,M):=\operatorname{Hom}(A^{\hat{\otimes}\bullet},
M)$ with $d:C^\bullet(A,M)\rightarrow C^{\bullet+1}(A,M)$ by
\[
d\varphi(a_1,\cdots, a_{k+1}):=a_1\varphi(a_2, \cdots,
a_{k+1})-\varphi(a_1a_2, \cdots, a_{k+2})+\cdots
+(-1)^{k+1}\varphi(a_1,\cdots, a_k)a_{k+1}.
\]

Let $\mathbb{K}$ be a field. The cyclic homology of an algebra $A$
over $\mathbb{K}$ is computed by Connes' $b$-$B$ bicomplex. For
simplicity, we will consider the normalized complex
$\overline{C}_\bullet(A)=A\hat{\otimes}
(A/\mathbb{K})^{\hat{\otimes} \bullet}$ and
$\overline{C}^\bullet(A):=\operatorname{Hom}(\overline{C}_\bullet(A),
\mathbb{K})$ with the same differentials $b$ and $d$. Define
$\overline{B}:\overline{C}_{\bullet}(A)\to
\overline{C}_{\bullet+1}(A)$ by
\[
\overline{B}(a_0\otimes a_1\otimes \cdots \otimes
a_k):=\sum_{i}(-1)^{ik}1\otimes a_i\otimes \cdots \otimes a_k
\otimes a_0\otimes \cdots \otimes a_{i-1}.
\]
We can form the $(b,\overline{B})$-bicomplex
\[
\xymatrix{
  \ldots \ar[d]_{b} & \ldots\ar[d]_{b} &\ldots\ar[d]_{b}\\
  \overline{C}_2(A)\ar[d]_{b}  &
  \overline{C}_1(A)\ar[d]_{b} \ar[l]^{\overline{B}} &
  \overline{C}_0(A) \ar[l]^{\overline{B}} \\
  \overline{C}_1(A) \ar[d]_{b} & \overline{C}_0(A\ar[l]^{\overline{B}})\\
  \overline{C}_0(A).
  }
\]
The homology of the total complex is equal to the cyclic homology of
$A$. We can define the (periodic) cyclic cohomology of $A$ using a similar
bicomplex.

The (periodic) cyclic (co)homology of the groupoid algebra $C_c^\infty(\calg)$
for a proper \'etale groupoid $\calg$ was first computed by
Brylinski and Nistor \cite{br-ni} (see also Crainic \cite{crainic})
to be
\[
HP_\bullet(C_c^\infty(\calg))=H^\bullet_{\text{\rm cpt}}(I\X).
\]
The cohomology groups $H^\bullet_{\text{\rm cpt}}(-)$ (and $H^\bullet_{\text{\rm cpt}}(-,c)$) are compactly supported ($c$-twisted) cohomology groups of an orbifold defined by compactly supported differential forms. If $c$ is a $U(1)$-valued 2-cocycle on $\calg$, the cyclic
(co)homology of the $c$-twisted groupoid algebra was computed by Tu
and Xu \cite{tu-xu} to be
\[
HP_\bullet(\calc^\infty\rtimes_c \calg)=H^\bullet_{\text{\rm cpt}}(I\X,
c).
\]

When $\calg_0$ is equipped with a $\calg$-invariant symplectic form,
the Hochschild (co)homology and cyclic (co)homology of the algebra
$\cala^{((\hbar))}\rtimes \calg$ were computed in the works
\cite{nppt} and \cite{pptt} to be
\[
HH^\bullet(\cala^{((\hbar))}\rtimes \calg,\cala^{((\hbar))}\rtimes
\calg)=H^{\bullet-\ell}(I\X)((\hbar)),\qquad
HP_\bullet(\cala^{((\hbar))}\rtimes \calg)=H^\bullet_{
\text{\rm cpt}}(I\X)((\hbar)).
\]

\subsection{Morita equivalence}
Two algebras $A$ and $B$ are Morita equivalent if there is an
additive equivalence between the category of $A$ modules and the
category of $B$ modules. When both $A$ and $B$ are commutative, $A$ is
Morita equivalent to $B$ if and only if $A$ is isomorphic to $B$.
Such a concept becomes very interesting in the study of
noncommutative algebras. For example, for any positive integer $n$,
the algebra of $n\times n$ matrices over a field $\mathbb{K}$ is
Morita equivalent to the field $\mathbb{K}$.

The additive equivalence between the categories of modules in the
definition of Morita equivalence can be realized by $A$-$B$
bimodules. Two unital algebras $A$ and $B$ are Morita equivalent if there
is an $A$-$B$ bimodule $M$ and a $B$-$A$ bimodule $N$ such that as an
$A$-$A$ bimodule $M\otimes_B N$ is isomorphic to $A$ and as a
$B$-$B$ bimodule $N\otimes_A M$ is isomorphic to $B$.

The equivalence functor preserves exact sequences and projective
modules. Hence, Morita equivalent algebras have isomorphic
Hochschild and cyclic (co)homologies and also $K$-groups. See
\cite{loday}.

Morita theory is generalized to bornological algebras
\cite[Appendix]{pptt} with similar results extended. For
examples, Morita equivalent bornological algebras have isomorphic
Hochschild and cyclic (co)homologies.

\section{Group extension and Mackey machine}\label{sec:gps_extenstion_mackey_machine}
In this section, we study the local model of the duality theorems. Namely, we look at the case when there is
no space direction, i.e. a $G$-gerbe over an orbifold
$[pt/Q]=BQ $, where $G$ and $Q$ are both finite groups. As explained
in Sec.  \ref{subsec:gerbe}, such a gerbe can be presented by a
group extension. Different group extensions present the same gerbe over $[pt/Q]$ if and only if they are isomorphic. We refer the readers to \cite{la-st-xu} for more details.
\subsection{Group algebra}\label{subsec:groupalgebra}
The results in this subsection are standard in group theory. We
collect them here for later use. Let $G,H,Q$ be finite groups that
fit into the following exact sequence,
\begin{equation}\label{eq:extension}
1\longrightarrow G\stackrel{i}{\longrightarrow} H
\stackrel{j}{\longrightarrow} Q\longrightarrow 1.
\end{equation}
We would like to understand the structure of the group algebra $\complex
H$ by using the information of $G$,
$Q$, and the above exact sequence (\ref{eq:extension}). The study of
representations of the group $H$ in terms of representations of $G$
and $Q$ goes back to Frobenius, Schur, and Clifford \cite{cl}. In the case of
continuous groups, this approach is often called the {\em Mackey
machine}. We shall adapt this terminology in this paper.

We start by choosing a section $s:Q\rightarrow H$ of the group
homomorphism $j$ in the exact sequence (\ref{eq:extension}) such
that $j\circ s=id$, and $s(1)=1$. Since $G$ and $Q$ are finite
groups, such a section $s$ always exists. In general, there are many
possible choices of $s$, which later lead to equivalent structures.

It is important to point out that if the extension does not split,
i.e., $H$ is not isomorphic to a semi-direct product $G\rtimes Q$ as
a group, the section $s$ fails to be a group homomorphism. The
failure of $s$ being a group homomorphism is measured by
$$\tau:Q\times Q\rightarrow G, \quad \tau(q_1,
q_2):=s(q_1)s(q_2)s(q_1q_2)^{-1}.$$ Since
$j(\tau(q_1,q_2))=j(s(q_1))j(s(q_2))j(s(q_1q_2)^{-1})=q_1q_2(q_1q_2)^{-1}=1$,
we get $\tau(q_1,q_2)\in \ker(j)=G$.

We rewrite the defining property of $\tau$ as
\begin{equation}
\label{eq:tau-def} s(q_1)s(q_2)=\tau(q_1,q_2)s(q_1q_2).
\end{equation}
Applying the above equation to
$(s(q_1)s(q_2))s(q_3)=s(q_1)(s(q_2)s(q_3))$ yields
\begin{equation}\label{eq:tau-cocycle}
\tau(q_1,q_2)\tau(q_1q_2,q_3)=s(q_1)\tau(q_2,q_3)s(q_1)^{-1}\tau(q_1,q_2q_3).
\end{equation}
Note that Eq. (\ref{eq:tau-def}) and (\ref{eq:tau-cocycle})
define a $G$-gerbe $BH$ over $BQ$.

With the above choice of the section $s$, the group $H$ is
isomorphic to $G\times Q$ as a set via the map $$\alpha:
H\rightarrow G\times Q,\quad \alpha(h):=(hs(j(h))^{-1}, j(h)).$$ The
inverse of $\alpha$ is given by sending $(g,q)$ to $i(g)s(q)$. Via
the isomorphism $\alpha$, the group structure on $H$ defines a new
group structure $\cdot$ on $G\times Q$ given by
\begin{equation}\label{eq:twisted-prod-gp}
\begin{split}
&(g_1,q_1)\cdot (g_2,q_2)=\alpha\Big(\alpha^{-1}\big((g_1,q_1)\big)\alpha^{-1}\big((g_2,q_2)\big)\Big)=\alpha(i(g_1)s(q_1)i(g_2)s(q_2))\\
=&\alpha(i(g_1)s(q_1)i(g_2)s(q_1)^{-1}\tau(q_1,q_2)s(q_1q_2))=(g_1\Ad_{s(q_1)}(g_2)\tau(q_1,q_2),q_1q_2),
\end{split}
\end{equation}
where we use $\Ad_{h}(\cdot)$ to denote the conjugation action of an
element $h\in H$ on $G$ as a normal subgroup of $H$. Let
$G\rtimes_{s,\tau}Q$ denote the set $G\times Q$ with the above
group structure (\ref{eq:twisted-prod-gp}) induced by $H$.
Then $\alpha$ is a group isomorphism from $H$ to
$G\rtimes_{s,\tau}Q$. Note that different choices of the section $s$
define isomorphic group structures on $G\times Q$.

The group isomorphism $\alpha$ defines a natural isomorphism of
group algebras $$\alpha: \complex H \overset{\simeq}{\longrightarrow} \complex (G\rtimes_{s,\tau}Q).$$
With the data $s$ and $\tau$, the group algebra of
$G\rtimes_{s,\tau}Q$ can be written as a twisted crossed product
algebra $\complex G\rtimes _{s,\tau}Q$, where an element $q\in Q$
acts on $\complex G$ via conjugation by $s(q)$, and
the failure of the action to be a group homomorphism is governed by
$\tau$.

To have a better understanding of the group algebra $\complex H$, we
look at the group algebra $\complex G$. As a finite group, up to
isomorphisms $G$ has only finitely many irreducible unitary
representations. Let $\widehat{G}$ be the set\footnote{If $G$ is abelian, $\widehat{G}$ is the Pontryagin dual group of $G$. In this paper we only need $\widehat{G}$ to be a set.} of isomorphism classes
of irreducible unitary $G$ representations. Furthermore, for every
element $[\rho]$ in $\widehat{G}$, we {\em fix} a choice of an
irreducible representation in the class $[\rho]$ denoted by
$\rho:G\to \End(V_\rho),$ where $V_\rho$ is a certain finite
dimensional vector space. It is well-known (see, e.g.,
\cite[Proposition 3.29]{fu-ha}) that the group algebra $\complex G$
is isomorphic to a direct sum of matrix algebras $\oplus_{[\rho]\in
\widehat{G}}\End(V_\rho)$, via the natural map 
$$\beta: \complex G\overset{\simeq}{\longrightarrow} \oplus_{[\rho]\in \widehat{G}}\End(V_\rho), \quad g\mapsto (\rho(g))_{[\rho]\in \widehat{G}}.$$

Our goal in the rest of this subsection is to replace the group
algebra $\complex G$ in the twisted crossed product $\complex
G\rtimes _{s,\tau}Q$ by the matrix algebras $\oplus_{[\rho]\in
\widehat{G}}\End(V_\rho)$. We start with some preparations.

Let $\rho: G\to \End(V_\rho)$ be an irreducible $G$ representation.
For $q\in Q$, consider the $G$ representation $\tilde{\rho}$ defined
by $G\ni g\mapsto \rho(\Ad_{s(q)}(g)).$ It is easy to see that
$\tilde{\rho}$ is again an irreducible representation of $G$. If we
consider isomorphism classes of irreducible $G$ representations, the
above construction defines a right $Q$-action on $\widehat{G}$;
namely, $q\in Q$ sends the class $[\rho]\in \widehat{G}$ to the
class $[\tilde{\rho}]\in \widehat{G}$. This is a well-defined
$Q$-action because conjugations by elements in $G$ preserve
isomorphism classes of irreducible $G$-representations. For
convenience, we will write this right action as a left action. The
image of the isomorphism class $[\rho]\in \widehat{G}$ under the
action by $q$ will be denoted by $q([\rho])$. By abuse of notation,
the chosen irreducible $G$-representation that represents the class
$q([\rho])$ will also be denoted by $q([\rho]): G\to
\End(V_{q([\rho])})$.

Since the representation $q([\rho]): G\to \End(V_{q([\rho])})$ is,
by definition, equivalent to the representation $\tilde{\rho}: G\to
\End(V_\rho)$ defined by $g\mapsto \rho(\Ad_{s(q)}(g))$, there
exists an isomorphism of vector spaces,
$$T_q^{[\rho]}: V_{\rho}\to V_{q([\rho])},\ 
{\rm{such\ that}}\ \rho(\Ad_{s(q)}(g))={T^{[\rho]}_q}^{-1}\circ q([\rho])(g)\circ T^{[\rho]}_q.$$ To simplify our computation, we will always fix $T^{[\rho]}_1$ to be the identity map on $V_\rho$.

We compute $\rho(\Ad_{s(q_1)}(\Ad_{s(q_2)}(g)))$. Using Eq.  (\ref{eq:tau-cocycle}), we prove by an easy computation
\[
\begin{split}
&{T^{[\rho]}_{q_1}}^{-1}\circ {T^{q_1([\rho])}_{q_2}}^{-1}\circ q_1q_2([\rho])(g)\circ T^{q_1([\rho])}_{q_2}\circ T^{[\rho]}_{q_1}\\
=&{T^{[\rho]}_{q_1}}^{-1}\circ {T^{q_1([\rho])}_{q_2}}^{-1}\circ q_2(q_1([\rho]))(g)\circ T^{q_1([\rho])}_{q_2}\circ T^{[\rho]}_{q_1}\\
=&{T^{[\rho]}_{q_1}}^{-1}\circ q_1([\rho])(\Ad_{s(q_2)}(g))\circ T^{[\rho]}_{q_1}\\
=&\rho(\Ad_{s(q_1)}(\Ad_{s(q_2)}(g)))
=\rho(\Ad_{s(q_1)s(q_2)}(g))
=\rho(\Ad_{\tau(q_1,q_2)s(q_1q_2)}(g))\\
=&\rho(\tau(q_1,q_2))\circ \rho(\Ad_{s(q_1q_2)}(g))\circ
\rho(\tau(q_1,q_2))^{-1}\\
=&\rho(\tau(q_1,q_2))\circ {T^{[\rho]}_{q_1q_2}}^{-1}\circ
q_1q_2(\rho)(g)\circ T^{[\rho]}_{q_1q_2}\circ
\rho(\tau(q_1,q_2))^{-1}.
\end{split}
\]
It follows that $T^{q_1([\rho])}_{q_2}\circ T^{[\rho]}_{q_1}\circ
\rho(\tau(q_1,q_2))\circ {T^{[\rho]}_{q_1q_2}}^{-1}$, as a map in
$\End(V_{q_1q_2([ \rho])})$, commutes with the representation
$q_1q_2([\rho])$. By Schur's lemma, there must be a constant
$c^{[\rho]}(q_1, q_2)$ such that $T^{q_1([\rho])}_{q_2}\circ
T^{[\rho]}_{q_1}\circ \rho(\tau(q_1,q_2))\circ
{T^{[\rho]}_{q_1q_2}}^{-1}$ is $c^{[\rho]}(q_1, q_2)$ times the
identity map. In other words,
\begin{equation}
\label{eq:dfn-c}T^{q_1([\rho])}_{q_2}\circ
T^{[\rho]}_{q_1}=c^{[\rho]}(q_1,q_2)T^{[\rho]}_{q_1q_2}\rho(\tau(q_1,q_2))^{-1}.
\end{equation}
When we require the family $\{V_\rho\}$ to consist of unitary representations, the operator $T^{[\rho]}_q$ can also be chosen
to be unitary. Therefore, $c^{[\rho]}(q_1,q_2)$ actually takes value
in $U(1)$.
\begin{prop}
\label{prop:u(1)-cocycle}The function
$$c:\widehat{G}\times Q\times Q\to U(1),\quad ([\rho], q_1,q_2)\mapsto c^{[\rho]}(q_1,q_2)$$ is a 2-cocycle on the
groupoid $\widehat{G}\rtimes Q$ such that
$c^{[\rho]}(1,q)=c^{[\rho]}(q,1)=1$ for any $[\rho]\in \widehat{G},
q\in Q$. The cohomology class defined by $c$ is independent of the
choices of the section $s$ and the operator $T^{[\rho]}_q$.
\end{prop}
\begin{proof}
Consider the composition of maps
$T^{q_1q_2([\rho])}_{q_3}\circ T^{q_1([\rho])}_{q_2}\circ
T^{[\rho]}_{q_1}$. By associativity, we have
\begin{equation}\label{eq:T-ass}
(T^{q_1q_2([\rho])}_{q_3}\circ T^{q_1([\rho])}_{q_2})\circ
T^{[\rho]}_{q_1}=T^{q_1q_2([\rho])}_{q_3}\circ
(T^{q_1([\rho])}_{q_2}\circ T^{[\rho]}_{q_1}).
\end{equation}

Using Eq.  (\ref{eq:dfn-c}), we compute the left hand side of
(\ref{eq:T-ass}) to be
\[
\begin{split}
&c^{q_1([\rho])}(q_2,q_3)T^{q_1([\rho])}_{q_2q_3}q_1([\rho])(\tau(q_2,q_3))^{-1}\circ
T^{[\rho]}_{q_1}\\
=&c^{q_1([\rho])}(q_2,q_3)T^{q_1([\rho])}_{q_2q_3}\circ
T^{[\rho]}_{q_1}\circ
{T^{[\rho]}_{q_1}}^{-1}q_1([\rho])(\tau(q_2,q_3))^{-1}\circ
T^{[\rho]}_{q_1}\\
=&c^{q_1([\rho])}(q_2,q_3)c^{[\rho]}(q_1, q_2q_3)
T^{[\rho]}_{q_1q_2q_3} \rho(\tau(q_1,q_2q_3))^{-1}\circ
{T^{[\rho]}_{q_1}}^{-1}q_1([\rho])(\tau(q_2,q_3))^{-1}\circ
T^{[\rho]}_{q_1}\\
=& c^{q_1([\rho])}(q_2,q_3)c^{[\rho]}(q_1, q_2q_3)
T^{[\rho]}_{q_1q_2q_3}
\rho(\tau(q_1,q_2q_3))^{-1}\rho(\Ad_{s(q_1)}(\tau(q_2,q_3)^{-1}))\\
=& c^{q_1([\rho])}(q_2,q_3)c^{[\rho]}(q_1, q_2q_3)
T^{[\rho]}_{q_1q_2q_3}
\rho(\Ad_{s(q_1)}(\tau(q_2,q_3))\tau(q_1,q_2q_3))^{-1}.
\end{split}
\]

Similarly, using Eq.  (\ref{eq:dfn-c}), we compute the right hand side of Eq.  (\ref{eq:T-ass}),
\[
\begin{split}
&T^{q_1q_2([\rho])}_{q_3}\circ
T^{[\rho]}_{q_1q_2}c^{[\rho]}(q_1,q_2)\rho(\tau(q_1,q_2))^{-1}\\
=&T^{[\rho]}_{q_1q_2q_3}c^{[\rho]}(q_1q_2,q_3)c^{[\rho]}(q_1,q_2)\rho(\tau(q_1q_2,q_3))^{-1}\rho(\tau(q_1,q_2))^{-1}\\
=&T^{[\rho]}_{q_1q_2q_3}c^{[\rho]}(q_1q_2,q_3)c^{[\rho]}(q_1,q_2)\rho(\tau(q_1,q_2)\tau(q_1q_2,q_3))^{-1}.
\end{split}
\]

From the above computations and Eq.  (\ref{eq:tau-cocycle}), we
obtain the cocycle equation
\begin{equation}\label{eq:cocycle_condition_c}
c^{q_1([\rho])}(q_2,q_3)c^{[\rho]}(q_1,
q_2q_3)=c^{[\rho]}(q_1q_2,q_3)c^{[\rho]}(q_1,q_2).
\end{equation}

Note that by Eq.  (\ref{eq:tau-def}) and  $s(1)=1$, we have $\tau(q,1)=\tau(1,q)=1$. As $T^{[\rho]}_{1}=1$ for any $[\rho]\in
\widehat{G}$, we conclude from Eq.  (\ref{eq:dfn-c}) that
$c^{[\rho]}(1,q)=c^{[\rho]}(q,1)=1$.

One easily computes that changing the section $s$ to a new section
$s'$ amounts to changing the isomorphism $T_q^{[\rho]}$ to
$T^{[\rho]}_{q}\circ\rho(s'(q)s(q)^{-1})$. Schur's lemma implies that different choices of the isomorphism $T^{[\rho]}_q$ are
different by a scalar in $U(1)$. By straightforward computations and
the above observations, one can show that, up to a coboundary, the
cocycle $c$ is independent of the choices of the section $s$ and
$T^{[\rho]}_q$.
\end{proof}

Consider
the space $\oplus_{[\rho]\in \widehat{G}}\End(V_\rho)\otimes
\complex Q.$ This space is spanned by elements of the form $(x_\rho,
q)$, where $x_\rho$ is an element in $\End(V_\rho)$ with $[\rho]\in
\widehat{G}$ and $q\in Q$. Define the product $\circ$ of
$(x_{\rho_1},q_1)$ and $(\tilde{x}_{\rho_2},q_2)$ by
\[
(x_{\rho_1},q_1)\circ (\tilde{x}_{\rho_2},q_2):=
\left\{
\begin{array}{ll}
(x_{\rho_1}{T^{[\rho_1]}_{q_1}}^{-1}\tilde{x}_{q_1([\rho_1])}T^{[\rho_1]}_{q_1}\rho_1(\tau(q_1,q_2)),
q_1q_2),&
\text{if}\ [\rho_2]=q_1([\rho_1]),\\
0&\text{otherwise}.
\end{array}
\right.
\]
We check the associativity of the product $\circ$. It is sufficient
to check the identity
\[
\begin{split}
&[(x_\rho,q)\circ (y_{q([\rho])}, q')]\circ (z_{qq'([\rho])}, q'')
=(x_\rho
{T^{[\rho]}_q}^{-1}y_{q([\rho])}T^{[\rho]}_q\rho(\tau(q,q')),qq')\circ (z_{qq'([\rho])}, q'')\\
=&(x_\rho{T^{[\rho]}_q}^{-1}y_{q([\rho])}T^{[\rho]}_q\rho(\tau(q,q'))
{T^{[\rho]}_{qq'}}^{-1}z_{qq'([\rho])}T^{[\rho]}_{qq'}\rho(\tau(qq',q'')),qq'q'')\\
=&(x_\rho{T^{[\rho]}_q}^{-1}y_{q([\rho])}{T^{q'([\rho])}_{q'}}^{-1}
c^{[\rho]}(q,q')z_{qq'([\rho])}T^{[\rho]}_{qq'}\rho(\tau(qq',q'')),qq'q'')\\
=&(x_\rho{T^{[\rho]}_q}^{-1}y_{q([\rho])}{T^{q'([\rho])}_{q'}}^{-1}z_{qq'([\rho])}
T^{q([\rho])}_{q'}T^{[\rho]}_{q}\rho(\tau(q,q')\tau(qq',q'')),
qq'q'')\\
=&(x_\rho{T^{[\rho]}_q}^{-1}y_{q([\rho])}{T^{q'([\rho])}_{q'}}^{-1}z_{qq'([\rho])}
T^{q([\rho])}_{q'}q([\rho])(\tau(q',q''))T^{[\rho]}_q\rho(\tau(q,q'q'')),
qq'q'')\\
=&(x_\rho,q)\circ[(y_{q([\rho])},q')\circ (z_{qq'([\rho])},q'')],
\end{split}
\]
where in the second and third equalities, we used Eq. 
(\ref{eq:dfn-c}) and the fact that $c$ takes value in $U(1)$, and in
the fourth equality, we used the cocycle condition
(\ref{eq:tau-cocycle}) for $\tau$.

The space $\oplus_{[\rho]\in \widehat{G}}\End(V_\rho)\otimes
\complex Q$ with the product structure introduced above will be
denoted by $\oplus_{[\rho]}\End(V_{\rho})\rtimes_{T,\tau}Q$, and
will be called the twisted crossed product algebra.

\begin{prop}\label{prop:matrix-coeff-algebra}
The map $\chi: G\times Q\ni (g,q)\mapsto \sum_{[\rho]\in
\widehat{G}}(\rho(g),q)$ defines an algebra isomorphism from the
group algebra $\complex G\rtimes_{s,\tau}Q$ to the twisted crossed
product algebra $\oplus_{[\rho]}\End(V_{\rho})\rtimes_{T,\tau}Q$.
Hence, $\chi\circ\alpha:\complex H\to
\oplus_{[\rho]}\End(V_{\rho})\rtimes_{T,\tau}Q$ is an algebra
isomorphism.
\end{prop}
\begin{proof}
We first prove that $\chi$ is an isomorphism of vector spaces. Since
both algebras are finite dimensional with the same dimension, it
suffices to prove that $\chi$ is injective. We denote a general element in $\complex G\rtimes_{s,\tau}Q$ by $\sum_q\sum_g c_{g,q}(g,q)$, where $(g,q)$ is a group element in $G\rtimes_{s,\tau}Q$, and $c_{g,q}\in \mathbb{C}$. If
$\sum_q\sum_{g}c_{g,q}(g,q)\in \complex G\rtimes_{s,\tau}Q$ is in
the kernel of $\chi$, then as the map $\complex G\to
\oplus_{[\rho]\in \widehat{G}}\End(V_\rho)$ is an isomorphism, we
see that for any fixed $q$, $\sum_{g}c_{g,q}(g,q)$ must be zero.
Hence, $c_{g,q}=0$ for any $g,q$. Hence,
$\sum_q\sum_{g}c_{g,q}(g,q)=0$. This shows that the kernel of $\chi$
is trivial. Therefore $\chi$ is an isomorphism of vector spaces.

By the definition of $\chi$, we can check the identity $\chi((g_1,q_1))\circ \chi((g_2,q_2))=\chi((g_1,q_2)\cdot (g_2,q_2))$:
\[
\begin{split}
&\chi((g_1,q_1))\circ \chi((g_2,q_2))=\sum_{[\rho]}(\rho(g_1),q_1)\circ \sum_{[\rho']}(\rho'(g_2),q_2)
=\sum_{[\rho]}(\rho(g_1),q_1)\circ (q_1([\rho])(g_2),q_2)\\
=&\sum_{[\rho]}(\rho(g_1){T^{[\rho_1]}_{q_1}}^{-1}q_1([\rho])(g_2)T^{[\rho_1]}_{q_1}\rho(\tau(q_1,q_2)),q_1q_2)\\
=&\sum_{[\rho]}(\rho(g_1)\rho(\Ad_{s(q_1)}(g_2))\rho(\tau(q_1,q_2)),q_1q_2)\\
=&\sum_{[\rho]}(\rho(g_1s(q_1)g_2s(q_1)^{-1}\tau(q_1,q_2)),q_1q_2)
=\chi((g_1\Ad_{s(q_1)}(g_2)\tau(q_1,q_2), q_1q_2))\\
=&\chi((g_1,q_1)\cdot (g_2,q_2)).
\end{split}
\]
This proves that $\chi$ is an algebra homomorphism.
\end{proof}
\subsection{Mackey machine and Morita equivalence}\label{sec:mackey}
We introduce our framework to
explain the Mackey machine on representations of groups. Our
treatment is essentially a reformulation of the Mackey machine using Morita equivalence of algebras. Such a reformulation
seems to be known among experts \cite{fe-do}. However, we are unable
to locate a good reference that provides the exact statements we
need.

For an extension of finite groups $1\rightarrow G\rightarrow
H\rightarrow Q\rightarrow 1$, the Mackey machine provides a way to
describe representations of $H$ via representations of $G$ and $Q$.
In our language, the
Mackey machine may be formulated as saying that the representation
theory of the group $H$ is equivalent to the representation theory
of the transformation groupoid $\widehat{G}\rtimes Q$ together with
the $U(1)$-valued groupoid cocycle $c$, which is introduced in
Eq.  (\ref{eq:dfn-c}).

The category of representations of the group $H$ is known to be equivalent to the category of modules of the group algebra $\complex
H$. And the category of representations of the groupoid $\widehat{G}\rtimes Q$ with the $U(1)$-valued groupoid cocycle $c$, 
is equivalent to the category of modules of $C(\widehat{G}\rtimes Q,c)$, the twisted groupoid algebra
associated to the cocycle $c$ on $\widehat{G}\rtimes Q$, as
introduced in \cite{tu-la-xu} (see Sec.  \ref{subsec:algebra} for
a review). We spell out the definition of $C(\widehat{G}\rtimes Q,
c)$. As a space\, $C(\widehat{G}\rtimes Q, c)$ consists of
$C(\widehat{G})$-valued functions on $Q$, i.e.,
$C(\widehat{G}\rtimes Q, c)$ consists of functions on
$\widehat{G}\times Q$. For $([\rho],q)\in \widehat{G}\times Q$ we
abuse notation and denote by $([\rho], q)$ the function on
$\widehat{G}\times Q$ which takes value $1$ at the point
$([\rho],q)$ and $0$ elsewhere. The collection $\{([\rho], q)\}$ of
functions on $\widehat{G}\times Q$ form a basis of
$C(\widehat{G}\rtimes Q, c)$. The product on $C(\widehat{G}\rtimes
Q, c)$ is defined by
\[
([\rho],q)\circ([\rho'],q')=\left\{\begin{array}{ll}c^{[\rho]}(q,q')([\rho],
qq')&\ \text{if\ }[\rho']=q([\rho])\\ 0&\
\text{otherwise}\end{array}\right. .
\]
The associativity of the above product $\circ$ follows from the
cocycle condition (\ref{eq:cocycle_condition_c}) of $c$.

Our formulation of the Mackey machine is the following theorem.
\begin{theorem}\label{thm:local-mackey}The group algebra $\complex H$ is
Morita equivalent to the twisted groupoid algebra
$C(\widehat{G}\rtimes Q, c)$.
\end{theorem}
\begin{proof}
Set
$A=\complex H$ and
$B=C(\widehat{G}\rtimes Q, c)$. By Proposition \ref{prop:matrix-coeff-algebra}, 
$A=\complex H$ is isomorphic to the algebra $\oplus_{[\rho]\in
\widehat{G}}\End(V_\rho)\rtimes_{T,\tau} Q$.  To prove the theorem, we will
construct an $A$-$B$ bimodule $M$ and a $B$-$A$ bimodule $N$,
such that $M\otimes_B N\cong A$ as $A$-$A$ bimodules, and $N\otimes_A M\cong B$ as $B$-$B$ bimodules.

The $A$-$B$ bimodule $M$, as a vector space, is defined to be the space
$\oplus_{[\rho]\in \widehat{G}}V_{\rho}\times Q$, which is spanned by
elements of the form $(\xi_\rho, q)$ with $\xi_\rho\in V_\rho$ and
$q\in Q$. The left $A$-module structure on $M$ is defined by
\[
(x_{\rho_0},q_0)(\xi_\rho, q):=\left\{\begin{array}{ll}
(c^{[\rho_0]}(q_0,
q)x_{\rho_0}{T^{[\rho_0]}_{q_0}}^{-1}(\xi_{q_0([\rho_0])}),
q_0q)&\quad
\text{if }[\rho]=q_0([\rho_0])\\
0&\quad\text{otherwise.}\end{array}\right.
\]

The right $B$-module structure on $M$ is defined by
\[
(\xi_\rho, q)([\rho_1],
q_1):=\left\{\begin{array}{ll}(c^{[\rho]}(q,q_1)\xi_\rho,
qq_1)&\quad \text{if}\ [\rho_1]=q([\rho])\\
0&\quad \text{otherwise.}\end{array}\right.
\]
Using Eq.  (\ref{eq:dfn-c}) and (\ref{eq:cocycle_condition_c}) about $c$ ,  we can easily check that $M$ is an $A$-$B$ bimodule. We leave the details to the reader. 

The $B$-$A$ bimodule $N$ is defined to be $\oplus_{[\rho]\in
\widehat{G}}V_\rho^*\times Q$ as a vector space. So $N$ is spanned by elements of the
form $(\eta_\rho, q)$, where $q\in Q$ and $\eta_\rho\in V_\rho^*$ is a linear functional on
$V_\rho$. The right $A$-module structure on $N$ is defined by
\[
(\eta_{\rho}, q)(x_{\rho_0},
q_0)=\left\{\begin{array}{ll}(c^{[\rho_0]}(q_0,q_0^{-1}q)\eta_{\rho_0}\circ
x_{\rho_0}\circ{T^{[\rho_0]}_{q_0}}^{-1}, q_0^{-1}q)&\text{if\ }[\rho]=[\rho_0]\\
0&\text{otherwise.}\end{array}\right.
\]

The left $B$-module structure on $N=\oplus_{[\rho]\in
\widehat{G}}V_\rho^*\times Q$ is defined by
\[
([\rho_0],q_0)(\eta_{\rho},
q):=\left\{\begin{array}{ll}(c^{[\rho]}(qq_0^{-1},
q_0)\eta_\rho, qq_0^{-1})&\text{if}\ [\rho_0]=qq_0^{-1}([\rho])\\
0&\text{otherwise}.\end{array}\right.
\]

It is straightforward to check that $N$ is a $B$-$A$ bimodule using Eq.  (\ref{eq:dfn-c}) and (\ref{eq:cocycle_condition_c}) about $c$ .  We leave the details to the reader. 

Next we show that $M\otimes_B N\cong A$ as $A$-$A$ bimodules. Define a map $\Xi: M\otimes_\complex N\to A$ by
\[
\Xi\big((\xi_\rho,
q),(\eta_{\rho'},q')\big):=\left\{\begin{array}{ll}(\xi_\rho\otimes
\eta_{qq'^{-1}([\rho])}\circ T^{[\rho]}_{qq'^{-1}}
\frac{c^{qq'^{-1}([\rho])}(q'q^{-1},q)}{c^{qq'^{-1}([\rho])}(q'q^{-1},
qq'^{-1})},
qq'^{-1})&\quad\text{if}\  qq'^{-1}([\rho])=[\rho']\\
0&\quad\text{otherwise}.\end{array}\right.
\]
We check that
\[
\begin{split}
&\Xi\big((\xi_\rho, q)(q([\rho]), q_0),
(\eta_{qq_0q'^{-1}([\rho])},q')
\big)\\
=&\Xi\big((c^{[\rho]}(q,q_0)\xi_\rho,qq_0),
(\eta_{qq_0q'^{-1}([\rho])},q')\big)\\
=&(\frac{c^{[\rho]}(q,q_0)c^{qq_0q'^{-1}([\rho])}(q'q_0^{-1}q^{-1},
qq_0)}{c^{qq_0q'^{-1}([\rho])}(q'q_0^{-1}q^{-1},qq_0q'^{-1})}\xi_\rho\otimes
\eta_{qq_0q'^{-1}([\rho])}\circ T^{[\rho]}_{qq_0q'^{-1}},
qq_0q'^{-1})\\
=&(\frac{c^{qq_0q'^{-1}([\rho])}(q'q_0^{-1}q^{-1},
q)c^{qq_0q'^{-1}([\rho])}(q'q_0^{-1},
q_0)}{c^{qq_0q'^{-1}([\rho])}(q'q_0^{-1}q^{-1},qq_0q'^{-1})}\xi_\rho\otimes
\eta_{qq_0q'^{-1}([\rho])}\circ T^{[\rho]}_{qq_0q'^{-1}},
qq_0q'^{-1})\\
=&\Xi\big(\xi_\rho, q), (q([\rho]),q_0)(\eta_{qq_0q'^{-1}([\rho])},
q')\big),
\end{split}
\]
where in the third equality, we used the cocycle condition
(\ref{eq:cocycle_condition_c}) of $c$. Therefore, $\Xi$ passes to
define a map, which we still denote by $\Xi$, from $M\otimes_B N$ to
$A$.

We check that the map $\Xi$ is compatible with the left $A$-module structure:
\[
\begin{split}
&(x_\rho, q)\Xi\big((\xi_{q([\rho])}, q_0),
(\eta_{qq_0q_1^{-1}([\rho])},q_1)\big)\\
=&(x_\rho, q)(\xi_{q([\rho])}\otimes \eta_{qq_0q_1^{-1}([\rho])}\circ
T^{q([\rho])}_{q_0q_1^{-1}} \frac{c^{qq_0q_1^{-1}([\rho])}(q_1q_0^{-1},q_0)}{c^{qq_0q_1^{-1}([\rho])}(q_1q_0^{-1},q_0q_1^{-1})}, q_0q_1^{-1})\\
=&(x_\rho\circ {T^{[\rho]}_{q}}^{-1}\circ \xi_{q([\rho])}\otimes
\eta_{qq_0q_1^{-1}([\rho])}\circ T^{q([\rho])}_{q_0q_1^{-1}}
\frac{c^{qq_0q_1^{-1}([\rho])}(q_1q_0^{-1},q_0)}{c^{qq_0q_1^{-1}([\rho])}(q_1q_0^{-1},q_0q_1^{-1})}T^{[\rho]}_{q}\rho(\tau(q,q_0q_1^{-1})),
qq_0q_1^{-1})\\
=&(x_\rho\circ {T^{[\rho]}_{q}}^{-1}\circ \xi_{q([\rho])}\otimes
\eta_{qq_0q_1^{-1}([\rho])}\frac{c^{qq_0q_1^{-1}([\rho])}(q_1q_0^{-1},q_0)}{c^{qq_0q_1^{-1}([\rho])}(q_1q_0^{-1},q_0q_1^{-1})}\circ
T^{q([\rho])}_{q_0q_1^{-1}} T^{[\rho]}_{q}\rho(\tau(q,q_0q_1^{-1})),
qq_0q_1^{-1})\\
=&(x_\rho\circ {T^{[\rho]}_{q}}^{-1}\circ \xi_{q([\rho])}\otimes
\eta_{qq_0q_1^{-1}([\rho])}\circ
\frac{c^{qq_0q_1^{-1}([\rho])}(q_1q_0^{-1},q_0)c^{[\rho]}(q,q_0q_1^{-1})}{c^{qq_0q_1^{-1}([\rho])}(q_1q_0^{-1},q_0q_1^{-1})}T^{[\rho]}_{qq_0q_1^{-1}},
qq_0q_1^{-1}).
\end{split}
\]
On the other hand,
\[
\begin{split}
&\Xi\big((x_\rho, q)(\xi_{q([\rho])},q_0),
(\eta_{qq_0q_1^{-1}([\rho])},q_1)\big)\\
=&\Xi\big((x_\rho\circ {T^{[\rho]}_q}^{-1}\circ
\xi_{q([\rho])}c^{[\rho]}(q,q_0),qq_0),
(\eta_{qq_0q_1^{-1}([\rho])},q_1)\big)\\
=&(x_\rho\circ {T^{[\rho]}_q}^{-1}(\xi_{q([\rho])})\otimes
\eta_{qq_0q_1^{-1}([\rho])}\circ\frac{c^{[\rho]}(q,q_0)c^{qq_0q_1^{-1}([\rho])}(q_1q_0^{-1}q^{-1},qq_0)}{c^{qq_0q_1^{-1}([\rho])}(q_1q_0^{-1}q^{-1},qq_0q_1^{-1})}
T^{[\rho]}_{qq_0q_1^{-1}}, qq_0q_1^{-1}).
\end{split}
\]
Using the cocycle condition (\ref{eq:cocycle_condition_c}) of $c$, it is not difficult to check that
\[
\frac{c^{qq_0q_1^{-1}([\rho])}(q_1q_0^{-1},q_0)c^{[\rho]}(q,q_0q_1^{-1})}{c^{qq_0q_1^{-1}([\rho])}(q_1q_0^{-1},q_0q_1^{-1})}=
\frac{c^{[\rho]}(q,q_0)c^{qq_0q_1^{-1}([\rho])}(q_1q_0^{-1}q^{-1},qq_0)}{c^{qq_0q_1^{-1}([\rho])}(q_1q_0^{-1}q^{-1},qq_0q_1^{-1})}.
\]
It follows that $\Xi$ is compatible with the left $A$-module
structure. Namely, we have the following equality:
\[
(x_\rho, q)\Xi\big((\xi_{q([\rho])}, q_0),
(\eta_{qq_0q_1^{-1}([\rho])},q_1)\big)=\Xi\big((x_\rho,
q)(\xi_{q([\rho])},q_0), (\eta_{qq_0q_1^{-1}([\rho])},q_1)\big).
\]

We check that $\Xi$ is compatible with the right $A$-module structure:
\[
\begin{split}
&\Xi\big((\xi_\rho,q_0), (\eta_{q_0q_1^{-1}([\rho])},
q_1)\big)(x_{q_0q_1^{-1}([\rho])},q)\\
=&\big(\xi_\rho\otimes \eta_{q_0q_1^{-1}([\rho])}\circ
T^{[\rho]}_{q_0q_1^{-1}}
\frac{c^{q_0q_1^{-1}([\rho])}(q_1q_0^{-1},q_0)}{c^{q_0q_1^{-1}([\rho])}(q_1q_0^{-1},
q_0q_1^{-1})}, q_0q_1^{-1}\big)(x_{q_0q_1^{-1}([\rho])}, q)\\
=&(\frac{c^{q_0q_1^{-1}([\rho])}(q_1q_0^{-1},q_0)}{c^{q_0q_1^{-1}([\rho])}(q_1q_0^{-1},
q_0q_1^{-1})}\xi_\rho\otimes \eta_{q_0q_1^{-1}([\rho])}\circ
T^{[\rho]}_{q_0q_1^{-1}}{T^{[\rho]}_{q_0q_1^{-1}}}^{-1}
x_{q_0q_1^{-1}([\rho])}T^{[\rho]}_{q_0q_1^{-1}}\rho(\tau(q_0q_1^{-1},q)),
q_0q_1^{-1}q)\\
=&(\frac{c^{q_0q_1^{-1}([\rho])}(q_1q_0^{-1},q_0)}{c^{q_0q_1^{-1}([\rho])}(q_1q_0^{-1},
q_0q_1^{-1})}\xi_\rho\otimes \eta_{q_0q_1^{-1}([\rho])}\circ
x_{q_0q_1^{-1}([\rho])}T^{[\rho]}_{q_0q_1^{-1}}\rho(\tau(q_0q_1^{-1},q)),
q_0q_1^{-1}q)\\
=&(\frac{c^{[\rho]}(q_0q_1^{-1},q)c^{q_0q_1^{-1}([\rho])}(q_1q_0^{-1},q_0)}{c^{q_0q_1^{-1}([\rho])}(q_1q_0^{-1},
q_0q_1^{-1})}\xi_\rho\otimes \eta_{q_0q_1^{-1}([\rho])}\circ
x_{q_0q_1^{-1}([\rho])}{T^{q_0q_1^{-1}([\rho])}_{q}}^{-1}T^{[\rho]}_{q_0q_1^{-1}q},
q_0q_1^{-1}q).
\end{split}
\]
On the other hand,
\[
\begin{split}
&\Xi\big((\xi_\rho,q_0), (\eta_{q_0q_1^{-1}([\rho])},
q_1)(x_{q_0q_1^{-1}([\rho])},q)\big)\\
=&\Xi\big((\xi_\rho,q_0),
(c^{q_0q_1^{-1}([\rho])}(q,q^{-1}q_1)\eta_{q_0q_1^{-1}([\rho])}\circ
x_{q_0q_1^{-1}([\rho])}\circ {T^{q_0q_1^{-1}([\rho])}_{q}}^{-1},
q^{-1}q_1)\big)\\
=&(\xi_{\rho}\otimes \eta_{q_0q_1^{-1}([\rho])}\circ
x_{q_0q_1^{-1}([\rho])} {T^{q_0q_1^{-1}([\rho])}_{q}}^{-1}
T^{[\rho]}_{q_0q_1^{-1}q}\frac{c^{q_0q_1^{-1}([\rho])}(q,q^{-1}q_1)c^{q_0q_1^{-1}q([\rho])}(q^{-1}q_1q_0^{-1},
q_0)}{c^{q_0q_1^{-1}q([\rho])}(q^{-1}q_1q_0^{-1}, q_0q_1^{-1}q)},
q_0q_1^{-1}q).
\end{split}
\]
We need to show that
\[
\frac{c^{[\rho]}(q_0q_1^{-1},q)c^{q_0q_1^{-1}([\rho])}(q_1q_0^{-1},q_0)}{c^{q_0q_1^{-1}([\rho])}(q_1q_0^{-1},
q_0q_1^{-1})}=\frac{c^{q_0q_1^{-1}([\rho])}(q,q^{-1}q_1)c^{q_0q_1^{-1}q([\rho])}(q^{-1}q_1q_0^{-1},
q_0)}{c^{q_0q_1^{-1}q([\rho])}(q^{-1}q_1q_0^{-1}, q_0q_1^{-1}q)},
\]
which is equivalent to
\[
\begin{split}
&c^{[\rho]}(q_0q_1^{-1},q)c^{q_0q_1^{-1}([\rho])}(q_1q_0^{-1},q_0)c^{q_0q_1^{-1}q([\rho])}(q^{-1}q_1q_0^{-1},
q_0q_1^{-1}q)\\
=&c^{q_0q_1^{-1}([\rho])}(q_1q_0^{-1},
q_0q_1^{-1})c^{q_0q_1^{-1}([\rho])}(q,q^{-1}q_1)c^{q_0q_1^{-1}q([\rho])}(q^{-1}q_1q_0^{-1},
q_0).
\end{split}
\]

Using the cocycle condition (\ref{eq:cocycle_condition_c}) of $c$, we have
\[
\begin{split}
&c^{q_0q_1^{-1}q([\rho])}(q^{-1}q_1q_0^{-1},
q_0q_1^{-1}q)c^{[\rho]}(q_0q_1^{-1},q)c^{q_0q_1^{-1}([\rho])}(q_1q_0^{-1},q_0)\\
=&c^{q_0q_1^{-1}q([\rho])}(q^{-1}q_1q_0^{-1},
q_0q_1^{-1})c^{q_0q_1^{-1}q([\rho])}(q^{-1},q)c^{q_0q_1^{-1}([\rho])}(q_1q_0^{-1},q_0),
\end{split}
\]
and
\[
\begin{split}
&c^{q_0q_1^{-1}([\rho])}(q_1q_0^{-1},
q_0q_1^{-1})c^{q_0q_1^{-1}([\rho])}(q,q^{-1}q_1)c^{q_0q_1^{-1}q([\rho])}(q^{-1}q_1q_0^{-1},
q_0)\\
=&c^{q_0q_1^{-1}([\rho])}(q_1q_0^{-1},
q_0q_1^{-1})c^{q_0q_1^{-1}([\rho])}(q,
q^{-1}q_1q_0^{-1})c^{q_0q_1^{-1}([\rho])}(q_1q_0^{-1},q_0)\\
=&c^{q_0q_1^{-1}([\rho])}(q,
q^{-1})c^{q_0q_1^{-1}q([\rho])}(q^{-1}q_1q_0^{-1},
q_0q_1^{-1})c^{q_0q_1^{-1}([\rho])}(q_1q_0^{-1},q_0).
\end{split}
\]
To prove the above two expressions are equal, we are left to prove
\[
c^{q_0q_1^{-1}q([\rho])}(q^{-1},q)=c^{q_0q_1^{-1}([\rho])}(q,
q^{-1}).
\]
Using the fact that
$c^{q_0q_1^{-1}([\rho])}(1,q)=1=c^{q_0q_1^{-1}([\rho])}(q,1)$, the above equality is equivalent to
\[
c^{q_0q_1^{-1}([\rho])}(q,1)c^{q_0q_1^{-1}q([\rho])}(q^{-1},q)=c^{q_0q_1^{-1}([\rho])}(q,
q^{-1})c^{q_0q_1^{-1}([\rho])}(1,q),
\]
which again follows by the cocycle condition
(\ref{eq:cocycle_condition_c}) of $c$.

In summary, we have shown that $\Xi$ is an $A$-$A$ bimodule map from
$M\otimes_B N$ to $A$. From the definition of $\Xi$, we can see that
the image of $\Xi$ contains all elements of the form $(x_\rho, q)$,
where $x_\rho$ is a rank 1 operator on $V_\rho$ for any $[\rho]\in
\widehat{G},\ q\in Q$. As these elements span the whole space of
$A$, we conclude that $\Xi$ is surjective. By counting dimensions
of $M\otimes_B N$ and $A$, we conclude that $\Xi$ must be an
isomorphism.

Next, we define an isomorphism, $\Theta:N\otimes_A M \to B$, of
$B$-$B$ bimodules:
\[
\Theta\big((\eta_\rho, q), (\xi_{\rho'},
q')\big)=\left\{\begin{array}{ll}(\eta_\rho(\xi_\rho)\frac{1}{c^{[\rho]}(q,q^{-1}q')}q([\rho]),q^{-1}q')&\quad\text{if\
}[\rho']=[\rho]\\0&\quad\text{otherwise}.\end{array}\right.
\]
Here, $(a[\rho], q)$ is the function on
$\widehat{G}\times Q$ that takes the value $a\in \complex$ at
$([\rho],q)\in \widehat{G}\times Q$ and $0$ elsewhere. 

Using the cocycle condition (\ref{eq:cocycle_condition_c}) of $c$, we can easily check that $\Theta$ defines a $B$-$B$ bimodule map $\Theta: N\otimes_A M\to B$. Furthermore, it is not difficult to check that the image
of $\Theta$ contains all elements of the form $([\rho],q)$.
Hence $\Theta$ is surjective. By dimension counting,
we conclude that $\Theta$ must be an isomorphism.
\end{proof}

\begin{remark}\label{rmk:morita-comp}With our
definitions of the bimodules maps $\Xi$ and $\Theta$, one can easily
check that
\[
\begin{split} \Xi\big((\xi_\rho,q),
(\eta_{\rho'},q')\big)(\xi''_{\rho''}, q'')&=(\xi_\rho,q)\Theta\big(
(\eta_{\rho'},q'),(\xi''_{\rho''}, q'')\big),\\
\Theta\big((\eta_{\rho},q), (\xi_{\rho'},q')\big)(\eta_{\rho''},
q'')&=(\eta_{\rho},q)\Xi\big( (\xi_{\rho'},q'),(\eta_{\rho''},
q'')\big).
\end{split}
\]
This is crucial in the study of the centers of $A$ and $B$ in
Sec.  \ref{subsec:center}.
\end{remark}

Since Hochschild cohomology and $K$-theory are both Morita
invariants, we have the following:
\begin{corollary}
\label{cor:morita-K-coh}The Hochschild cohomology (respectively,
$K$-theory) of the algebra $\complex H$ is isomorphic to the
Hochschild cohomology (respectively, $K$-theory) of
$C(\widehat{G}\rtimes Q, c)$.
\end{corollary}

In the rest of this subsection, we study the twisted groupoid
algebra $C(\widehat{G}\rtimes Q, c)$. Denote the set of orbits of
the $Q$ action on $\widehat{G}$ by $\orb^Q(\widehat{G})$. The
groupoid $\widehat{G}\rtimes Q$ decomposes into a disjoint union of
groupoids, $\sqcup_{O_k\in \orb^Q(\widehat{G})}O_k\rtimes Q$, where
each component, $O_k\rtimes Q$, is a subgroupoid of
$\widehat{G}\rtimes Q$. The restriction of the $U(1)$-valued cocycle
$c$ yields a $U(1)$-valued cocycle $c_k$ on the component
$O_k\rtimes Q$. It is easy to see that the twisted groupoid algebra
decomposes into a direct sum of subalgebras,
\[
C(\widehat{G}\rtimes Q,c)=\bigoplus_{O_k\in \orb^Q(\widehat{G})} C(O_k\rtimes Q, c_k).
\]

Fix an orbit $O_k$ in $\orb^Q(\widehat{G})$ and consider the twisted
groupoid algebra $C(O_k\rtimes Q, c_k)$. Choose a point $[\rho]\in
O_k$ and let $\stab([\rho])\subset Q$ be the stabilizer subgroup of
$[\rho]$, which consists of $q\in Q$ such that $q([\rho])=[\rho]$.
The cocycle $c$ restricts to define a $U(1)$-valued cocycle
$c_{[\rho]}: \stab([\rho])\times \stab([\rho])\to U(1)$ on the
group $\stab([\rho])$. Given these data, we consider the {\em
twisted group algebra} $C(\stab([\rho]), c_{[\rho]})$. By
definition, $C(\stab([\rho]), c_{[\rho]})$ is additively equal to the group algebra $\complex \stab([\rho])$ but equipped with a
twisted product: $q_1\cdot q_2:=c_{[\rho]}(q_1,q_2)(q_1q_2).$ See e.g. \cite{adem_ruan}, \cite{kar} for discussions of twisted group algebras.


\begin{theorem}
\label{thm:orb-morita} The twisted groupoid algebra $C(O_k\rtimes Q,
c_k)$ is Morita equivalent to the twisted group algebra
$C(\stab([\rho]), c_{[\rho]})$.
\end{theorem}
\begin{proof}
We explain the geometric correspondent of this Morita equivalence in
the case when the cocycle $c$ is trivial. The transformation
groupoid $O_k\rtimes Q$ is Morita equivalent to the group
$\stab([\rho])$ with the equivalence bibundle $M_0:=\{([\rho],
q):q\in Q\}$. The groupoid $O_k\rtimes Q$ acts on $M_0$ from the
right by right multiplication, and the group $\stab([\rho])$ acts on
$M_0$ from the left by left multiplication. It is straightforward to
check that $M_0$ defines a Morita equivalent bibundle between
$\stab([\rho])$ and $O_k\rtimes Q$. The result of \cite{rmw} shows
that such an equivalent bibundle defines a Morita equivalence
between the group algebra $\complex \stab([\rho])$ and the groupoid
algebra $C(O_k\rtimes Q)$.

Inspired by this, we write down a Morita equivalence
$C(\stab([\rho]), c_{[\rho]})$-$C(O_k\rtimes Q, c_k)$ bimodule. Let
$M$ be the  space of functions on $M_0$.  Define the left
$C(\stab([\rho]), c_{[\rho]})$ module structure on $M$ by
\[
q_0\delta_{([\rho],q)}=c^{[\rho]}(q_0,q)\delta_{([\rho], q_0q)},\quad q_0\in \stab([\rho]), q\in Q,
\]
where $\delta_{([\rho],q)}$ denotes the function on $M_0$ taking the
value 1 at $([\rho],q)$ and zero everywhere else. We define the
right $C(O_k\rtimes Q, c_k)$ module structure on $M$ by
\[
\begin{split}
\delta_{([\rho],
q)}([\rho_0],q_0):=\left\{\begin{array}{ll}c^{\rho}(q,q_0)\delta_{[\rho],qq_0}&\quad\text{if\
}q([\rho])=[\rho_0]\\ 0&\quad\text{otherwise.}\end{array}\right.
\end{split}
\]

We also write down a Morita equivalence $C(O_k\rtimes Q,
c_k)$-$C(\stab([\rho]),c_{[\rho]})$ bimodule $N$. Let $N$ be the
space of functions on the set $\{(q^{-1}([\rho]),q):q\in Q\}$. Define
the left $C(O_k\rtimes Q, c_k)$ module structure by
\[
\begin{split}
([\rho_0],q_0)\delta_{(q^{-1}([\rho]),q)}:=\left\{\begin{array}{ll}c^{[\rho_0]}(q_0,q)\delta_{(q^{-1}_0q^{-1}[\rho],q_0q)}
&\quad\text{if}\ q_0([\rho])=q^{-1}([\rho])\\ 0&\quad\text{otherwise.}\end{array}\right.
\end{split}
\]
We also define the right $C(\stab([\rho]),c_{[\rho]})$ module structure by
\[
\delta_{(q^{-1}([\rho]),q)}q_0:=c^{q^{-1}([\rho])}(q,q_0)\delta_{(q^{-1}([\rho]), qq_0)}.
\]
And we define the bimodule map $X: M\otimes_{C(O_k\rtimes Q, c_k)}
N\to C(\stab([\rho]), c_{[\rho]})$ by
\[
\begin{split}
X\big(\delta_{([\rho], q_0)}, \delta_{(q_1^{-1}([\rho]),
q_1)}\big):=\left\{\begin{array}{ll}c^{[\rho]}(q_0,q_1)q_0q_1&\quad\text{if}\
q_0([\rho])=q_1^{-1}([\rho])\\
0&\quad\text{otherwise.}\end{array}\right.
\end{split}
\]
We define the bimodule map $Y:N\otimes_{C(\stab([\rho]), c_{[\rho]})}M\to C(O\rtimes Q, c)$ by
\[
Y\big(\delta_{(q_0^{-1}([\rho]),q_0)}, \delta_{([\rho],
q_1)}\big):=c^{q_0^{-1}([\rho])}(q_0,q_1)(q_0^{-1}([\rho]), q_0q_1).
\]
The verification that these data define Morita equivalence
bimodules is routine and left to the reader.
\end{proof}

In conclusion, we know that the category of representations of the
group $H$ is isomorphic to the category of modules of the
sum of twisted group algebras $C(\stab([\rho]), c_{[\rho]})$, where
the sum is taken over the set of $Q$-orbits in $\widehat{G}$ and we
choose an element $[\rho]$ from each $Q$-orbit. Modules of
the $c_{[\rho]}$-twisted group algebra correspond to projective
representations of the group $\stab([\rho])$ with cocycle
$c_{[\rho]}$. This is exactly what the Mackey machine \cite{fe-do}
states about representation theory of a finite group $H$ with a
normal subgroup $G$.

\subsection{Isomorphism between centers}\label{subsec:center}
According to Theorem \ref{thm:local-mackey}, the group algebra
$\complex H$ is Morita equivalent to the twisted groupoid algebra
$C(\widehat{G}\rtimes Q, c)$. It is known (see, e.g., \cite{an-fu})
that Morita equivalent unital algebras have isomorphic centers. Therefore the center of $\complex H$ is isomorphic to the center
of $C(\widehat{G}\rtimes Q, c)$. In this subsection, we analyze
the isomorphisms between the centers of $\complex H$ and
$C(\widehat{G}\rtimes Q, c)$ induced from the explicit bimodules constructed in the proof of Theorem \ref{thm:local-mackey}.

Let $A=\complex H$ and $B=C(\widehat{G}\rtimes Q, c)$. Let $M$
(respectively, $N$) be the $A$-$B$ (respectively, $B$-$A$) Morita
equivalence bimodule defined in the proof\footnote{We use
Proposition \ref{prop:matrix-coeff-algebra} to identify $\complex H$
with $\oplus_{[\rho]\in \widehat{G}}\End(V_\rho)\rtimes_{T,\tau}
Q$.} of Theorem \ref{thm:local-mackey}. Using the isomorphism
$\Xi:M\otimes_B N\to A$ (respectively, $\Theta:N\otimes_AM\to B$)
defined in the proof of Theorem \ref{thm:local-mackey}, we know
that the algebra $\End_B(M)$ (respectively, $\End_A(N)$) of linear
endomorphisms of $M$ (respectively, $N$) commuting with the action
of $B$ (respectively, $A$) is isomorphic to $A$ (respectively, $B$)
under the isomorphism $\Xi$ (respectively, $\Theta$).

We write out these isomorphisms more explicitly. For any $[\rho]\in
\widehat{G} $, choose a basis $\xi_i^\rho$ of $V_\rho$ and a dual basis $\eta^i_\rho$ of $V^*_\rho$ such that
$\eta^i_\rho(\xi_j^\rho)=\delta^i_j$ (the Kronecker delta function).
Define two maps 
\begin{eqnarray*}
&\Psi:\End_B(M)\to A,\quad x&\longmapsto\sum_{\rho, i} \Xi\big(x(\xi^{\rho}_i, 1), (\eta^i_\rho, 1)\big)\\
&\Phi:\End_A(N)\to B, \quad y&\longmapsto\sum_{\rho, i} \frac{1}{\dim(V_\rho)}\Theta\big(y(\eta^i_\rho, 1), (\xi_i^\rho, 1)\big).
\end{eqnarray*}
We explain the construction in more detail in the case of $\Psi$. As
$x$ acts on $M$ and commutes with the $B$ action, $x$ defines a
linear endomorphism on $M\otimes_B N$ commuting with the right $A$
action. $\Xi$ is an isomorphism from $M\otimes_B N$  to $A$ as an
$A$-$A$ bimodule. Hence, under the isomorphism $\Xi$, $x$ becomes  a
linear endomorphism of $A$ commuting with the right $A$ action. Such
a linear endomorphism is naturally identified with an element in $A$
by taking its value on the unit element of $A$. Applying this to
$x$, we map $x$ to $\Xi(x\Xi^{-1}(1))$. It is easy to check that
$\sum_{\rho, i}(\xi^\rho_i,1)\otimes (\eta^i_\rho,1)$ is mapped to
the unit under $\Xi$. This gives the above definition of $\Psi$. The same reasoning leads to the definition of $\Phi$. It is not hard to check that $\Psi$ and $\Phi$ are identity maps when we
restrict them to $A$ and $B$. Hence the action of $\Psi(x)$
(respectively, $\Phi(y)$) as an element in $A$ (respectively, $B$)
on $M$ (respectively, $N$) is identical to the action of $x$ on $M$
(respectively, $y$). From this, we can easily check that $\Psi$ and
$\Phi$ are algebra isomorphisms and inverses to each other.

Now we consider the application of the above maps to the centers of
$A$ and $B$. Let $y$ be an element in the center of $A=\complex H$.
As $y$ commutes with elements of $A$, $y$ as an endomorphism on $N$
is in $\End_A(N)$. Therefore, under the map $\Phi$, $y$ is mapped to
an element of $B$. Notice that $\Phi(y)$, as an element in $B$, acts
on $N$ in the same way as the action of $y$ on $N$. Therefore,
$\Phi(y)$ as an element in $B$ acts on $N$ and commutes with any
element of $B$. Therefore, we conclude that $\Phi(y)$ must be in the
center of $B$. A similar reasoning shows that if $x$ is an element
in the center of $B$, then $\Psi(x)$ is in the center of $A$. Thus,
we have constructed two algebra isomorphisms, $\Phi$ and $\Psi$,
that identify the centers of $\complex H$ and $C(\widehat{G}\rtimes
Q, c)$.

We study how the above maps are compatible with the
center of the group algebra $\complex Q$. Recall that $j:H\to Q$ is
a group epimorphism, and induces an algebra epimorphism
from $\complex H$ to $\complex Q$. The center $Z(\complex Q)$ has a
canonical basis indexed by conjugacy classes of $Q$. Define $Z(\complex H)_{{\<q\>}}$ to be the subspace of  the center $Z(\complex H)$ mapped under $j$ to the subspace of $Z(\complex Q)$ spanned by the basis element associated to the conjugacy class $\<q\>$. Accordingly the center of $\complex H$, as a vector space, has the following direct sum decomposition,
$$Z(\complex H)=\bigoplus_{\<q\>\subset Q} Z(\complex H)_{\<q\>}.$$

We examine the center of $C(\widehat{G}\rtimes Q,c)$ more carefully.
Recall that $\{([\rho],q)\}_{q\in Q, [\rho]\in \widehat{G}}$ forms a
basis of $C(O_k\rtimes Q, c_k)$. If
$f=\sum_{[\rho],q}a_{[\rho],q}([\rho],q)$ is in the center of
$C(\widehat{G}\rtimes Q, c)$, then for any $([\rho_0],q_0)\in
C(O_k\rtimes Q, c_k)$, we have
$f([\rho_0],q_0)=([\rho_0],q_0)f,$  
which is equivalent to
\[
\sum_{q}a_{q^{-1}([\rho_0]),q}c^{q^{-1}([\rho_0])}(q,q_0)(q^{-1}([\rho_0]),qq_0)=\sum_{q}a_{q_0([\rho_0]),q}c^{[\rho_0]}(q_0,q)([\rho_0],
q_0q).
\]
Replacing $q$ by $q_0qq_0^{-1}$ in the left hand side of the above equality, we obtain
\[
\begin{split}
\sum_{q}a_{q_0q^{-1}q_0^{-1}([\rho_0]),q_0qq_0^{-1}}&c^{q_0q^{-1}q_0^{-1}([\rho_0])}(q_0q^{-1}q_0^{-1},
q_0)(q_0q^{-1}q_0^{-1}([\rho_0]),
q_0q)\\
&=\sum_{q}a_{q_0([\rho_0]),q}c^{[\rho_0]}(q_0,q)([\rho_0], q_0q).
\end{split}
\]
Comparing the two sides of the above equation, we see that, for every $q\in Q$,
\[
a_{q_0q^{-1}q_0^{-1}([\rho_0]),q_0qq_0^{-1}}c^{q_0q^{-1}q_0^{-1}([\rho_0])}(q_0q^{-1}q_0^{-1}([\rho_0]),
q_0q)=a_{q_0([\rho_0]),q}c^{[\rho_0]}(q_0,q)([\rho_0], q_0q).
\]
This implies the following results.
\begin{enumerate}
\item  If $q$ is
such that $q_0q^{-1}q_0^{-1}([\rho_0])\ne [\rho_0]$ (equivalently
$q(q_0^{}([\rho_0]))\ne q_0([\rho_0])$), then
\[
a_{q_0q^{-1}q_0^{-1}([\rho_0]),q_0qq_0^{-1}}=a_{q_0([\rho_0]),q}=0.
\]
\item If $q$ is such that $q_0q^{-1}q_0^{-1}([\rho_0])= [\rho_0]$ (equivalently
$q(q_0([\rho_0]))=q_0([\rho_0])$), then
\[
a_{q_0q^{-1}q_0^{-1}([\rho_0]),q_0qq_0^{-1}}c^{q_0q^{-1}q_0^{-1}([\rho_0])}(q_0q^{-1}q_0^{-1},
q_0q)=a_{q_0([\rho_0]),q}c^{[\rho_0]}(q_0,q).
\]
As $q(q_0([\rho_0]))=q_0([\rho_0])$, the above equality is
equivalent to
\[
a_{[\rho_0],
q_0qq_0^{-1}}c^{[\rho_0]}(q_0q^{-1}q_0^{-1},q_0)=a_{q_0([\rho_0]),
q}c^{[\rho_0]}(q_0,q),
\]
which is equivalent to (replacing $q_0([\rho_0])$ by $[\rho_0]$)
\[
a_{q_0^{-1}([\rho_0]),q_0qq_0^{-1}}c^{q_0^{-1}([\rho_0])}(q_0q^{-1}q_0^{-1},
q_0)=a_{[\rho_0], q}c^{q_0^{-1}([\rho_0])}(q_0,q).
\]
\end{enumerate}
In summary, we can write an element $f$ in the center of
$C(\widehat{G}\rtimes Q, c)$ as
\begin{equation}\label{eq:center}
f=\sum_q \sum_{[\rho], q([\rho])=[\rho]}a_{[\rho],q}([\rho],q),
\end{equation}
such that
\begin{equation}\label{eq:conj-a} a_{[\rho_0],
q}=a_{q_0^{-1}([\rho_0]),q_0qq_0^{-1}}\frac{c^{q_0^{-1}([\rho_0])}(q_0q^{-1}q_0^{-1},q_0)}{c^{q_0^{-1}([\rho_0])}(q_0,q)}.
\end{equation}
By Eq.  (\ref{eq:conj-a}), $f$ can be written as
\[
f=\sum_{\<q\>\subset Q}\sum_{q_0\in \<q\>}\sum_{[\rho],
q_0([\rho])=[\rho]}a_{[\rho],q_0}([\rho],q_0),
\]
such that every component
\begin{equation}\label{eq:comp-q}
\sum_{q_0\in \<q\>}\sum_{[\rho],
q_0([\rho])=[\rho]}a_{[\rho],q_0}([\rho],q_0)
\end{equation}
is in the center of $C(\widehat{G}\rtimes Q, c)$. Define $Z(C(\widehat{G}\rtimes Q,
c))_{\<q\>}$ to be the subspace of $Z(C(\widehat{G}\rtimes Q, c))$ consisting of elements of the form of Eq. (\ref{eq:comp-q}). The center $Z(C(\widehat{G}\rtimes Q, c))$
decomposes into a direct sum of subspaces $Z(C(\widehat{G}\rtimes Q,
c))_{\<q\>}$ indexed by conjugacy classes of $Q$,
$$Z(C(\widehat{G}\rtimes Q, c))=\bigoplus_{\<q\>\subset Q}Z(C(\widehat{G}\rtimes Q,
c))_{\<q\>}.$$

The map $\chi\circ \alpha: \complex H\rightarrow\bigoplus_{[\rho]}\End(V_\rho)\rtimes_{T,\tau}Q$ constructed in Proposition \ref{prop:matrix-coeff-algebra} is an algebra isomorphism. Composing this
isomorphism with the above map $\Phi$, we obtain an algebra
isomorphism from the center of $\complex H$ to the center of
$C(\widehat{G}\rtimes Q, c)$. We denote this map by $I$.

\begin{prop}\label{prop:conjugacy}
The isomorphism $I$ is
compatible with the decompositions into subspaces indexed by
conjugacy classes of $Q$, i.e., $I$ is an isomorphism from
$Z(\complex H)_{\<q\>}$ to $Z(C(\widehat{G}\rtimes Q, c))_{\<q\>}$.
\end{prop}
\begin{proof}
We prove that the isomorphism $I$ maps the subspace $Z(\complex
H)_{\<q\>}$, indexed by the conjugacy class $\<q\>$ of $Q$, into
the subspace of $Z(C(\widehat{G}\rtimes Q, c))_{\<q\>}$ with the
same index. Then, by the fact that $I$ is an isomorphism, we know
that $I$ must be an isomorphism between each pair of subspaces.

We use the isomorphism $\alpha$ to identify $H$ with the group
$G\rtimes_{s,\tau}Q$. With the identification, the homomorphism $j$
maps the group $G\rtimes_{s,\tau}Q$ to $Q$ by taking the second
component. In particular, if an element $f$ is in the component
$Z(\complex H)_{\<q\>}$, $f$ must be of the form $\sum_{g, q_0\in
\<q\>}f_{g,q_0}(g,q_0).$ Applying $\chi$ to $f$, we get an element
$\chi(f)=\sum_{[\rho]\in \widehat{G}, q_0\in
\<q\>}\sum_{g}f_{g,q_0}(\rho(g),q_0).$ Since $f$ is in the center of
$\complex H$, we can apply the map $\Phi$ to obtain an element of
$Z(C(\widehat{G}\rtimes Q,c))$,
\begin{eqnarray*}
&\Phi(\chi(f))&=\sum_{\rho_1,
i}\frac{1}{\dim(V_{\rho_1})}\Theta\big((\eta_{\rho_1}^i,1)\chi(f),
(\xi_i^{\rho_1},1)\big)\\
&&=\sum_{q_0\in \<q\>}\sum_{[\rho_0]\in \widehat{G}}\sum_{g\in
G}f_{g,q_0}\sum_{\rho_1,i}\frac{1}{\dim(V_{\rho_1})}\Theta((\eta_{\rho_1}^i,1)(\rho_0(g),q_0),
(\xi_i^{\rho_1},1))\\
&&=\sum_{q_0\in \<q\>}\sum_{[\rho_0]\in \widehat{G}, i}\sum_{g\in
G}\frac{f_{g,q_0}}{\dim(V_{\rho_0})}\Theta\big((c^{[\rho_0]}(q_0,q_0^{-1})\eta_{\rho_0}^{i}
\circ \rho_0(g) \circ {T^{[\rho_0]}_{q_0}}^{-1}, q_0^{-1}),(\xi_i^{\rho_0},1)\big)\\
&&=\sum_{q_0\in \<q\>}\sum_{\tiny \begin{array}{c}[\rho_0]\in \widehat{G},\\
q_0([\rho_0])=[\rho_0]\end{array}}\sum_{i, g\in
G}\frac{f_{g,q_0}}{\dim(V_{\rho_0})}(\frac{c^{[\rho_0]}(q_0,q_0^{-1})}{c^{[\rho_0]}(q_0^{-1},q_0)}\eta_{\rho_0}^{i}
\circ \rho_0(g) \circ
{T^{[\rho_0]}_{q_0}}^{-1}(\xi_i^{\rho_0})[\rho_0], q_0)\\
&&=\sum_{q_0\in \<q\>}\sum_{\tiny{\begin{array}{c}[\rho_0]\in \widehat{G},\\
q_0([\rho_0])=[\rho_0]\end{array} }}\sum_{i, g\in
G}\frac{f_{g,q_0}}{\dim(V_{\rho_0})}(\eta_{\rho_0}^{i} \circ
\rho_0(g) \circ {T^{[\rho_0]}_{q_0}}^{-1}(\xi_i^{\rho_0})[\rho_0],
q_0).
\end{eqnarray*}
From this computation, we see that $\Phi(\chi(f))$ belongs to
$Z(C(\widehat{G}\rtimes Q, c))_{\<q\>}$.
\end{proof}

By Theorem \ref{thm:local-mackey}, $\complex H$ is Morita equivalent
to $C(\widehat{G}\rtimes Q, c)$. By Theorem \ref{thm:orb-morita},
$C(\widehat{G}\rtimes Q, c)$ is Morita equivalent to $\oplus_{O_k\in
\orb^Q(\widehat{G})}C(\stab([\rho_k]), c_{[\rho_k]}) $, where
$[\rho_k]$ is an element in the orbit $O_k$. With the Morita equivalence bimodules constructed in
Theorems \ref{thm:local-mackey}, we obtain an algebra isomorphism $I$,
\begin{equation}\label{eqn:orbit_decomp_of_center}
I: Z(\complex H)\longrightarrow Z(\oplus_{O_k\in
\orb^Q(\widehat{G})}C(\stab([\rho_k]), c_{[\rho_k]}))=\oplus_{O_k\in
\orb^Q(\widehat{G})}Z(C(\stab([\rho_k]), c_{[\rho_k]})).
\end{equation}
 It is easy to check that on $\complex H$ (respectively, on each
$C(\stab([\rho_k]), c_{[\rho_k]}$)), there is a canonical trace\footnote{We obtain $\tr_H$ (and $\tr_{[\rho]}$) by taking the trace on the canonical representation of $H$ (and $\stab([\rho_k]), c_{[\rho_k]}$) on $\complex H$ (and $C(\stab([\rho_k]), c_{\rho_k})$).}
$\tr_H$ (respectively, $\tr_{[\rho_k]}$) which takes the
value\footnote{We have chosen this particular normalization to make
it compatible with the Poincar\'e pairing on $BH$.} $1/|H|$
(respectively, $1/|\stab([\rho_k])|$) on the identity element and 0
otherwise. Furthermore, the linear combination
$\sum_{k}a_k\tr_{[\rho_k]}$ for arbitrary coefficients $a_k$ defines
a trace on the algebra $\oplus_{O_k\in
\orb^Q(\widehat{G})}C(\stab([\rho_k]), c_{[\rho_k]})$. Both $\tr_H$ and
$\sum_k a_k\tr_{[\rho_k]}$ restricts to define traces on the corresponding centers, $Z(\complex H)$ and 
$
Z(\oplus_{O_k\in
\orb^Q(\widehat{G})}C(\stab([\rho_k]),c_{[\rho_k]})). $
The next
result identifies which trace on $Z(\oplus_{O_k\in
\orb^Q(\widehat{G})}C(\stab([\rho_k]), c_{[\rho_k]}))$ pulls back to
$\tr_H$ via $I$.

\begin{prop}\label{prop:trace}
The isomorphism $I$ pulls back the trace $\sum_{k}\frac{\dim(V_{\rho_k})^2}{|G|^2}\tr_{[\rho_k]}$ to the trace $\tr_H$.
\end{prop}
\begin{proof}
In the proof of Proposition \ref{prop:conjugacy}, we obtained an
explicit map $I$ from $Z(\complex H)$ to $Z(C(\widehat{G}\rtimes Q,
c))$. We recall the formula of this map. If $f=\sum_{g,q}f_{g,q}(g,
q)$ is in the center of $\complex H$, then
\begin{eqnarray*}
&I(f)&=\sum_{g,q}\sum_{[\rho]\in \widehat{G},
q([\rho])=[\rho],i}\frac{1}{\dim(V_\rho)}f_{g,q}(\eta^i_\rho\circ
\rho(g)\circ {T^{[\rho]}_q}^{-1}(\xi^\rho_i)[\rho],q)\\
&&=\sum_{[\rho]\in \widehat{G},
q([\rho])=[\rho]}\sum_{i,g,q}\frac{1}{\dim(V_\rho)}f_{g,q}(\eta^i_\rho\circ
\rho(g)\circ {T^{[\rho]}_q}^{-1}(\xi^\rho_i)[\rho],q)\\
&&=\sum_{O_k\in \orb^Q(\widehat{G})}\sum_{\tiny
\begin{array}{c}[\rho]\in
O_k,\\
q([\rho])=[\rho]\end{array}}\sum_{i,g,q}\frac{1}{\dim(V_\rho)}f_{g,q}(\eta^i_\rho\circ
\rho(g)\circ {T^{[\rho]}_q}^{-1}(\xi^\rho_i)[\rho],q).
\end{eqnarray*}
Define
\[
I(f)_k=\sum_{\tiny \begin{array}{c}[\rho]\in
O_k,\\
q([\rho])=[\rho]\end{array}}\sum_{i,g}\frac{1}{\dim(V_\rho)}f_{g,q}(\eta^i_\rho\circ
\rho(g)\circ {T^{[\rho]}_q}^{-1}(\xi^\rho_i)[\rho],q)\in
C(O_k\rtimes Q, c).
\]
Then $I(f)=\sum_{O_k\in \orb^Q(\widehat{G})}I(f)_k$, and $I(f)_k$ is
an element in the center of $C(O_k\rtimes Q, c_k)$.

Using the data of Morita equivalent bimodules constructed in Theorem
\ref{thm:orb-morita}, we write down an explicit map $I'_k$ from
$Z(C(O_k\rtimes Q, c_k))$ to $Z(C(\stab([\rho_k]), [c_{\rho_k}]))$.
If $f_k$ is in $Z(C(O_k\rtimes Q, c_k))$, then
\[
I'_k(f_k):=\sum_{q\in
Q}\frac{{c^{[\rho_k]}(q,q^{-1})}^{-1}}{|Q|}X\big(\delta_{([\rho_k],q)}f_k,
\delta_{(q([\rho_k]),q^{-1})}\big).
\]
Here the map $X$ is defined in the proof of Theorem \ref{thm:orb-morita}. Applying the above formula to the explicit expression of $f_k$, we
have
\begin{eqnarray*}
I'_k(f_k)&&=\sum_{q\in
Q}\frac{{c^{[\rho_k]}(q,q^{-1})}^{-1}}{|Q|}\sum_{\tiny
\begin{array}{c}[\rho_1]\in
O_k,\\
q_1([\rho_1])=[\rho_1]\end{array}}\sum_{i,g}\\
&&\qquad\qquad
X\Big(\delta_{([\rho_k],q)}\frac{1}{\dim(V_{\rho_1})}f_{g,q_1}(\eta^i_{\rho_1}\circ
\rho_1(g)\circ
{T^{[\rho_1]}_{q_1}}^{-1}(\xi^{\rho_1}_i)[\rho_1],q_1),
\delta_{(q([\rho_k]),q^{-1})}\Big)\\
&&=\sum_{q\in Q}\frac{{c^{[\rho_k]}(q,q^{-1})}^{-1}}{|Q|}\sum_{\tiny
\begin{array}{c}[\rho_1]\in
O_k,\\
q_1([\rho_1])=[\rho_1]\end{array}}\sum_{i,g}\frac{1}{\dim(V_{\rho_1})}f_{g,q_1}\eta^i_{\rho_1}\circ
\rho_1(g)\circ {T^{[\rho_1]}_{q_1}}^{-1}(\xi^{\rho_1}_i)\\
&&\qquad\qquad X\Big(\delta_{([\rho_k],q)}([\rho_1],q_1),
\delta_{(q([\rho_k]),q^{-1})}\Big)\\
&&=\sum_{q\in Q}\sum_{
q_1(q([\rho_k]))=q([\rho_k])}\sum_{i,g}\frac{{c^{[\rho_k]}(q,q^{-1})}^{-1}}{|Q|\dim(V_{q([\rho_k])})}f_{g,q_1}\eta^i_{q([\rho_k])}\circ
q(\rho_k)(g)\circ {T^{q([\rho_k])}_{q_1}}^{-1}(\xi^{q([\rho_k])}_i)\\
&&\qquad\qquad X\Big(\delta_{([\rho_k],q)}(q([\rho_k]),q_1),
\delta_{(q([\rho_k]),q^{-1})}\Big)\\
&&=\sum_{q\in Q}\sum_{q_1(q([\rho_k]))=q([\rho_k])}\sum_{i,g}
\frac{{c^{[\rho_k]}(q,q^{-1})}^{-1}}{|Q|\dim(V_{q([\rho_k])})}f_{g,q_1}\eta^i_{q([\rho_k])}\circ
q(\rho_k)(g)\circ
{T^{q([\rho_k])}_{q_1}}^{-1}(\xi^{q([\rho_k])}_i)\\
&&\qquad \qquad c^{[\rho_k]}(q,q_1)X\Big(\delta_{([\rho_k],qq_1)},
\delta_{(q([\rho_k]),q^{-1})}\Big)\\
&&=\sum_{q\in Q}\sum_{q_1(q([\rho_k]))=q([\rho_k])}\sum_{i,g}
\frac{{c^{[\rho_k]}(q,q^{-1})}^{-1}}{|Q|\dim(V_{q([\rho_k])})}f_{g,q_1}\eta^i_{q([\rho_k])}\circ
q([\rho_k])(g)\circ
{T^{q([\rho_k])}_{q_1}}^{-1}(\xi^{q([\rho_k])}_i)\\
&&\qquad \qquad c^{[\rho_k]}(q,q_1)c^{[\rho_k]}(qq_1,q^{-1})qq_1q^{-1}.
\end{eqnarray*}

We notice that if the coefficient of $f_k$ on every component
$([\rho],1)$ for any $[\rho]\in O$ is equal to zero, then the
coefficient of $I_k'(f_k)$ at the identity element of
$C(\stab([\rho_k]), c_{[\rho_k]})$ is equal to zero. In particular,
the evaluation of $\tr_{\rho_k}$ on $I'_k(f_k)$ is equal to zero.
Hence
\begin{eqnarray*}
&\tr_{\rho_k}(I'_k(f_k))&=\frac{1}{|\stab([\rho_k])|}\sum_{q\in
Q}\sum_{i,g}
\frac{1}{|Q|\dim(V_{q([\rho_k])})}f_{g,1}\eta^i_{q([\rho_k])}\circ
q([\rho_k])(g)(\xi^{q([\rho_k])}_i)\\
&&=\sum_{\rho\in O_k}\sum_{i,g}
\frac{1}{|Q|\dim(V_{\rho})}f_{g,1}\eta^i_{\rho}\circ
\rho(g)(\xi^{\rho}_i),
\end{eqnarray*}
where in the last equality, we used the fact that for $q\in
\stab([\rho_k])$, $q([\rho_k])=[\rho_k]$.

Now summing over all orbits $O_k\in \orb^Q(\widehat{G})$, we have
that
\begin{eqnarray*}
&\sum_{O_k\in
\orb^Q(\widehat{G})}\frac{\dim(V_{\rho_k})^2}{|G|}\tr_{\rho_k}(I'(f))&=\sum_{\rho\in
\widehat{G}}\sum_{i,g}
\frac{\dim(V_{\rho})^2}{|G|^2|Q|\dim(V_{\rho})}f_{g,1}\eta^i_{\rho}\circ
\rho(g)(\xi^{\rho_k}_i)\\
&=\sum_g f_{g,1}\frac{1}{|H|}\sum_{\rho,
i}\frac{\dim(V_{\rho})}{|G|}\eta^i_{\rho}\circ
\rho(g)(\xi^{\rho_k}_i)&=\sum_g f_{g,1}\frac{1}{|H|}\sum_{\rho}\dim(V_{\rho})\tr_\rho(g),
\end{eqnarray*}
where we used the fact that $\dim(V_{\rho})=\dim(V_{\rho_k})$ if
$[\rho]\in O_k$, and $\tr_\rho$ is the $1/|G|$ times the standard
trace on $\End(V_\rho)$. A standard result (see, e.g.,
\cite[Exercise 3.32]{fu-ha}) in group representation theory implies
that
\[
\sum_g
f_{g,1}\frac{1}{|H|}\sum_{\rho}\dim(V_{\rho})\tr_\rho(g)=\sum_g
f_{g,1}\frac{1}{|H|}\tr_G(g)=\frac{1}{|H|}f_{1,1}=\tr_H(f).
\]
This proves that the pull back of $\sum_k\frac{\dim(V_{\rho_k})^2}{|G|^2}\tr_{[\rho_{k}]}$ along the
algebra isomorphism $I$ is equal to $\tr_H$.
\end{proof}

\begin{remark}\label{rmk:I-iso-group}The explicit formula
for the isomorphism $I$,
\[
I(f)=\sum_{q}\sum_{\tiny\begin{array}{c}[\rho_0]\in \widehat{G}\\
q([\rho])=[\rho]
\end{array}}\sum_g
\frac{f_{g,q}}{\dim(V_\rho)}\tr\big(\rho(g){T^{[\rho]}_q}^{-1}\big)([\rho],q),
\]
is crucial in the above proofs of Propositions
\ref{prop:conjugacy}-\ref{prop:trace}. In Sec. 
\ref{subsec:twisted-coh}, we will give a generalization of the
formula for $I$ on a $G$-gerbe. See Eq.  (\ref{eq:I-alpha}).
\end{remark}
\section{Groupoid algebras and Hochschild cohomology of gerbes on orbifolds}\label{sec:hochschild}
In this section, we generalize the formulation of the Mackey machine in Sec.  \ref{sec:mackey} to groupoids. Our results relate several aspects of the geometry of a $G$-gerbe $\Y$ on an orbifold $\B$ to aspects of the geometry of the dual orbifold $\widehat{\Y}$ twisted by the flat $U(1)$-gerbe $c$.
\subsection{Global quotient}\label{subsec:global-quo}
In this subsection, we consider the special case of a global
quotient, which is very close to the developments in Sec. 
\ref{sec:mackey}. Consider the group extension as in Eq. 
(\ref{eq:extension}). Let $M$ be a smooth manifold. We assume that $H$ acts on $M$ by
diffeomorphisms so that the induced $G$-action on
$M$ is trivial. Therefore, the composition of $H\to \text{Diff}(M)$
with a section $s: Q\to H$ of (\ref{eq:extension}) yields a
well-defined $Q$-action on $M$. Consequently, this defines a
$G$-gerbe $\Y=[M/H]$ over $\B=[M/Q]$.
In terms of groupoids, this gives the following
extension of Lie groupoids, which generalizes Eq. 
(\ref{eq:extension}),
\begin{equation}\label{eq:gpd-ext}
M\times G\rightarrow M\rtimes H\rightarrow M\rtimes
Q.
\end{equation}
We follow the idea in Sec.  \ref{sec:mackey} to study the above
groupoid extension. First, note that the section $s$ for the
extension (\ref{eq:extension}) defines a section of the groupoid
extension (\ref{eq:gpd-ext}).

Let $\cala$ be a sheaf of algebras on $M$ with an $H$ action. In
this paper, $\cala$ is either the sheaf of smooth functions on $M$
or the sheaf of deformation quantization of $M$ when $M$ is a
symplectic manifold. Such a sheaf $\cala$ with an $H$ action is
called an $\frakH$-sheaf on the transformation groupoid
$\frakH:=M\rtimes H\rightrightarrows M$. We are interested in the
crossed product algebra $\cala\rtimes H$.

Similar to the approach in Sec.  \ref{sec:mackey}, we use the data
about $G$ and $Q$ to study the algebra $\cala\rtimes H$. With the
same notations, we consider the space $M\times \widehat{G}$ equipped
with the diagonal $Q$-action. Here we use the $Q$-action on
$\widehat{G}$ as defined in Sec.  \ref{subsec:groupalgebra}. The
dual orbifold associated to the $G$-gerbe
$\Y$ is the quotient
$\widehat{\Y}=[(M\times \widehat{G})/Q]$. On $M\times \widehat{G}$,
there is a vector bundle\footnote{Here we allow the ranks of a
vector bundle to be different on different connected components.}
$\calv_G\to M\times \widehat{G}.$ 
For every $[\rho]\in \widehat{G}$,
the bundle $\calv_G$ on the component $M\times [\rho]$ is equal to
$M\times V_\rho\rightarrow M\times [\rho]$.  It is important to
observe that the vector bundle $\calv_G$ is not $H$-equivariant.
Instead, the cocycle $c$ defined in Proposition
\ref{prop:u(1)-cocycle} determines the obstruction for $\calv_G$ to
be $H$-equivariant. More precisely, $\calv_G$ is a $c^{-1}$-twisted
$H$-equivariant vector bundle on $M\times \widehat{G}$. See Lemma
\ref{lem:Q_str_on_V} below. In the language of gerbes, this means
that there is a $U(1)$-gerbe over the orbifold $[(M\times
\widehat{G})/H]$ on which $\calv_G$ becomes a vector bundle.
Furthermore, as $H$ is finite, this $U(1)$-gerbe is flat and it may
be represented by the cocycle $c^{-1}$.

The $H$-equivariant sheaf $\cala$ on $M$ lifts to define a
$Q$-equivariant sheaf $\widetilde{\cala}$ on $M\times \widehat{G}$.
Using the $U(1)$-valued $2$-cocycle $c$, we define a twisted crossed
product algebra $\widetilde{\cala}\rtimes_c Q$ to be the space
$\Gamma(\widetilde{\cala})$ of sections of $\widetilde{\cala}$ with
the following product:
\[
(a_1,q_1)\circ (a_2,q_2):=(a_1q_1(a_2)c(q_1, q_2),q_1q_2), \quad a_1, a_2\in \Gamma(\widetilde{\cala}),
\]
where $c(q_1,q_2)$ is an element in $U(1)$ identified as a scalar on the unit circle of $\complex$.

\begin{proposition} \label{prop:global-quotient}The ($c$-twisted) crossed product algebras $\cala\rtimes H$ and $\widetilde{A}\rtimes_c Q$ are Morita equivalent.
\end{proposition}
\begin{proof}
We will define the bimodules. Set
$A:=\cala\rtimes H$ and $B:=\widetilde{\cala}\rtimes_c Q$.

Consider the vector bundle $\calv_G$ over $M\times \widehat{G}$.
Define an $A$-$B$ bimodule $\calm$ by
\[
\calm:=\Gamma(\widetilde{\cala}\otimes \calv_G)\times Q.
\]
The left $A$-module structure on $\calm$ is defined by
\[
(a_0,x_{\rho_0}, q_0)(a, \xi_\rho, q):=\left\{\begin{array}{ll}(a_0q_0(a), c^{[\rho_0]}(q_0,q)x_{\rho_0}{T^{[\rho_0]}_{q_0}}^{-1}(\xi_{q_0([\rho_0])}), q_0q)&\quad\text{if}\ [\rho]=q_0([\rho_0])\\ 0&\quad\text{otherwise}\end{array}\right. .
\]
The right $B$-module structure on $\calm$ is defined by
\[
(a, \xi_\rho, q)(a_1, [\rho_1],
q_1):=\left\{\begin{array}{ll}(aq(a_1), c^{[\rho]}(q,q_1)\xi_\rho,
qq_1)&\quad \text{if}\ [\rho_1]=q([\rho])\\
0&\quad \text{otherwise}\end{array}\right..
\]

Consider the dual bundle $\calv^*_G$ over $M\times \widehat{G}$.
Define a $B$-$A$ bimodule $\caln$ by
\[
\caln:=\Gamma(\widetilde{\cala}\otimes \calv^*_G)\times Q.
\]
The right $A$-module structure on $\caln$ is defined by
\[
(a, \eta_{\rho}, q)(a_0, x_{\rho_0},
q_0)=\left\{\begin{array}{ll}(q_0^{-1}(aa_0),
c^{[\rho_0]}(q_0,q_0^{-1}q)\eta_{\rho_0}\circ
x_{\rho_0}\circ{T^{[\rho_0]}_{q_0}}^{-1}, q_0^{-1}q)&\text{if\ }[\rho]=[\rho_0]\\
0&\text{otherwise}\end{array}\right. ,
\]
and the left $B$-module structure on $\caln$ is defined by
\[
(a_1, [\rho_1],q_1)(a, \eta_{\rho},
q):=\left\{\begin{array}{ll}(qq_1^{-1}(a_1)a, c^{[\rho]}(qq_1^{-1},
q_1)\eta_\rho, qq_1^{-1})&\text{if}\ [\rho_1]=qq_1^{-1}([\rho])\\
0&\text{otherwise}.\end{array}\right.
\]

Next, we define an isomorphism $\Theta:\caln\otimes_A \calm \to B$
of $B$-$B$ bimodules by
\[
\Theta\big((a, \eta_\rho, q), (a', \xi_{\rho'},
q')\big)=\left\{\begin{array}{ll}(q^{-1}(aa'),
\eta_\rho(\xi_\rho)\frac{1}{c^{[\rho]}(q,q^{-1}q')}q([\rho]),q^{-1}q')&\quad\text{if\
}[\rho']=[\rho]\\0&\quad\text{otherwise},\end{array}\right.
\]
and an isomorphism $\Xi: \calm\otimes_B \caln\to A$ of $A$-$A$
bimodules by
\[
\begin{split}
&\Xi\big((a, \xi_\rho, q),(a', \eta_{\rho'},q')\big)\\
:=&\left\{\begin{array}{ll}(aq{q'}^{-1}(a'), \xi_\rho\otimes \eta_{qq'^{-1}([\rho])}\circ T^{[\rho]}_{qq'^{-1}}
\frac{c^{qq'^{-1}([\rho])}(q'q^{-1},q)}{c^{qq'^{-1}([\rho])}(q'q^{-1},
qq'^{-1})},
qq'^{-1})&\text{if}\  qq'^{-1}([\rho])=[\rho']\\
0&\text{otherwise}.\end{array}\right.
\end{split}
\]
The details of checking the equivalence bimodule structures are routine and are omitted.
\end{proof}
Similar to Proposition \ref{prop:global-quotient}, we can easily
generalize Theorem \ref{thm:orb-morita}.

\begin{proposition}\label{prop:glob-quot-decomp}
The algebra $\widetilde{A}\rtimes_c Q$ is Morita equivalent to the
algebra $\bigoplus_{O\in \orb^Q(\widehat{G})}
\cala\rtimes_{c_{\rho_O}}\stab(\rho_O)$, where $\rho_O$ is the
chosen representative of an element $[\rho_O]\in O$ in the orbit
$O$, and $c_{\rho_O}$ is the cocycle obtained from $c$ by
restriction.
\end{proposition}
\begin{proof}
The proof closely follows the proof of Theorem \ref{thm:orb-morita}
and is left to the reader.
\end{proof}

In summary,  we have a natural
generalization of the Mackey machine to the groupoid $M\rtimes H$.
\begin{theorem}\label{thm:main-global-quot}
The crossed product algebra $\cala\rtimes H$ is Morita equivalent to
the direct sum $$\bigoplus_{O\in \orb^Q(\widehat{G})}
\cala\rtimes_{c_{\rho_O}}\stab(\rho_O)$$ of twisted crossed product
algebras.
\end{theorem}
\subsection{General case}\label{subsec:general}
In this subsection, we generalize the groupoid
extension Eq.  (\ref{eq:gpd-ext}). Let $\frakH, \frakQ, \frakG$
be proper \'etale groupoids and $i:\frakG\rightarrow \frakH$ and
$j:\frakH\rightarrow \frakG$  groupoid morphisms that fit in the
exact sequence
\begin{equation}\label{eq:gpd-extension}
\frakG\stackrel{i}{\longrightarrow}\frakH\stackrel{j}{\longrightarrow}\frakQ.
\end{equation}
We make the following assumptions:
\begin{enumerate}
\item The groupoid $\frakG$ is a locally trivial bundle of groups over 
$\frakG_0$ with fibers isomorphic to $G$.
\item The morphisms $i$ and $j$ are identity maps when
restricted to the unit space.
\end{enumerate}
According to \cite{la-st-xu}, the Morita equivalence class of the
above $G$-groupoid extension defines a $G$-gerbe
$\Y=[\frakH_0/\frakH]$ on the orbifold $\B=[\frakQ_0/\frakQ]$
associated with the groupoid $\frakQ$, and every $G$-gerbe over $\B$
arises in this way. We are interested in the geometry of the
$G$-gerbe $\Y\to \B$.

We recall the procedure of refining an \'etale groupoid
$\calg\rightrightarrows \calg_0$ by a countable open covering
$\calu$ of $\calg_0$. For every $U_i,U_j\in \calu$, consider the
subspace $\calg^{U_i}_{U_j}:=\{g\in \calg| s(g)\in U_i, t(g)\in
U_j\}$. Define a new groupoid $\calg_\calu$ by
$\coprod_{ij}\calg^{U_i}_{U_j}\rightrightarrows \coprod_i U_i$.
Moerdijk and Pronk \cite{mo-pr-ind} proved that $\calg_\calu$ is
Morita equivalent to $\calg$, and that if we choose  every $U_i$ to
be sufficiently small, the groupoid $\calg_\calu$ and every level of
its nerve space are disjoint unions of contractible open sets. (Let
$\calg$ be a groupoid. For $n\geq 0$, the $n$-th level $\calg_n$ of
the nerve space of $\calg$ is defined to be the space of $n$
composable arrows in $\calg$.)

Laurent-Gengoux, Sti\'enon, and Xu proved \cite{la-st-xu} that a
groupoid extension like Eq.  (\ref{eq:gpd-extension}) always has
a refinement with respect to a covering $\calu$ of $\frakQ_0$ such
that the kernel of the groupoid extension is a trivial bundle of groups isomorphic to $G$.
Combining this result with the result from Moerdijk and Pronk
\cite{mo-pr-ind}, we conclude that there is a covering $\calu$ of
$\frakQ_0$ such that the associated refinement of the groupoid
extension
\[
\coprod_i U_i\times G\rightarrow \frakH_\calu \rightarrow
\frakQ_\calu
\]
is Morita equivalent to the original extension
(\ref{eq:gpd-extension}) and every level of the
nerve space of $\frakQ$ is a disjoint unions of contractible open sets. Since
$\frakH_\calu$ is Morita equivalent to $\frakH$, the groupoid
algebra of $\frakH_\calu$ is Morita equivalent to the groupoid
algebra of $\frakH$. For our purpose, we may thus study the groupoid algebra of $\frakH_\calu$.
To simplify notations, we will assume to work with the
groupoid extension $M\times G\stackrel{i}{\rightarrow} \frakH\stackrel{j}{\rightarrow} \frakQ,$
where $\frakH_0=\frakQ_0=M$, and $\frakQ$ is a disjoint union of contractible open sets. We generalize the Mackey machine formulated in Sec. 
\ref{sec:mackey} to study the above $G$-extension of proper \'etale groupoids.

As $\frakQ$ is a disjoint union of contractible open sets, the
principal $G$ bundle $\frakH$ over $\frakQ$ has a global section
$\sigma$, i.e., there is a smooth map $\sigma: \frakQ \to \frakH$
such that $j\circ \sigma=id$, and the restriction of $\sigma$ to the
unit space $\frakQ_0$ is the identity map.  Fix such a choice of
$\sigma$.

\begin{lemma}\label{lem:sigma}For a section $\sigma:\frakQ\to \frakH$, $s(\sigma(q))=s(q)$ and $t(\sigma(q))=t(q)$ for any $q\in \frakQ$.
\end{lemma}
\begin{proof}
As $j\circ \sigma=id$, we have $s(q)=s(j\circ \sigma(q))=j\circ
s(\sigma(q))=s(\sigma(q))$. The same argument works for the target
map $t$.
\end{proof}
We study the failure of $\sigma$ to be a groupoid morphism. If
$\sigma$ is a groupoid morphism, then the above extension is a
semi-direct product of the groupoid $\frakQ$ and the bundle $M\times
G$. In general, for two composable arrows $q_1,q_2\in \frakQ$, Lemma
\ref{lem:sigma} implies that $\sigma(q_1)$ and $\sigma(q_2)$ are
composable arrows in $\frakH$. However, $\sigma(q_1)\sigma(q_2)$
usually differs from $\sigma(q_1q_2)$. Generalizing
(\ref{eq:tau-def}), we define a map
$$\tau: \frakQ\times_{\frakQ_0} \frakQ \rightarrow \frakH^{(0)},\quad \tau(q_1,q_2):=\sigma(q_1)\sigma(q_2)\sigma(q_1q_2)^{-1},$$
where $\frakH^{(0)}$ is the loop space in $\frakH$ consisting of
arrows $h\in\frakH$ such that $s(h)=t(h)$. By Lemma \ref{lem:sigma},
$\tau(q_1,q_2)$ is a well-defined element in $\frakH$ such that
$s(\tau(q_1,q_2))=s(\sigma(q_1))=s(q_1)=t(\sigma(q_1q_2)^{-1})=t(\tau(q_1,q_2))$.
We also observe that
\[
j(\tau(q_1,
q_2))=j(\sigma(q_1)\sigma(q_2)\sigma(q_1q_2)^{-1})=j(\sigma(q_1))j(\sigma(q_2))j(\sigma(q_1q_2)^{-1})=s(q_1).
\]
This shows that $\tau(q_1, q_2)$ is in the kernel of $j$. Hence, we
can consider $\tau(q_1, q_2)$ as an element in $G$.

We compute the cocycle condition that $\tau$ satisfies. For
composable arrows $q_1,q_2,q_3\in \frakQ$,
\[
\begin{split}
\tau(q_1,q_2)\tau(q_1q_2, q_3)\sigma(q_1q_2q_3)=&(\sigma(q_1)\sigma(q_2))\sigma(q_3)
=\sigma(q_1)(\sigma(q_2)\sigma(q_3))
=\sigma(q_1)\tau(q_2,q_3)\sigma(q_2q_3)\\
=&\sigma(q_1)\tau(q_2,q_3)\sigma(q_1)^{-1}\tau(q_1,q_2q_3)\sigma(q_1q_2q_3).
\end{split}
\]
Note that $s(\tau(q_2,q_3))=t(\tau(q_2,q_3))=s(q_2)=t(q_1)$. So,
$
\Ad_{\sigma(q_1)}(\tau(q_2,q_3)):=\sigma(q_1)\tau(q_2,q_3)\sigma(q_1)^{-1}
$
is well-defined. In summary, we have the following equation, which
generalizes (\ref{eq:tau-cocycle}):
\begin{equation}\label{equation:gpd-tau-cocycle}
\tau(q_1,q_2)\tau(q_1q_2,
q_3)=\Ad_{\sigma(q_1)}(\tau(q_2,q_3))\tau(q_1,q_2q_3).
\end{equation}

With the above preparation, we have a new description of the
groupoid $\frakH$.  Define $G\rtimes_{\sigma,\tau}
\frakQ\rightrightarrows \frakQ_0$ to be the groupoid with the
following structures. The space of arrows of this groupoid is
$G\times \frakQ$. Given an arrow $(g,q)\in G\times \frakQ$, the
source map takes $(g,q)$ to $s(q)$ and the target map takes $(g,q)$
to $t(q)$. Given composable arrows $(g_1,q_1)$ and $(g_2,q_2)$,
their product is defined to be
\[
(g_1,q_1)(g_2,q_2)=(g_1\Ad_{\sigma(q_1)}(g_2)\tau(q_1,q_2), q_1q_2).
\]
Define the isomorphism $\alpha:\frakH\to
G\rtimes_{\sigma,\tau}\frakQ$ by $\alpha(h)=(h\sigma(j(h))^{-1},j(h))$.
The same argument  that proves Eq.  (\ref{eq:twisted-prod-gp}) now shows that
$\alpha$ is a groupoid isomorphism. In the following construction,
we will always work with the groupoid $G\rtimes_{\sigma,
\tau}\frakQ$.

Let $\cala$ be a $\frakQ$-sheaf of algebras over $\frakQ_0$. Pulling
back $\cala$ along the groupoid morphism $j$, we obtain an
$\frakH=G\rtimes_{\sigma, \tau}\frakQ$-sheaf\footnote{Since
$\frakQ_0=\frakH_0$, we use $\cala$ to denote the same sheaf on
$\frakH_0$ but equipped with an $\frakH$ action.} $\cala$ of
algebras over $\frakH_0=\frakQ_0$ such that every element in the
kernel of $j$ acts on $\cala$ trivially. We are interested in the
crossed product algebra $\cala\rtimes \frakH$ which, via the
isomorphism $\alpha$, is isomorphic to the crossed product algebra
$\cala\rtimes(G\rtimes_{\sigma,\tau} \frakQ)$. As the $G$ component
acts on $\cala$ trivially, the crossed product algebra $\cala\rtimes
(G\rtimes_{\sigma,\tau}\frakQ)$ is isomorphic to $(\cala\otimes
\complex G)\rtimes_{\sigma, \tau} \frakQ$.

Recall that in Sec.  \ref{sec:mackey}, for any class $[\rho]\in
\widehat{G}$ we fix a choice of an irreducible representation in the
class $[\rho]$ denoted by $\rho:G\to \End(V_\rho)$. Also, there is
an isomorphism $\complex G\cong \oplus_{[\rho]\in
\widehat{G}}\End(V_\rho)$ of algebras. We replace $\complex G$ by
$\oplus_{[\rho]\in \widehat{G}}\End(V_\rho)$ and study the $\frakQ$
action on the sheaf $\oplus_{[\rho]\in \widehat{G}}\End(V_\rho)$.

Let $q$ be an element in $\frakQ$. The adjoint action
$\Ad_{\sigma(q)}$ defines a group homomorphism on $G$. Accordingly,
$\Ad_{\sigma(q)}$ acts on $\widehat{G}$ as follows: for $[\rho]\in
\widehat{G}$, the class $q([\rho])\in \widehat{G}$ is the class of
the $G$ representation defined by
$g\mapsto\rho(\Ad_{\sigma(q)}(g))$. Again, by abuse of notation, the
chosen irreducible $G$-representation that represents the class
$q([\rho])$ will also be denoted by $q([\rho]): G\to
\End(V_{q([\rho])})$.
Since the representation $q([\rho]): G\to \End(V_{q([\rho])})$ is,
by definition, equivalent to the representation $G\to \End(V_\rho)$
defined by $g\mapsto \rho(\Ad_{\sigma(q)}(g))$, there is an
intertwining isomorphism $T^{[\rho]}_q:V_{\rho}\to V_{q([\rho])}$
such that $\rho(\Ad_{\sigma(q)}(g))={T^{[\rho]}_q}^{-1}\circ
q([\rho])(g)\circ T^{[\rho]}_q$. For a pair of composable arrows
$q_1, q_2\in \frakQ$, we have the following equation generalizing
(\ref{eq:dfn-c}):
\[
T^{q_1([\rho])}_{q_2}\circ
T^{[\rho]}_{q_1}=c^{[\rho]}(q_1,q_2)T^{[\rho]}_{q_1q_2}\rho(\tau(q_1,q_2))^{-1}.
\]

The action\footnote{We will write this action as a left action.} of
$\frakQ$ on $\widehat{G}$ defines a transformation groupoid
$\widehat{G}\rtimes \frakQ$. Since we have assumed that every level
of the nerve space of $\frakQ$ is a disjoint union of contractible
open sets, every level of the nerve space of the groupoid
$\widehat{G}\rtimes \frakQ$ is also a disjoint union of contractible
open sets. Generalizing Proposition \ref{prop:u(1)-cocycle}, we
obtain the following result whose proof is left to the reader.

\begin{proposition}\label{prop:u(1)-cocycle-gpd}The $U(1)$-valued
function $c:\widehat{G}\times \frakQ\times_{\frakQ_0}\frakQ\to
U(1), \ ([\rho],q_1,q_2)\mapsto c^{[\rho]}(q_1,q_2)$ defines a
groupoid cocycle on $\widehat{G}\rtimes \frakQ$ such that
$c^\bullet(q,-)=c^\bullet(-,q)=1$ if $q\in \frakQ_0$.
\end{proposition}

\begin{proposition}\label{prop:local-const}The cocycle $c$
introduced in Proposition \ref{prop:u(1)-cocycle-gpd} can be chosen
to be locally constant. Consequently the corresponding gerbe $c$
defined on the orbifold $\widehat{\Y}=[\widehat{G}\times \frakQ_0/\frakQ]$ is flat.
\end{proposition}
\begin{proof}
Note that given any $q\in \frakQ$, $\Ad_{\sigma(q)}(\cdot)$ defines
a group automorphism of $G$. This gives a smooth map
$\Ad:\frakQ\to \operatorname{Aut}(G)$. The set of group automorphisms
of $G$ is a finite set. Therefore, the map $\Ad:\frakQ\to
\operatorname{Aut}(G)$ must be locally constant. Similarly, as $G$
is finite, the continuous map $\tau:\frakQ\times_{\frakQ_0}\frakQ\to
G$ is also locally constant. Now following the construction in
Sec.  \ref{sec:mackey}, we can associate to an automorphism $\nu$
of $G$ a collection of isomorphisms $T^{[\rho]}_{\nu}:V_{\rho}\to
V_{\nu([\rho])}$. Hence, we can choose $T^{[\rho]}_{q}: V_{\rho}\to
V_{q([\rho])} $ to be locally constant in $q\in \frakQ$. This
implies that we can choose $c:\frakQ\times_{\frakQ_0}\frakQ\to U(1)$
to be a locally constant cocycle. By \cite[Proposition
3.26]{la-st-xu-1}, we conclude that the $U(1)$-gerbe defined by $c$
over $\widehat{\Y}$ is flat.
\end{proof}

With the above structure, the same arguments used in Proposition
\ref{prop:matrix-coeff-algebra} prove that the crossed product
algebra $(\cala\otimes \complex
G)\rtimes_{\sigma,\tau}\frakQ$ is isomorphic to the crossed product
algebra $(\cala\otimes
\oplus_{[\rho]\in\widehat{G}}\End(V_\rho))\rtimes_{T,c}\frakQ$,
whose product structure is defined by
\[
\begin{split}
&\big((a\otimes x)\circ (a'\otimes
x')\big)(q)
:=\sum_{q_1q_2=q, \rho}a(q_1)q_1(a'(q_2))\otimes
x(q_1)_{[\rho]}{T^{[\rho](q_1)}_{q_1}}^{-1}x'(q_2)_{q_1([\rho])}T^{[\rho]}_{q_1}
\rho(\tau(q_1,q_2)),
\end{split}
\]
where $x$ is a section of the sheaf $s^*(\oplus_{[\rho]\in\widehat{G}}\End(V_\rho))$ over $\frakQ$ (a trivial bundle over $\frakQ$), and
we use $x(q)_{[\rho]}$ to denote the $[\rho]$ component of $x$ at $q$.

Let ${\cala}$ be a $\frakQ$-sheaf of algebras over $\frakQ_0$. This
defines a $\widehat{G}\rtimes \frakQ$-sheaf $\widetilde{\cala}$ of
algebras on $\widehat{G}\times \frakQ_0$. Using the $U(1)$-valued
groupoid 2-cocycle $c$ on $\widehat{G}\rtimes \frakQ$
in Proposition \ref{prop:u(1)-cocycle-gpd}, we introduce the following definition. 
\begin{definition}Define a $c$-twisted
crossed product algebra $\widetilde{\cala}\rtimes_c
(\widehat{G}\rtimes \frakQ)$ by
\begin{equation}\label{eq:twisted-prod-gpd}
a_1\circ
a_2([\rho],q)=\sum_{q=q_1q_2}a_1([\rho],q_1)q_1\Big(a_2\big(q_1([\rho]),
q_2\big)\Big)c^{[\rho]}(q_1,q_2),
\end{equation}
for compactly supported smooth sections of the sheaf $\widetilde{\cala}$.
\end{definition}
In the following developments, the property that $c$ is a locally
constant 2-cocycle plays a crucial role. Otherwise, the product
defined in (\ref{eq:twisted-prod-gpd}) is not even associative.

\begin{theorem}\label{thm:global-mackey}
The ($c$-twisted) crossed product algebras $\cala\rtimes \frakH$  and  $\widetilde{\cala}\rtimes_c(\widehat{G}\rtimes \frakQ)$ are Morita equivalent.
\end{theorem}
\begin{remark}\label{rmk:morita-quasi-unital}
In general, the algebras $\cala\rtimes \frakH$ and
$\widetilde{\cala}\rtimes_c(\widehat{G}\rtimes \frakQ)$ are
quasi-unital algebras \cite[Appendix A]{pptt}. We use the methods of
\cite[Appendix A]{pptt} to work with Morita equivalence of
quasi-unital algebras.
\end{remark}
\begin{proof}
The proof is a straightforward generalization of the proofs of Theorems \ref{thm:local-mackey} and \ref{thm:main-global-quot}. We leave the details to the reader. 
\end{proof}
\begin{remark}
There is no doubt that
the corresponding $C^*$-algebra completions of $\cala\rtimes \frakH$
and $\widetilde{\cala}\rtimes_c(\widehat{G}\rtimes \frakQ)$ are also
Morita equivalent. Since we do not need this result in this paper,
we do not discuss its proof.
\end{remark}

Since Morita equivalent algebras have isomorphic categories of
modules, we deduce the following corollary from \cite[Theorem
A.12]{pptt}.
\begin{corollary}\label{cor:morita-gpd}
The algebras $\cala\rtimes \frakH$ and
$\widetilde{\cala}\rtimes_c(\widehat{G}\rtimes \frakQ)$ have
isomorphic categories of modules. In particular, the K-theory and
(Hochschild) cohomology groups of the two algebras are isomorphic.
\end{corollary}
\begin{remark} As explained in \cite[Theorem A.12]{pptt}, in order to
have an isomorphism between the Hochschild cohomologies of Morita
equivalent algebras, one needs the maps $\Xi: \calm\otimes_B \caln\to A$
and $\Theta: \caln\otimes_A \calm\to B$ to satisfy that for any $p,p'$ in $\calm$ and $q,q'$ in $\caln$, 
\[
q\Xi(p,q')=\Theta(q, p)q',\qquad p\Theta(q,p')=\Xi(p,q)p'.
\]
We can easily check these identities with the explicit formulas for $\Xi$ and
$\Theta$ in Remark \ref{rmk:morita-comp}.
\end{remark}
\subsection{Cohomology} \label{subsec:cohomology}
In this subsection, we study geometric consequences derived from the
Morita equivalence results, Theorem \ref{thm:global-mackey} and
Corollary \ref{cor:morita-gpd}.
All of the algebras considered in this section are Fr\'echet
algebras, and the bornologies on them are the ones associated with
the Fr\'echet topologies.

We start with the $K$-theory of the algebras $\cala\rtimes \frakH$
and $\widetilde{\cala}\rtimes_c(\widehat{G}\rtimes \frakQ)$. For
this purpose, we will consider $\cala$ to be the sheaf
$\calc^\infty$ of smooth functions on $\frakQ_0$, which is 
equipped with $\frakQ$ and $\frakH$ actions. The $K_\bullet$ group ($\bullet=0,1$) of the
algebra $\calc^\infty\rtimes \frakH$ is equal to the $K^\bullet$ group of
the orbifold $\Y$ associated with the groupoid $\frakH$. And the
$K_\bullet$ group of the twisted crossed product algebra
$\widetilde{\calc}^\infty\rtimes_c (\widehat{G}\rtimes \frakQ)$ is
equal to the twisted $K^\bullet_\frakc$-group \cite{tu-la-xu} of the gerbe
$\frakc$ on the orbifold $\widehat{\Y}$ associated with the groupoid
$\widehat{G}\rtimes \frakQ$.

\begin{proposition}\label{prop:k-theory}
The $K^\bullet$-group of the $G$-gerbe $\Y=[\frakH_0/\frakH]$ is
isomorphic to the $K^\bullet_\frakc$-group of the $U(1)$-gerbe $\frakc$ on
the orbifold $\widehat{\Y}=[\widehat{G}\times \frakQ_0/\frakQ]$, for $\bullet=0,1$.
\end{proposition}

Brylinski and Nistor \cite{br-ni} proved that the cyclic homology of
the groupoid algebra $\calc^\infty\rtimes \frakH$ is equal to the
cohomology (with compact support) of the inertia orbifold $I\Y$
associated with the orbifold $\Y$. Tu and Xu \cite{tu-xu} proved
that the cyclic homology of the twisted groupoid algebra
$\widetilde{\calc}^\infty\rtimes_c (\widehat{G}\rtimes \frakQ)$ is
equal to the $\frakc$-twisted cohomology (with compact support) of
the inertia orbifold $I\widehat{\Y}$ of $\widehat{\Y}$. By the Chern
character isomorphism between $K_\bullet(\calc^\infty\rtimes \frakH)\otimes \complex$ and $HP_\bullet(\calc^\infty\rtimes \frakH)$ (similarly between  $K_\bullet(\widetilde{\calc}^\infty\rtimes_c (\widehat{G}\rtimes \frakQ))\otimes \complex$ and $HP_\bullet(\widetilde{\calc}^\infty\rtimes_c (\widehat{G}\rtimes \frakQ))$), we have the following identification of
cohomologies.
\begin{proposition}\label{prop:periodic-homology}
\[
\bigoplus_{n\in \integers} H^{\bullet +2n}_{\text{\rm cpt}}(I\Y,
\complex)\cong HP_{\bullet}(\calc^\infty\rtimes \frakH)\cong
HP_\bullet(\widetilde{\calc}^\infty\rtimes_c (\widehat{G}\rtimes \frakQ))\cong
\bigoplus_{n\in \integers}H_{\text{\rm cpt}}^{\bullet+2n}(I\widehat{\Y},
\frakc, \complex).
\]
\end{proposition}

We want to better understand the isomorphisms in Proposition
\ref{prop:periodic-homology} in the symplectic case. In what
follows, we assume that there is a $\frakQ$-invariant symplectic
form on $\frakQ_0$. This happens when the base $\B$ is a symplectic
orbifold and the $G$-gerbe $\Y$ and its dual $\widehat{\Y}$ are
equipped with symplectic structures pulled back from the one on
$\B$. By choosing a $\frakQ$-invariant torsion free symplectic
connection on $\frakQ_0$, following \cite{pptt}, we consider the
sheaf of deformation quantization $\cala^{((\hbar))}$ on $\frakQ_0$.
$\cala^{((\hbar))}$ is also a $\frakQ$-sheaf of algebras.

It follows from Theorem \ref{thm:global-mackey} that the algebras
$\cala^{((\hbar))}\rtimes \frakH$ and $\widetilde{\cala}^{((\hbar))}\rtimes_c
(\widehat{G}\rtimes \frakQ)$ are also Morita equivalent.
Consequently, the Hochschild cohomology of the algebra
$\cala^{((\hbar))}\rtimes \frakH$ is isomorphic to the Hochschild
cohomology of the algebra $\widetilde{\cala}^{((\hbar))}\rtimes_c
(\widehat{G}\rtimes \frakQ)$.

The Hochschild cohomology of $\cala^{((\hbar))}\rtimes \frakH$ was
computed in our joint work \cite{pptt} with M. Pflaum and H. Posthuma, and
it is equal to the de Rham cohomology
$H^{\bullet-\ell}(I\Y)((\hbar))$ of the inertia orbifold $I\Y$ with
coefficients in $\complex ((\hbar))$ and a degree shifting $\ell$
given by the codimensions of the embeddings of components of $I\Y$
into $\Y$.  We have the following result for the Hochschild
cohomology of the algebra $\widetilde{\cala}^{((\hbar))}\times_c
(\widehat{G}\rtimes \frakQ)$.

\begin{theorem}\label{thm:hoch-coh-twisted-alg}
The Hochschild cohomology of the algebra
$\widetilde{\cala}^{((\hbar))}\rtimes_c (\widehat{G}\rtimes \frakQ)$
is equal to the $\frakc$-twisted de Rham cohomology of the orbifold
$\widehat{\Y}$ with coefficients in $\complex ((\hbar))$ and a
shifting $\ell$ defined by the codimensions of embeddings of
components of $I\widehat{\Y}$ into $\widehat{\Y}$. That is,
$$
HH^\bullet(\widetilde{\cala}^{((\hbar))}\rtimes_c (\widehat{G}\rtimes
\frakQ), \widetilde{\cala}^{((\hbar))}\rtimes_c (\widehat{G}\rtimes
\frakQ))\cong H^{\bullet-\ell}(I\widehat{\Y},\frakc)((\hbar)).
$$
\end{theorem}

The proof of Theorem \ref{thm:hoch-coh-twisted-alg} will be given in
Sec.  \ref{subsec:twisted-coh}. 
Theorem \ref{thm:global-mackey} about the Morita equivalence between $\cala^{((\hbar))}\rtimes \frakH$ and
$\widetilde{\cala}^{((\hbar))}\rtimes_c (\widehat{G}\rtimes
\frakQ)$ can
now be combined with Theorem \ref{thm:hoch-coh-twisted-alg} to yield
the following:

\begin{theorem}\label{thm:cohomology}
\hfill
\begin{enumerate}
\item \label{coh_isom_part1}
There are isomorphisms of cohomologies,
\[
\begin{split}
H^{\bullet-\ell}(I\Y)((\hbar))\cong &HH^\bullet(\cala^{((\hbar))}\rtimes \frakH, \cala^{((\hbar))}\rtimes \frakH)\\
\cong &HH^\bullet(\widetilde{\cala}^{((\hbar))}\rtimes_c
(\widehat{G}\rtimes \frakQ), \widetilde{\cala}^{((\hbar))}\rtimes_c
(\widehat{G}\rtimes \frakQ))\cong
H^{\bullet-\ell}(I\widehat{\Y},\frakc)((\hbar)).
\end{split}
\]
\item \label{coh_isom_part2}
Moreover, the above isomorphisms yield an isomorphism of graded
vector spaces,
\[
H^{\bullet-2\age}(I\Y)((\hbar))\cong
H^{\bullet-2\age}(I\widehat{\Y},\frakc)((\hbar)),
\]
where the vector spaces are equipped with the {\em age grading} as
defined in Sec.  \ref{subsec:orb}.
\end{enumerate}
\end{theorem}
\begin{proof}
The proof of Part (\ref{coh_isom_part1}) has been explained. Part (\ref{coh_isom_part2}) follows from Proposition \ref{prop:quasi-isomorphism-coh} (\ref{step4}). 

We explain the idea of the proof of Part (2). In the special case that $\Y$ is a trivial $G$-gerbe over $\B$, i.e. $\Y:=[\B/G]=\B\times BG$ where the $G$ action on $\B$ is trivial, the inertia orbifold $I\Y$ is the product orbifold $I\B\times IBG$, where $IBG$ is isomorphic to the quotient $[G/G]$ of the $G$ conjugation  action on $G$. Alternatively, $IBG=\cup_{\<g\> \subset G}BC_G(g)$, where the union is taken over conjugacy classes $\<g\>$ of $G$ and $C_G(g)\leq G$ is the centralizer of $g\in G$. The dual orbifold $\widehat{\Y}$ is the product orbifold $\B\times \widehat{G}$, on which the cocycle $c$ is trivial; and the inertia orbifold $I\widehat{\Y}$ is the product orbifold $I\B\times \widehat{G}$. Our proof of Part (\ref{coh_isom_part1}) yields a natural isomorphism
\[
I:\Omega^{\bullet-\ell}_{I\Y}((\hbar))=\Omega^{\bullet-\ell}_{I\B}((\hbar))\otimes_\complex Z (\complex G)\simeq \Omega^{\bullet-\ell}_{I\widehat{\Y}}((\hbar))=\Omega^{\bullet-\ell}_{I\B}((\hbar))\otimes_\complex C(\widehat{G}),
\]
where $Z(\complex G)$ is the center of $\complex G$. By explicit computations, we find that on $\Omega^\bullet_{I\B}((\hbar))$, the isomorphism $I$ is the identity operator; and on $Z(\complex G)$, $I$ is the isomorphism (with the same name) from $Z(\complex G)$ to $C(\widehat{G})$ introduced in Prop. \ref{prop:conjugacy} (See Remark \ref{rmk:I-iso-group}) for the special case that $Q$ is trivial and $H=G$.  As the map $I$ is constant on $I\B$, it is obviously compatible with the age filtrations on $I\Y$ and $I\widehat{\Y}$, which come from the same age filtration on $I\B$.  

For a general $G$-gerbe $\Y\to \B$, the quasi-isomorphism $I$ obtained in the proof of Part (\ref{coh_isom_part1}) is a generalization of the above form. We obtain an explicit formula of $I$ via quasi-isomorphisms on the corresponding Hochschild cohomology (pre)sheaves. As locally the quasi-isomorphism $I$ is compatible with the age filtrations by explicit computations, we deduce that globally $I$ is also compatible with the age filtrations. 

\end{proof}

\subsection{Quasi-isomorphism between Hochschild complexes}\label{subsec:twisted-coh}

In this subsection, we explain the proofs of Theorem
\ref{thm:hoch-coh-twisted-alg} and \ref{thm:cohomology}. We will
show that our method of computation actually gives a slightly stronger
result than what Theorem \ref{thm:hoch-coh-twisted-alg} states. More
precisely, the Hochschild cochain complexes of
$\cala^{((\hbar))}\rtimes \frakH$ and $\widetilde{\cala}^{((\hbar))}\times_c
(\widehat{G}\rtimes \frakQ)$ form presheaves over the orbifold
$\calb=[\frakQ_0/\frakQ]$. We will show the following
quasi-isomorphisms between (pre)sheaves and their properties.

\begin{proposition}\label{prop:quasi-isomorphism-coh}
\hfill
\begin{enumerate}
\item The (pre)sheaf of Hochschild cochain complexes of
$\cala^{((\hbar))}\rtimes \frakH$  is quasi-isomorphic to the sheaf
of de Rham differential forms on $I\Y$ as (pre)sheaves over $\B$.
\item The (pre)sheaf of Hochschild cochain
complexes of $\widetilde{\cala}^{((\hbar))}\rtimes_c
(\widehat{G}\rtimes \frakQ)$ is quasi-isomorphic to the sheaf of
$\frakc$-twisted de Rham differential forms on $I\widehat{\Y}$ as
(pre)sheaves over $\B$.
\item The Morita equivalence constructed in Theorem
\ref{thm:global-mackey} defines a quasi-isomorphism from the
(pre)sheaf of Hochschild cochain complexes of
$\cala^{((\hbar))}\rtimes \frakH$ to the (pre)sheaf of Hochschild
cochain complexes of $\widetilde{\cala}^{((\hbar))}\rtimes_c
(\widehat{G}\rtimes \frakQ)$ as (pre)sheaves over $\calb$.
\item\label{step4} The sheaf of de Rham complexes on $I\Y$ is
quasi-isomorphic to the sheaf of $\frakc$-twisted de Rham complexes on
$I\widehat{\Y}$ over $\calb$. The quasi-isomorphism between sheaves
of complexes is compatible with the filtration defined by the age
function $\age$ on $\calb$.
\end{enumerate}
\end{proposition}

\begin{proof}We divide the proof into four parts according to the four
corresponding statements (1)-(4). 

Our strategy is to generalize the methods in our joint work \cite{pptt} to $\widetilde{\cala}\rtimes_c(\widehat{G}\rtimes \frakQ)$. The main idea is to sheafify the computation of Hochschild cohomology to the sheaf cohomology of the Hochschild cohomology {\em presheaf} over the orbifold $\calb$ viewed as a topological space. As $\calb$ has a good cover, we can compute the corresponding sheaf cohomology via the \v{C}ech cohomology associated to a good cover. This method reduces our proof to local charts of $\Y$ and $\widehat{\Y}$ over $\calb$. On local charts, $\Y$ and $\widehat{\Y}$ can be represented by global quotient orbifolds as discussed in Sec. \ref{subsec:global-quo}. The computation of Hochschild cohomology of a global quotient orbifold is obtained in \cite{pptt}. We generalize this computation for $\widetilde{\cala}^{((\hbar))}\rtimes_c(\widehat{G}\rtimes \frakQ)$ using the trick of Tu and Xu \cite{tu-xu} in the presence of $c$. The compatibility with the age function is a corollary of the explicit formula (\ref{eq:I-alpha}) of the quasi-isomorphism between Hochschild cohomologies via Morita equivalence. \\

\noindent{\bf Part I:} This result is proved in \cite{pptt}. We explain the
proof in more detail as a preparation for the generalization in the
next Parts. See \cite{pptt} for a complete treatment. The main idea is to sheafify the computation of the Hochschild cohomology of $\cala^{((\hbar))}\rtimes \frakH$. 

Let $Bar_\bullet(\cala^{((\hbar))}\rtimes \frakH)$ be the bar
complex of the algebra $\cala^{((\hbar))}\rtimes \frakH$, and let
$(\cala^{((\hbar))}\rtimes \frakH )^e$ be the enveloping algebra of
$\cala^{((\hbar))}\rtimes \frakH$ as defined in Sec. 
\ref{subsec:hoch} (see \cite[Appendix]{pptt} for more details). The
Hochschild cochain complex of the algebra $\cala^{((\hbar))}\rtimes
\frakH$ is
\[
\operatorname{Hom}_{(\cala^{((\hbar))}\rtimes
\frakH)^e}(Bar_\bullet(\cala^{((\hbar))}\rtimes \frakH),
\cala^{((\hbar))}\rtimes \frakH).
\]

Let $\pi:\frakH_0\to \Y=[\frakH_0/\frakH]$ be the canonical
projection map. We define a sheaf $\cals^{((\hbar))}$ on $\Y$ by
\[
\cals^{((\hbar))}(U):=C^\infty((\pi\circ s)^{-1}(U))((\hbar)),
\]
where $s$ is the source map on $\frakH$. It is easy to check that the space 
$\cals^{((\hbar))}$ forms a sheaf, but the deformed convolution
product is not well-defined on $\cals^{((\hbar))}$ because functions
in $\cals^{((\hbar))}(U)$ may not have compact supports. This
suggests that we should consider the sheaf
$\cals_{\text{cf}}^{((\hbar))}$ defined by
\[
\cals_{\text{cf}}^{((\hbar))}(U):=\{f\in C^\infty((\pi\circ
s)^{-1}(U))((\hbar))|\text{supp}(f)\cap (\pi\circ s)^{-1}(K)\
\text{is compact for all compact\ }K\subset U\}.
\]
It is easy to check that the deformed convolution product is
well-defined on $\cals_{\text{cf}}^{((\hbar))}$ and turns
$\cals_{\text{cf}}^{((\hbar))}$ into a sheaf of algebras over $\Y$.

Define a (pre)sheaf $\calh_{\frakH}^\bullet$ on $\Y$ as follows: for any open subset $U$ of $\Y$,
\begin{equation}\label{eq:coh-sheaf-H}
\calh_{\frakH,\hbar}^\bullet(U):=\operatorname{Hom}\big((\cala^{((\hbar))}|_{U}\rtimes
\frakH |_U)^{\otimes \bullet}, \cals_{\text{cf}}^{((\hbar))}\big),
\end{equation}
In the above equation, $\operatorname{Hom}$ means bounded $\complex((\hbar))$-linear maps
with respect to the bornologies on the algebras. In \cite[Theorem
I]{pptt}, it is proved that the natural inclusion map
\[
\iota: C^\bullet(\cala^{((\hbar))}\rtimes \frakH,
\cala^{((\hbar))}\rtimes \frakH)\rightarrow
\calh_{\frakH,\hbar}^\bullet(\Y)
\]
is a quasi-isomorphism of cochain complexes compatible with the cup
products.

It is not difficult to check that the (pre)sheaf
$\calh_{\frakH,\hbar}^\bullet$ is fine. Therefore, we can use
the \v{C}ech double complex of the (pre)sheaf
$\calh_{\frakH,\hbar}^\bullet$ to compute the sheaf cohomology of
the (pre)sheaf $\calh_{\frakH, \hbar}^\bullet$, which is isomorphic
to the Hochschild cohomology of $\cala^{((\hbar))}\rtimes \frakH$ as
an algebra. In particular, when we choose a sufficiently fine
covering of $\Y$, the \v{C}ech double complex degenerates at $E_1$,
so we can use the \v{C}ech cohomology of the sheaf
$\calh_{\frakH,\hbar}$ to compute the Hochschild cohomology of
$\cala^{((\hbar))}\rtimes \frakH$. In \cite[Sec.  4]{pptt}, we
prove that on each open chart $U$ the complex
$\calh_{\frakH,\hbar}(U)$ is quasi-isomorphic to the de Rham complex
of $I\Y|_{U}$ over $U$, where $I\Y|_{U}$ is the inertia orbifold
associated to $U$. This proves that the \v{C}ech cohomology of the
sheaf $\calh_{\frakH,\hbar}^\bullet$ is equal to the de Rham
cohomology of
the inertia orbifold $I\Y$.\\

\noindent{\bf Part II:} We use the method developed in \cite{pptt}, as recalled in Part I, to compute the Hochschild cohomology of
$\widetilde{\cala}^{((\hbar))}\rtimes _c(\widehat{G}\rtimes
\frakQ)$. The new input is to compute the Hochschild cohomology on a local chart by generalizing the idea of Tu and Xu in \cite{tu-xu}.  We divide our computation into 3 steps.

\noindent{\bf Step 1.} We consider the orbifold $\widehat{\Y}$ presented by the groupoid
$\widehat{G}\rtimes \frakQ$. Let $\pi: \widehat{G}\times \frakQ_0\to
\widehat{\Y}$ be the canonical projection. Consider the sheaf
\[
\cals_{\text{cf},c}^{((\hbar))}(U):=\{f\in C^\infty((\pi\circ
s)^{-1}(U))((\hbar))|\text{supp}(f)\cap (\pi\circ s)^{-1}(K)\
\text{is compact for all compact\ }K\subset U\}.
\]
It is easy to check that $\cals_{\text{cf}, \frakc}^{((\hbar))}$
with the $\frakc$-twisted deformed convolution product, Eq. 
(\ref{eq:twisted-prod-gpd}), defines a sheaf of algebras over
$\widehat{\Y}$. Define the (pre)sheaf of Hochschild complex over
$\widehat{\Y}$ by
\begin{equation}\label{eq:coh-sheaf-Q-c}
\calh_{\widehat{G}\rtimes \frakQ, \hbar}^\bullet(U):=
\operatorname{Hom}\big( (\widetilde{\cala}^{((\hbar))}|_{U}\rtimes
(\widehat{G}\rtimes \frakQ|_U))^{\otimes \bullet},
\cals_{\text{cf},c}^{((\hbar))}(U)\big)
\end{equation}
for any open subset $U$ of $\widehat{\Y}$. The same arguments as
\cite[Sec.  4]{pptt} prove that the natural inclusion map
\[
\iota:C^\bullet\big(\widetilde{\cala}^{((\hbar))}\rtimes _c(\widehat{G}\rtimes
\frakQ), \widetilde{\cala}^{((\hbar))}\rtimes _c(\widehat{G}\rtimes
\frakQ)\big)\rightarrow \calh^\bullet_{\widehat{G}\rtimes \frakQ,
\hbar}(\widehat{\Y})
\]
is a quasi-isomorphism of differential graded algebras.

 With the exact same proof of \cite[Theorem II]{pptt}, we can show that the presheaf
$\calh^\bullet_{\widehat{G}\rtimes \frakQ,\hbar}$ is fine.
Therefore, we can use the \v{C}ech double complex to compute the
cohomology of $\calh^\bullet_{\widehat{G}\rtimes \frakQ,
\hbar}(\widehat{\Y})$. This allows us to reduce to local
computations of $\calh^\bullet_{\widehat{G}\rtimes \frakQ,
\hbar}(U)$ for a sufficiently small open subset $U$ of
$\widehat{\Y}$.

\noindent{\bf Step 2.} In what follows, we prove that the groupoid $\widehat{G}\rtimes \frakQ |_U$ is naturally Morita equivalent to a global quotient groupoid from a finite group $Q$ action on a symplectic manifold $P$. Consequently, the Hochschild cohomology $\calh^\bullet_{\widehat{G}\rtimes \frakQ,
\hbar}(U)$ is equal to the cohomology of the crossed product algebra $\widetilde{\cala}^{((\hbar))}_P\rtimes_c Q$. 

For $([\rho_0], x)\in \widehat{G}\rtimes \frakQ_0$, let
$\frakQ_{[\rho_0],x}$ be the isotropy group of the groupoid
$\widehat{G}\rtimes \frakQ$ at $([\rho_0],x)$, which consists of
arrows of the form $([\rho_0],q)$ with $q\in \frakQ_x:=\{q\in
\frakQ, t(q)=s(q)=x\}$ and $q([\rho_0])=[\rho_0]$. Choose a connected open neighborhood $W_{[\rho_0],q}$ for each $([\rho_0],q)$ in
$\widehat{G}\rtimes \frakQ_{[\rho_0],x}$ such that $s$ and $t$ map
$W_{[\rho_0],q}$ onto their images in $\widehat{G}\times \frakQ_0$
by diffeomorphisms. Define $M_{[\rho_0],x}$ to be the connected,
open component of $\bigcap_{q\in \frakQ_x,
q([\rho_0])=[\rho_0]}s(W_{[\rho_0],q})\cap t(W_{[\rho_0],q})$.
Define a $\frakQ_{[\rho_0],x}$ action on $M_{[\rho_0],x}$ by
\[
\frakQ_{[\rho_0],x}\times M_{[\rho_0],x}\rightarrow
M_{[\rho_0],x},\qquad (q,[\rho_0],x)\mapsto
t\big(s^{-1}|_{[\rho_0],q}([\rho_0],x)\big).
\]

In \cite[Theorem IIIb]{pptt}, we proved that the canonical
inclusion induces a weak equivalence from the transformation
groupoid $M_{[\rho_0],x}\rtimes \frakQ_{[\rho_0],x}$ to the groupoid
$\widehat{G}\rtimes \frakQ|_{U_{[\rho_0],x}}$, where
$U_{[\rho_0],x}$ is the image of $M_{[\rho_0],x}$ in $\widehat{\Y}$
under the canonical projection $\pi$. Now we restrict the $U(1)$-valued cocycle $c$ to the groupoid
$M_{[\rho_0],x}\rtimes \frakQ_{[\rho_0],x}$ to get a cocycle denoted by
$c_{[\rho_0],x}$. Following the proof of \cite[Theorem III.b]{pptt},
we can check that the twisted groupoid algebra
$\widetilde{\cala}^{((\hbar))}|_{M_{[\rho_0],x}}\rtimes_{c_{[\rho_0],x}}(M_{[\rho_0],x}\rtimes
\frakQ_{[\rho_0],x})$ is Morita equivalent to
$\widetilde{\cala}^{((\hbar))}|_{U_{[\rho_0],x}}\rtimes_c(\widehat{G}\rtimes
\frakQ |_{U_{[\rho_0],x}})$. Therefore, similar to \cite[Theorem
III.b]{pptt}, the canonical cochain map
\[
\calh^\bullet_{\widehat{G}\rtimes \frakQ,
\varsigma\hbar}(U_{[\rho_0],x})\rightarrow C^\bullet
(\widetilde{\cala}^{((\hbar))}|_{M_{[\rho_0],x}}\rtimes_{c_{[\rho_0],x}}(M_{[\rho_0],x}\rtimes
\frakQ_{[\rho_0],x}),\widetilde{\cala}^{((\hbar))}|_{M_{[\rho_0],x}}\rtimes_{c_{[\rho_0],x}}(M_{[\rho_0],x}\rtimes
\frakQ_{[\rho_0],x}))
\]
is a quasi-isomorphism. This enables us to localize the computation
of the Hochschild cohomology of the algebra
$\widetilde{\cala}^{((\hbar))}\rtimes_c(\widehat{G}\rtimes \frakQ)$
to
$\widetilde{\cala}^{((\hbar))}|_{M_{[\rho_0],x}}\rtimes_{c_{[\rho_0],x}}(M_{[\rho_0],x}\rtimes
\frakQ_{[\rho_0],x})$. Furthermore, observe that since $c_{[\rho_0],x}$ is
locally constant, the cocycle $c_{[\rho_0],x}$ is a lifting of a
$U(1)$-valued cocycle on $\frakQ_{[\rho_0],x}$ via the natural
projection $M_{[\rho_0],x}\rtimes \frakQ_{[\rho_0],x}\to
\frakQ_{[\rho_0],x}$. 

In summary, we have reduced our computation on $U$ to the following case.  Let $c$ be a $U(1)$-valued 2-cocycle on a finite group $Q$ that acts on a
symplectic manifold $P$ by symplectic diffeomorphisms. We want to
compute the Hochschild cohomology of the algebra
$\widetilde{\cala}^{((\hbar))}_P\rtimes_c Q$.

\noindent{\bf Step 3.} We compute the Hochschild cohomology of $\widetilde{\cala}^{((\hbar))}_P\rtimes_c Q$. We start with a  simplification of $c$ to a special type of $U(1)$-valued 2-cocycle. As $Q$ is a finite group, $[c]$ is a torsion element in $H^2(Q, U(1))$. This implies that there
exists a sufficiently large integer $m$ and a 1-cochain $\phi \in
C^1(Q, U(1))$ such that $\tilde{c}=c\delta(\phi)$ is a 2-cocycle in
$Z^2(Q, U(1))$ with $\delta$ as the group cohomology coboundary map
and $\tilde{c}^m=1$. It is straightforward to check that the algebra
$\widetilde{\cala}^{((\hbar))}_P\rtimes_c Q$ is isomorphic to
$\widetilde{\cala}^{((\hbar))}_P\rtimes_{\tilde{c}} Q$ via the map
$\Upsilon_{P,Q}: \widetilde{\cala}^{((\hbar))}_P\rtimes_c Q
\rightarrow \widetilde{\cala}^{((\hbar))}_P\rtimes_{\tilde{c}}Q$
defined by $\Upsilon(f)(x,q)=\phi(q)^{-1}f(x,q)$. Therefore, the
Hochschild cohomology groups of the two algebras are naturally isomorphic. This suggests that 
without loss of generality we can assume that $c$
takes value in $\integers/m\integers$.

As $c$ is a $\integers/m\integers$-valued 2-cocycle on $Q$, the
cocycle $c$ defines a central extension of $Q$:
\begin{equation}\label{eq:zm-ext}
1\rightarrow \integers/m\integers\rightarrow
\integers/m\integers\rtimes_c Q\rightarrow Q\rightarrow 1.
\end{equation}
If we take the natural section $s:Q\rightarrow \integers/m\integers\rtimes_c Q$ by mapping $q$ to $(1,q)$, the theory in Sec.  \ref{sec:gps_extenstion_mackey_machine} applies to study this group extension. In particular, the function $\tau$ defined by Eq. (\ref{eq:tau-def}) is identical to $c$. Since $\integers/m\integers$ is abelian, irreducible representations of $\integers/m\integers$ are one dimensional and therefore all the intertwiners $T$ can be chosen to be identity. Proposition \ref{prop:matrix-coeff-algebra} applied to the group algebra $\complex (\integers/m\integers \rtimes_c Q)$ gives 
\[
\complex( \integers/m\integers \rtimes_c Q)\cong \bigoplus_{k=0}^{m-1}C(Q, c^k). 
\]
If we look at the left action of $\integers/m\integers$ on $\complex(\integers/m\integers \rtimes_c Q)$, the subspace associated to the representation $\rho^k$ corresponds exactly to $C(Q, c^k)$, where $\rho$ is the natural embedding of $\integers/m\integers$ into $U(1)$.  This observation gives us a new way to look at the algebra $\widetilde{\cala}^{((\hbar))}_P\rtimes_c Q$.  Eq. (\ref{eq:zm-ext}) defines the following groupoid extension
\[
P\times \integers/m\integers\rightarrow P\rtimes
(\integers/m\integers\rtimes_c Q)\rightarrow P\rtimes Q,
\]
where $\integers/m\integers\rtimes_c Q$ acts on $P$ via the
canonical group homomorphism $\integers/m\integers\rtimes_c
Q\rightarrow Q$.

Consider the crossed product algebra
$\widetilde{\cala}^{((\hbar))}_P\rtimes
(\integers/m\integers\rtimes_c Q)$. Notice that
$\integers/m\integers$ acts on the algebra
$\widetilde{\cala}^{((\hbar))}_P\rtimes
(\integers/m\integers\rtimes_c Q)$ by algebra automorphism, and the
algebra $\widetilde{\cala}^{((\hbar))}_P\rtimes_c Q$ appears in
$\widetilde{\cala}^{((\hbar))}_P\rtimes
(\integers/m\integers\rtimes_c Q)$ as the subspace associated to the weight
$\rho$. More precisely, we can easily check
that the algebra $\widetilde{\cala}^{((\hbar))}_P\rtimes
(\integers/m\integers\rtimes_c Q)$ decomposes into a direct sum of
subalgebras
\begin{equation}\label{eq:decomp-algebra}
\widetilde{\cala}^{((\hbar))}_P\rtimes
(\integers/m\integers\rtimes_c
Q)=\bigoplus_{k=0}^{m-1}\widetilde{\cala}^{((\hbar))}_P\rtimes_{c^k}
Q,
\end{equation}
where $c^k$ is a $\integers/m\integers$-valued 2-cocycle on $Q$
defined by the $k$-th power of $c$. 

The decomposition (\ref{eq:decomp-algebra}) of the algebra
$\widetilde{\cala}^{((\hbar))}_P\rtimes
(\integers/m\integers\rtimes_c Q)$ naturally induces a decomposition
of the Hochschild cohomology
\[
HH^\bullet(\widetilde{\cala}^{((\hbar))}_P\rtimes
(\integers/m\integers\rtimes_c Q),
\widetilde{\cala}^{((\hbar))}_P\rtimes
(\integers/m\integers\rtimes_c
Q))
\cong \bigoplus_{k=0}^{m-1}
HH^\bullet(\widetilde{\cala}^{((\hbar))}_P\rtimes_{c^k} Q,
\widetilde{\cala}^{((\hbar))}_P\rtimes_{c^k} Q).
\]
The above decomposition is taken with respect to the
$\integers/m\integers$ action on the coefficient component of the
Hochschild cohomology
$HH^\bullet(\widetilde{\cala}^{((\hbar))}_P\rtimes
(\integers/m\integers\rtimes_c Q),
\widetilde{\cala}^{((\hbar))}_P\rtimes
(\integers/m\integers\rtimes_c Q))$. The Hochschild cohomology of
the algebra $\widetilde{\cala}^{((\hbar))}_P\rtimes_{c} Q$ is
identified as the component with weight $\rho$.

With the above preparation, we can apply the result of \cite[Theorem
4]{pptt} to compute the Hochschild cohomology of
$\widetilde{\cala}^{((\hbar))}_P\rtimes
(\integers/m\integers\rtimes_c Q)$ to be
\begin{equation}\label{eq:hoch-coh}
HH^\bullet(\widetilde{\cala}^{((\hbar))}_P\rtimes
(\integers/m\integers\rtimes_c Q),
\widetilde{\cala}^{((\hbar))}_P\rtimes
(\integers/m\integers\rtimes_c Q))\cong\bigoplus_{(\gamma)\in
\integers/m\integers\rtimes_c Q}
H^{\bullet-\ell}_{Z(\gamma)}(P^\gamma)((\hbar)),
\end{equation}
where $P^\gamma$ is the fixed point submanifold of $\gamma$, and $\ell$ is
a locally constant function measuring the codimension of $P^\gamma$
in $P$, and $Z(\gamma)$ is the centralizer group of $\gamma$ in
$\integers/m\integers\rtimes_c Q$. By chasing through the
quasi-isomorphisms constructed in \cite[Sec.  4]{pptt}, we
conclude that the above equation is actually compatible with the
$\integers/m\integers$ actions on both sides. The
right hand side of Eq.  (\ref{eq:hoch-coh}) is defined by the
cohomology of
$$\Omega^\bullet\Big(\big(P\rtimes (\integers/m\integers\rtimes_c
Q)\big)^{(0)}\Big)^{\integers/m\integers\rtimes_c Q},$$ which has a
natural $\integers/m\integers$ action as defined in \cite[Lemma
4.13]{tu-xu}. The component with weight $\rho$ of the left side of Eq.  (\ref{eq:hoch-coh}) is
$HH^\bullet(\widetilde{\cala}^{((\hbar))}_P\rtimes_c Q,
\widetilde{\cala}^{((\hbar))}_P\rtimes_c Q)$. And by \cite[Eq. 
(25)]{tu-xu}, the component with weight $\rho$ of the cohomology of
$\Omega^\bullet\Big(\big(P\rtimes (\integers/m\integers\rtimes_c
Q)\big)^{(0)}\Big)^{\integers/m\integers\rtimes_c Q}$ is the
cohomology of $(\amalg_\gamma P^\gamma)/Q$ with values in the inner
local system $\call_c$ defined by $c$. We briefly explain the construction of $\call_c$. 

\begin{definition}\label{dfn:L-c}
Let $L$ be a line bundle on $P\rtimes Q$ defined by $L:=\complex
\times _{\integers/m\integers} \big(P\rtimes
(\integers/m\integers\rtimes_c Q)\big)$. Define $\call_c$ to be the restriction of $L$ on $(P\rtimes
Q)^{(0)}$. As $c$ is locally constant, $\call_c$ is equipped with a canonical flat connection $\nabla$ determined by $c$ as is explained in  \cite[Proposition
3.9]{tu-xu}. Hence $\call_c$ is a flat complex line bundle on the inertia orbifold $I\widehat{\Y}$ satisfying the conditions of an inner local system (See Definition \ref{dfn:inner-local}).
\end{definition}
Accordingly, taking the components of weight
$\rho$ on both sides of Eq. 
(\ref{eq:hoch-coh}), we conclude that
\[
HH^\bullet(\widetilde{\cala}^{((\hbar))}_P\rtimes_c Q,
\widetilde{\cala}^{((\hbar))}_P\rtimes_c Q)\cong
H^{\bullet-\ell}(I\widehat{\Y}, c)((\hbar)),
\]
where $I\widehat{\Y}$ is the inertia orbifold associated to the
orbifold $\widehat{\Y}=[P/Q]$.

In conclusion, we have shown that locally the Hochschild cochain
complex of $\widetilde{\cala}^{((\hbar))}_P\rtimes_c Q$ is
quasi-isomorphic to the cochain complex
$(\Omega^{\bullet-\ell}((P\rtimes Q)^{(0)}, {\call_c} )^{Q}, \nabla)$ via
a sequence $I_{c}$ of natural quasi-isomorphisms constructed in
\cite[Sec.  4]{pptt}. We can easily apply this sequence
together with its intermediate objects to write down a natural
sequence, $I_{c_{[\rho_0],x}}$, of cochain maps between the
Hochschild cochain complex
\[
C^\bullet
(\widetilde{\cala}^{((\hbar))}|_{M_{[\rho_0],x}}\rtimes_{c_{[\rho_0],x}}(M_{[\rho_0],x}\rtimes
\frakQ_{[\rho_0],x}),\widetilde{\cala}^{((\hbar))}|_{M_{[\rho_0],x}}\rtimes_{c_{[\rho_0],x}}(M_{[\rho_0],x}\rtimes
\frakQ_{[\rho_0],x}))
\]
and
\[
(\Omega^{\bullet-\ell}((M_{[\rho_0],x}\rtimes
\frakQ_{[\rho_0],x})^{(0)}, {\call}_{c_{[\rho_0],x}}
)^{\frakQ_{[\rho_0],x}}, \nabla^{c_{[\rho_0],x}}).
\]
Furthermore, we can check that
$I_{c_{[\rho_0],x}}$ is a sequence of natural quasi-isomorphisms and
glues together to define a sequence of quasi-isomorphisms between
the presheaf of Hochschild complexes $\calh^\bullet_{\widehat{G}\rtimes_c\frakQ, \hbar}$ and
the sheaf of the twisted de Rham complexes
$(\Omega^{\bullet-\ell}_{I\widehat{\Y}}(\call_c)((\hbar)), \nabla)$ as
(pre)sheaves of algebras over $\widehat{\Y}$. 
\begin{remark}
In this step, we have chosen to work locally with a
$\integers/m\integers$-valued 2-cocycle on $Q$ to obtain the
Hochschild cohomology of $\widetilde{\cala}^{((\hbar))}_P\rtimes Q$
via the trick of passing to the central extension
$\integers/m\integers\rtimes_c Q$ as in \cite{tu-xu}. One can take a
more direct path by repeating the computations in \cite[Sec. 
4]{pptt}. Because of the property that $c$ is locally constant,
methods in \cite[Sec.  4]{pptt} naturally generalize to compute
the Hochschild cohomology of the algebra
$\widetilde{\cala}^{((\hbar))}\rtimes_c (\widehat{G}\rtimes
\frakQ)$ together with the cup product. 
\end{remark}

\noindent{\bf{Part III:}} This is essentially a corollary of the Morita Equivalence Theorem \ref{thm:global-mackey}. Let $\calb$ be the orbifold defined by the
quotient $[\frakQ_0/\frakQ]$. As a topological space, it is easy to
see that $\calb$ is the same as $\Y$. Hence, the (pre)sheaf
$\calh^\bullet_{\frakH, \hbar}$ of differential graded algebras is
also a presheaf over $\calb$. Similarly, the orbifold $\widehat{\Y}$
is a fibration over $\calb$ with finite fibers. Hence, the
push-forward of the (pre)sheaf $\calh^\bullet_{\widehat{G}\rtimes
\frakQ, \hbar}$ defines a (pre)sheaf over $\calb$. In this part, we
want to show that the Morita equivalence bimodules $\calm$ and
$\caln$ between $\cala^{((\hbar))}\rtimes \frakH$ and
$\widetilde{\cala}^{((\hbar))}\rtimes_c(\widehat{G}\rtimes \frakQ)$
define quasi-isomorphisms between $\calh^\bullet_{\frakH, \hbar}$
and $\calh^\bullet_{\widehat{G}\rtimes \frakQ, \hbar}$ as
(pre)sheaves over $\calb$.

It is straightforward to see that the Morita equivalence bimodules
$\calm$ and $\caln$ constructed in the proof of Theorem
\ref{thm:global-mackey} (and Proposition \ref{prop:global-quotient}) are compatible with localization
to an open set $U$ of $\calb$. More precisely, let
$\lambda_\Y:\Y\rightarrow \calb$ and $\lambda_{\widehat{\Y}}:
\widehat{\Y}\rightarrow \B$ be the canonical projections, and let
$\pi_{\Y}$ (respectively, $\pi_{\widehat{\Y}}$) be the projection
from $\frakH_0\rightarrow \Y$ (respectively $\widehat{G}\times
\frakQ_0\rightarrow \widehat{\Y})$. It is not difficult to see that
the restrictions of $\calm$ and $\caln$ to $(\lambda_\Y\circ
\pi_\Y\circ s)^{-1}(U)$ define Morita equivalence bimodules between
$\cala^{((\hbar))}\rtimes \frakH|_{\lambda_\Y^{-1}(U)}$ and
$\widetilde{\cala}^{((\hbar))}\rtimes_c
(\widehat{G}\rtimes\frakQ)|_{\lambda_{\widehat{\Y}}^{-1}(U)}$.
Consequently, the Morita equivalence bimodules, together with the
maps $\Xi$ and $\Theta$ introduced in the proof of Theorem
\ref{thm:global-mackey} (and Proposition \ref{prop:global-quotient}), induce quasi-isomorphisms between the
cochain complexes $\calh^\bullet_{\frakH,\hbar}$ and
$\calh^\bullet_{\widehat{G}\rtimes\frakQ,\hbar}$ as (pre)sheaves of
differential graded algebras over $\calb$.

We conclude from Part I and II that the sheaf of de Rham
differential forms on $I\Y$ is quasi-isomorphic to the sheaf of
$\frakc$-twisted de Rham differential forms on $I\widehat{\Y}$
viewed as sheaves over $\calb$:
\begin{equation}\label{eq:quasi-local}
I:(\Omega^{\bullet-\ell}_{I\Y}((\hbar)), d)\simeq
(\Omega^{\bullet-\ell}_{I\widehat{\Y}}(\call_c)((\hbar)), \nabla).
\end{equation}
\noindent{\bf Part IV:} We are left to show that the quasi-isomorphism (\ref{eq:quasi-local}) obtained in Part III is compatible with the filtration defined by
the age function.  Note that the age filtration (respectively, the
codimension filtration $\ell$) on $\Omega^\bullet_{I\Y}((\hbar))$ is
determined by the age filtration (respectively, the codimension
filtration) on $I\calb$ via the map $\tilde{\lambda}_\Y:I\Y\to
I\calb$ since $G$ acts on $\frakQ_0$ trivially. Similarly, the age
filtration (respectively, the codimension filtration $\ell$) on
$I\widehat{\Y}$ is determined by the age filtration (respectively,
the codimension filtration) on $I\calb$ via the fibration
$\tilde{\lambda}_{\widehat{\Y}}:I\widehat{\Y}\to I{\calb}$. By finding an explicit formula of $I$, we prove
that $I$ is compatible with the fibrations on $I\Y$
and $I\widehat{\Y}$ over $\calb$ and conclude that $I$ is compatible
with the age filtration. 

We can decompose the quasi-isomorphism $I$ into the following
sequences of isomorphisms:
\[
\begin{split}
H^{\bullet-\ell}(I\Y)((\hbar))&\stackrel{I_1}{\cong} HH^\bullet(\cala^{((\hbar))}\rtimes \frakH, \cala^{((\hbar))}\rtimes \frakH)\\
&\stackrel{I_2}{\cong}HH^\bullet(\widetilde{\cala}^{((\hbar))}\rtimes_c (\widehat{G}\rtimes \frakQ), \widetilde{\cala}^{((\hbar))}\rtimes_c (\widehat{G}\rtimes \frakQ))\stackrel{I_3}{\cong}H^{\bullet-\ell}(I\widehat{\Y},\frakc)((\hbar)).
\end{split}
\]

The isomorphism $I_1$ between $HH^\bullet(\cala^{((\hbar))}\rtimes
\frakH, \cala^{((\hbar))}\rtimes \frakH)$ and
$H^{\bullet-\ell}(I{\Y})((\hbar))$ is constructed in \cite{pptt} and
reviewed in Part I. The map $I_2$ from $HH^{\bullet}(\calc\rtimes
\frakH)$ to $HH^\bullet(\calc\rtimes (\widehat{G}\rtimes_c \frakQ))$
is a standard construction for Morita invariance of Hochschild
cohomology from Theorem \ref{thm:global-mackey} as explained in Part
III. Construction of an explicit quasi-isomorphism is explained in
\cite[Theorem A.12]{pptt} and the isomorphism $I_3$ between
$H^{\bullet-\ell}(I\widehat{\Y},\frakc)((\hbar))$ and
$HH^\bullet(\widetilde{\cala}^{((\hbar))}\rtimes_c
(\widehat{G}\rtimes \frakQ), \widetilde{\cala}^{((\hbar))}\rtimes_c
(\widehat{G}\rtimes \frakQ))$ is explained in Part II. In order to
have an explicit formula for $I=I_3\circ I_2\circ I_1$, we need to
write down the formulas for $I_i$, for $i=1,2,3$.

Following Part III, we consider the presheaves
$\calh^\bullet_{\mathfrak {H}, \hbar}$ (defined by Eq. 
(\ref{eq:coh-sheaf-H})) and $\calh^\bullet_{\widehat{G}\rtimes
\mathfrak{Q}, \hbar}$ (defined by Eq.  (\ref{eq:coh-sheaf-Q-c}))
over the orbifold $\calb=[\mathfrak {Q}_0/\mathfrak {Q}]$. We prove
that the isomorphism $I$ is realized by a sequence of
quasi-isomorphisms of presheaves over $\calb$:
\begin{eqnarray*}
\cali_1: \calh^\bullet_{\mathfrak {H}, \hbar}\to
(\Omega^{\bullet-\ell}_{I{\Y}}((\hbar)),d),\ \cali_2: \calh^\bullet_{\mathfrak {H}, \hbar}\to
\calh^\bullet_{\widehat{G}\rtimes \mathfrak {Q}, \hbar},\ 
\cali_3: \calh^\bullet_{\widehat{G}\rtimes \mathfrak {Q}, \hbar}
\to (\Omega^{\bullet-\ell}_{I\widehat{\Y}}(\call_c)((\hbar)), \nabla).
\end{eqnarray*}

Since $\cali_j$, $j=1,2,3$, are quasi-isomorphisms of presheaves, it
suffices to look at their restrictions on a sufficiently small open
set $U$ of $\calb$. On a sufficiently small open set $U$, the
$G$-gerbe $\Y$ over $U$ can be represented by the groupoid extension
$V\times G\rightarrow V\rtimes H\rightarrow V\rtimes Q,$
so that the open set $U$ is identified with the quotient $[V/Q]$.
This is a special case of the global quotient considered in Sec. 
\ref{subsec:global-quo}.

The map $\cali_1$ is explained in \cite[Sec.  4]{pptt}. If
$\alpha$ is a cocycle in $\calh^\bullet_{\mathfrak {H}, \hbar}(U)$,
then $\cali_1(\alpha)$ is the image of the restriction of $\alpha$
on $I\Y$ in the cohomology of the double complex
$\calc_{\X}^{\bullet, \bullet}$ studied in \cite[Proposition
4.8]{pptt}. As is explained in Part II, Eq. 
(\ref{eq:hoch-coh}), the map $\cali_3$ is a
$\integers/m\integers$-equivariant version of the map $\cali_1$ if
we represent $c$ by a cocycle in $\integers/m\integers$.

The map $\cali_2$ is a standard map from Morita equivalence as is
explained in \cite[Theorem A.12]{pptt}. We follow the construction
in Sec.  \ref{subsec:center}. Let $\xi^\rho_i$ be a basis of
$V_\rho$ and $\eta^i_\rho$ the dual basis of $V_\rho ^*$. If
$\varphi$ is a Hochschild cocycle on $\cala^{((\hbar))}_U\rtimes H$,
$I_2(\varphi)$ is a Hochschild cocycle on
$(\cala^{((\hbar))}_U\otimes C(\widehat{G}))\rtimes_c Q$ defined by
\begin{eqnarray*}
I_2(\varphi)(a_1,\cdots, a_k):=\sum_{i_0,\cdots, i_k, \rho_0, \cdots,
\rho_k}\Big(\frac{1}{\dim(V_\rho)}\Big)^k\Theta\Big(\eta^{i_0}_{\rho_0},
\varphi\big(\Xi(\xi^{\rho_0}_{i_0},a_1\eta_{\rho_1}^{i_1} ), \cdots,
\Xi(\xi^{\rho_{k-1}}_{i_{k-1}},
a_k\eta_{\rho_k}^{i_k})\big)\xi^{\rho_k}_{i_k}\Big),
\end{eqnarray*}
where $a_1, \cdots, a_k$ are elements of
$(\cala^{((\hbar))}_U\otimes C(\widehat{G}))\rtimes_c Q$. Here,
$C(\widehat{G})$ is the algebra of functions on the finite set
$\widehat{G}$. The algebra $(\cala^{((\hbar))}_U\otimes
C(\widehat{G}))\rtimes_c Q$ can be identified with the twisted crossed
product algebra $\widetilde{A}^{((\hbar))}_{U\times \widehat{G}}\rtimes_c Q$ of deformation quantization on $U\times \widehat{G}$
by the $Q$-action.

Notice that $\xi^{\rho_{p-1}}_{i_{p-1}}$ (respectively,
$\eta_{\rho_p}^{i_p}$) are supported only on the identity component
of $(\widetilde{\cala}_U\otimes \calv_G)\rtimes_c Q$ (respectively,
$(\widetilde{\cala}_U\otimes \calv^*_G)\rtimes_c Q$). If we choose
$a_1, \cdots, a_k$ to be functions supported on the identity
component of $\widetilde{\cala}_U\rtimes_c Q$, then by the definition of $\Xi$, $\Xi(\xi^{\rho_{p-1}}_{i_{p-1}},
a_p\eta_{\rho_p}^{i_p})$ vanishes if $\rho_p\ne \rho_{p-1}$ and is
equal to $a_p([\rho])\Xi(\xi^{\rho_{p-1}}_{i_{p-1}},
\eta_{\rho_p}^{i_p})$ for any $p=1, \cdots, k$. As 
elements in $G$ act trivially on $U$ and the relation
$\eta^\rho_i(\xi_i^\rho)=1$, we see that when $a_1, \cdots, a_k$ are
supported on the identity component, $I_2(\varphi)(a_1, \cdots,
a_k)$ is equal to
\begin{eqnarray}\label{eq:mor-isom}
&&\sum_{i, \rho}\frac{1}{\dim(V_\rho)}\Theta\left(\eta^i_\rho,
\varphi(a_1([\rho]), \cdots, a_k([\rho]))\xi_i^{\rho}\right)\cr &=&
\sum_{i,\rho, g,q; q([\rho])=[\rho]}
\frac{1}{\dim(V_\rho)}\varphi(a_1([\rho]), \cdots, a_k([\rho]))
(g,q)\tr\big(\rho(g))(T^{[\rho]}_q)^{-1}\big).
\end{eqnarray}

Observe that the maps $\cali_1$ and $\cali_3$ only count the information of the cocycle $\varphi$ along the space $U$ direction,
which is completely determined by the part $\varphi(a_1([\rho]),
\cdots, a_k([\rho]))(g,q)$ in Eq.  (\ref{eq:mor-isom}).
Hence it follows that the map $I=\cali_3\circ \cali_2\circ
(\cali_1^{-1})$ on a cohomology class $\alpha$ of
$\Omega^{\bullet-\ell}_{I{\Y}}(U)((\hbar))$ can be expressed by
\begin{equation}\label{eq:I-iso}
I(\alpha)=\sum_{g,q,\rho;q([\rho])=[\rho]}\frac{1}{\dim(V_\rho)}\alpha(g,q)\tr(\rho(g)(T^{[\rho]}_q)^{-1})([\rho],q),
\end{equation}
where we write $\alpha=\sum_{g,q}\alpha(g,q)\in
\Omega^*(U^{g,q})$, which is invariant under the conjugation action of
$H$.

Now extending the expression of $I$ in Eq.  (\ref{eq:I-iso}) to
the whole orbifold, we have the following isomorphism. We represent
a cohomology class $\alpha$ on $I{\Y}$ as
$\alpha=\sum_{g,q:s(q)=t(q)}\alpha(g,q)$ such that $\alpha$ is a
closed differential form on $\mathfrak {H}^{(0)}=\{(g,q)\in
\mathfrak {H}: s(q)=t(q)\}\subset \mathfrak {H}$ that is invariant
under the conjugation action of $\mathfrak {H}$ on $\mathfrak
{H}^{(0)}$. Then $I(\alpha)$ can be written as a differential form
supported on $\mathfrak{Q}':=\{([\rho], q): q([\rho])=[\rho],
s(q)=t(q)\}\subset \widehat{G}\rtimes_c\mathfrak {Q}$ that is
invariant under the $c$-twisted conjugation action by
$\widehat{G}\rtimes_c\mathfrak {Q}$,
\begin{equation}\label{eq:I-alpha}
I(\alpha)([\rho],q)=\sum_{g}\frac{1}{\dim(V_\rho)}\alpha(g,q)\tr(\rho(g){T^{[\rho]}_q}^{-1}), 
\end{equation}
which is a full generalization of the map $I$ in Propositions
\ref{prop:conjugacy}-\ref{prop:trace} when $\calb=BQ$.

Similar to Proposition \ref{prop:conjugacy}, the above expression
(\ref{eq:I-iso}) shows that, locally, the quasi-isomorphism $I$ is
compatible with respect to the conjugacy classes of the group $Q_U$.
More explicitly, let $U^{\<q\>}$ be the component of the inertia
orbifold $IU\subset I\calb$ defined by the conjugacy class
$\<q\>\subset Q_U$ and let 
$\Omega^{\bullet-\ell}_{I\Y}(U)((\hbar))|_{U^{\<q\>}}$ (and
$\Omega^{\bullet-\ell}_{I\widehat{\Y}}(U)((\hbar))|_{U^{\<q\>}}$) be
the space of differential forms on $I\Y$ (and on $I\widehat{\Y}$)
supported on $\tilde{\lambda}^{-1}_{\Y}(U^{\<q\>})$ (and
$\tilde{\lambda}^{-1}_{\widehat{\Y}}(U^{\<q\>})$). The isomorphism
$I$ in (\ref{eq:I-iso}) defines a quasi-isomorphism
\[
I|_{U^{\<q\>}}:\big(\Omega^{\bullet-\ell}_{I\Y}(U)((\hbar))|_{U^{\<q\>}},
d\big)\longrightarrow
\big(\Omega^{\bullet-\ell}_{I\widehat{\Y}}(\call_c)((\hbar))|_{U^{\<q\>}},
\nabla\big).
\]
As the codimension functions $\ell$ on $I\Y$ and $I\widehat{\Y}$ are both determined by the corresponding function on $I\calb$, $\ell_{I\Y}=\ell_{I\widehat{\Y}}$. Noticing that the quasi-isomorphism $I|_{U^{\<q\>}}$ is a local map over $I\calb$, we conclude from the equality of the codimension functions that 
\begin{equation}\label{eq:quasi-i-no-shift}
I|_{U^{\<q\>}}:\big(\Omega^{\bullet}_{I\Y}(U)((\hbar))|_{U^{\<q\>}},
d\big)\longrightarrow
\big(\Omega^{\bullet}_{I\widehat{\Y}}(\call_c)((\hbar))|_{U^{\<q\>}},
\nabla\big)
\end{equation}
is also a quasi-isomorphism by looking at $I$ over each component of $I\calb$. 

As both the age functions on $I\Y$ and $I\widehat{\Y}$ are determined by their corresponding ones
on $I{\calb}$, we conclude from Eq. (\ref{eq:quasi-i-no-shift}) that the quasi-isomorphism
(\ref{eq:I-iso}) over each component of $I\calb$ also induces a quasi-isomorphism
\[
(\Omega^{\bullet-2\age}_{I\Y}(U)((\hbar))|_{U^{\<q\>}},
d)\longrightarrow
(\Omega^{\bullet-2\age}_{I\widehat{\Y}}(\call_c)((\hbar))|_{U^{\<q\>}},
\nabla).
\]

Noting that all the constructions above are canonical as morphisms
of (pre)sheaves over $\calb$, we can globalize the above local arguments via \v{C}ech arguments to the following
equality of vector spaces:
\[
H^\bullet_{\text{{\rm CR}}}(\Y,\complex((\hbar)))\overset{\text{def}}{=}
H^{\bullet-2\age}(I{\Y},\complex((\hbar)))\cong
H^{\bullet-2\age}(I\widehat{\Y},\frakc, \complex((\hbar))) \overset{\text{def}}{=}
H^\bullet_{orb}(\widehat{\Y}, \frakc, \complex((\hbar))).
\]
\end{proof}
\begin{remark}\label{rmk:cpx_coeff}
It can be seen from the explicit description of the
quasi-isomorphism  $I$ in (\ref{eq:I-alpha}) that $I$ induces an
isomorphism between cohomologies with $\complex$ coefficients,
namely, $H^\bullet_{\text{{\rm CR}}}(\Y,\complex)\simeq
H^\bullet_{orb}(\widehat{\Y}, \frakc, \complex).$
\end{remark}
\begin{remark}
Our proof of Theorem \ref{thm:cohomology} crucially relies on the assumption that the orbifold is symplectic\footnote{At least because of the use of deformation quantization.} and uses a compatible almost complex structure to define the Chen-Ruan orbifold cohomology age function.  On the other hand, we observe that our formula (\ref{eq:I-alpha}) for the quasi-isomorphism $I$ is well defined without a choice of a symplectic form. This suggests that the morphism $I$ can be introduced for a $G$-gerbe on a general orbifold with an almost complex structure. It is interesting to check whether $I$ is still a quasi-isomorphism in this generality. We will come back to this question in the near future. 
\end{remark}

\section{Results on Chen-Ruan orbifold cohomology rings}\label{sec:ring}
\subsection{Review of Chen-Ruan cohomology}\label{subsec:cr}
This section contains a summary of the Chen-Ruan orbifold cohomology
ring \cite{cr} and twisted orbifold cohomology rings \cite{ru1}. Throughout this section, let $\X$ be a compact almost complex
orbifold. The {\em inertia orbifold} of $\X$ is the orbifold $I{\X}$
whose points are pairs $(x, (g))$ where $x\in \X$ and $(g)\subset
\text{Iso}(x)$ is a conjugacy class of the isotropy subgroup of the
point $x\in \X$. There is a natural map
$\mathbf{p}_\X: I{\X}\to \X, \ (x, (g))\mapsto x.$ The inertia orbifold $I{\X}$ is, in general, disconnected; let
$I{\X}=\coprod_{i\in \sI}\X_i$ be the decomposition into connected
components, where $\sI$ is an index set.

\begin{definition}\label{Horb_groups}
The {\em Chen-Ruan orbifold cohomology groups} of $\X$ are defined to
be the cohomology groups of the inertia orbifold, $$H_{{\rm CR}}^\bullet(\X,
\com):=H^\bullet(I\X, \com)=\oplus_{i\in \sI}H^\bullet(\X_i, \com).$$ The
grading used in the Chen-Ruan cohomology is the {\em
age-grading}: the degree of a class $\alpha\in H^p(\X_i,\com)$ in $H_{{\rm CR}}^\bullet(\X,\com)$ is $p+2\text{age}\,(\X_i)$, where $\text{age}\,(\X_i)$ is the value of the age function on the component $\X_i$.
\end{definition}

\begin{remark}
Here, we use the field $\com$ of complex numbers as coefficients for
the cohomology groups. Other fields can (and will) be used as
coefficients.
\end{remark}

The Chen-Ruan cohomology $H_{{\rm CR}}^\bullet(\X,\com)$ is equipped with
a non-degenerate pairing called the {\em orbifold Poincar\'e
pairing}, which is constructed as follows. There is an isomorphism
$I_\X:I{\X}\to I{\X}$ given by $(x, (g))\mapsto (x, (g^{-1}))$.
Clearly, the composition $I_\X\circ I_\X$ is the identity map. A
component $\X_i$ is mapped isomorphically to a component we denote
by $\X_{i^I}$.
\begin{definition}\label{orb_pairing}
The orbifold Poincar\'e pairing $(-,-)_{orb}^\X$ is defined as
follows. Define 
$$
(\alpha,
\beta)_{orb}^\X:=\int_{\X_i}\alpha\cup I_\X^*\beta,\qquad\text{for}\ \alpha\in H^\bullet(\X_i,\com),\ \beta\in
H^\bullet(\X_{i^I},\com).
$$ 
The pairing
$(-,-)_{orb}^\X$ is extended to the whole $H^\bullet(I\X,\com)$ by
requiring bilinearity.
\end{definition}

\begin{remark}
The orbifold Poincar\'e pairing differs from the Poincar\'e pairing
on the cohomology $H^\bullet(I\X,\com)$ because of the factor $I^*_\X$.
\end{remark}

The Chen-Ruan cohomology $H_{{\rm CR}}^\bullet(\X,\com)$ is also equipped
with a new product structure called the {\em Chen-Ruan orbifold cup
product}. We briefly recall its construction. Let $\X_{(2)}$ be the
{\em $2$-multi-sector} of $\X$. It can be understood as the space\footnote{ The orbifold $\X_{(2)}$ also shows up in the geometric description of the cup product on the Hochschild cohomology of the corresponding groupoid algebra \cite{pptt}.}
whose points are $(x, (g,h))$, where $x\in \X$, $g,h\in
\text{Iso}(x)$, and $$(g,h):=\{(kgk^{-1}, khk^{-1})|k\in
\text{Iso}(x)\}\subset \text{Iso}(x) \times \text{Iso}(x)$$ is a
{\em biconjugacy class} of $\text{Iso}(x)$. There are three {\em
evaluation maps}:
\begin{eqnarray*}
&ev_{\X, 1}:& \X_{(2)}\to I{\X}, \ (x, (g,h))\mapsto (x, (g));\\ 
&ev_{\X, 2}:& \X_{(2)}\to I{\X}, \ (x, (g,h))\mapsto (x, (h));\\
&ev_{\X, 3}:& \X_{(2)}\to I{\X}, \ (x, (g,h))\mapsto (x,
((gh)^{-1})).
\end{eqnarray*}
There is also another natural map that forgets the biconjugacy
class: 
$$\mathbf{p}_{\X_{(2)}}:\X_{(2)}\to \X,\quad (x,(g,h))\mapsto x.$$

The key ingredient in the Chen-Ruan cup product is the so-called
{\em obstruction bundle}  $Ob_\X\to \X_{(2)}.$ The original
construction of $Ob_\X$ in \cite{cr} involves the moduli spaces of
genus-$0$ degree-$0$ orbifold stable maps to $\X$, which is
complicated. The construction has since been simplified. We present
two descriptions.

\begin{numcon}[see \cite{hepworth}, Theorem 2]\label{construction_Hepworth}
Let $(x,(g,h))\in \X_{(2)}$. The group $\< g,h\>$ generated by $g,h$
acts on the fiber $\mathbf{p}_{\X_{(2)}}^*T\X|_{(x,(g,h))}=T_x\X$,
yielding a decomposition
$\mathbf{p}_{\X_{(2)}}^*T\X|_{(x,(g,h))}=\bigoplus V_i\otimes T_i$,
where the $V_i$'s are irreducible representations of $\<g,h\>$.
Varying the point $(x,(g,h))\in \X_{(2)}$ in a connected component
of $\X_{(2)}$ yields a global decomposition of vector bundles over
this component:
$$\mathbf{p}_{\X_{(2)}}^*T\X=\bigoplus V_i\otimes T_i,$$ where
$T_i$'s are vector bundles over this component. Then the restriction
of $Ob_\X$ to this component of $\X_{(2)}$ is equal to
$\oplus_i T_i^{\oplus h_i},$ where
\begin{equation}\label{defn:h_i}
h_i=\text{age}_{V_i}(g)+\text{age}_{V_i}(h)-\text{age}_{V_i}(gh)+\text{dim}\,
V_i^{\<g,h\>}-\text{dim}\,V_i^{gh}.
\end{equation}
The terms $V_i^{\<g,h\>}$ and $V_i^{gh}$ are the vector subspaces
fixed by $\<g,h\>$ and $gh$ respectively. As stated in
\cite{hepworth}, it can be shown that the numbers $h_i$ are
nonnegative integers, so the above equation makes sense.
\end{numcon}

\begin{numcon}[see \cite{jkk}]\label{construction_jkk}
Consider the component $\X_i$. At  any point $(x,(g))\in \X_i$ the
group $\<g\>$ acts on the fiber
$\mathbf{p}_\X^*T\X|_{(x,(g))}=T_x\X$. The decomposition into
$g$-eigenspaces can be globalized:
$$(\mathbf{p}_\X^*T\X)|_{\X_i}=\bigoplus_k W^\X_{i,k}$$ where
$W^\X_{i,k}$ is the eigen-bundle on which $g$ acts with eigenvalue
$\exp(\frac{2\pi\sqrt{-1}k}{r})$ and $r$ is the order of $g$. Put
$$S_i^\X:=\bigoplus_{k\neq 0}\frac{k}{r}W^\X_{i,k}.$$ This is an
element in the $K$-theory of $\X_i$.

Consider the locus $\X_{i_1, i_2}:=ev_{\X,1}^{-1}(\X_{i_1})\cap
ev_{\X, 2}^{-1}(\X_{i_2})$. Then the image $ev_{\X, 3}(\X_{i_1,
i_2})$ is contained in a component of $I\X$ denoted by $\X_{i_3}$.
The $K$-theory class of the restriction of $Ob_\X$ to $\X_{i_1,
i_2}$ is given by
\begin{equation}
Ob_\X|_{\X_{i_1, i_2}}=T\X_{i_1, i_2}\ominus
\mathbf{p}_{\X_{(2)}}^*T\X|_{\X_{i_1,
i_2}}\oplus\bigoplus_{j=1}^3ev_{\X, j}^*S^\X_{i_j}.
\end{equation}
\end{numcon}

We now come to the definition of the Chen-Ruan cup product. Let
$e(-)$ denote the Euler class.
\begin{definition}\label{orb_cup_product}
For classes $\alpha_1, \alpha_2, \alpha_3\in H^\bullet(I{\X},\com)$,
define
\begin{equation*}
\<\alpha_1,\alpha_2,\alpha_3\>^\X:=\int_{\X_{(2)}}ev_{\X,1}^*\alpha_1\cup
ev_{\X,2}^*\alpha_2\cup ev_{\X,3}^*\alpha_3\cup e(Ob_\X).
\end{equation*}
Fix an additive basis $\{\phi_i\}$ of $H^\bullet(I{\X},\com)$
such that each element $\phi_i$ is homogeneous and is supported on
one connected component of $I{\X}$. Let
$\phi^i:=PD(\phi_i)$ be the class dual to $\phi_i$ under the
orbifold Poincar\'e pairing $(-,-)_{orb}^\X$. The Chen-Ruan cup
product of $\alpha_1, \alpha_2\in H_{{\rm CR}}^\bullet(\X, \com)$ is defined as
\begin{equation}
\alpha_1\star_{orb}\alpha_2:=\sum_i \<\alpha_1, \alpha_2,
\phi_i\>^\X \phi^i.
\end{equation}
\end{definition}
It follows from the definition that $(\alpha_1\star_{orb}\alpha_2,
\phi_i)_{orb}^\X=\<\alpha_1,\alpha_2,\phi_i\>^\X$. The Chen-Ruan
orbifold cohomology $H_{{\rm CR}}^\bullet(\X, \com)$ equipped with the above
structures is a graded (super-)commutative $\complex$-algebra.

Let $\X$ be a compact almost complex orbifold as above.  We recall the construction of the {\em twisted orbifold
cohomology} \cite{ru1}.
Let $c$ be a flat $U(1)$-gerbe on $\X$. It
follows from the discussion in \cite{pry} that given such a
$U(1)$-gerbe $c$, one can naturally construct an inner local system $\sL_c$.  
\begin{definition}\label{dfn:inner-local} (\cite[Definition 3.1]{ru1}) 
An {\em inner local system} on $\X$ is  a flat complex line bundle $\sL\to I{\X}$ satisfying the following properties:
\begin{enumerate}
\item The restriction of $\sL$ to the identity component $\X=\{(x,(id))|x\in \X\}\subset I\X$ is a trivial line bundle with a fixed trivialization.
\item There is a nondegenerate pairing $I^*_{\X}\sL\otimes \sL\to \complex$ on $I\X$.
\item There is a multiplication $\theta: ev_{\X,1}^*\sL\otimes ev_{\X,2}^* \sL \to I_{\X}^*(ev_{\X,3}^*\sL)$ on $\X_{(2)}$:
\item The multiplication $\theta$ is associative on $\X_{(3)}$, whose points are $(x,(g_1, g_2, g_3))$ with $x\in \X$ and 
\[
(g_1, g_2, g_3):=\{(kg_1k^{-1}, kg_2k^{-1}, kg_3k^{-1})|k\in {\rm Iso}(x)\}\subset {\rm Iso}(x)\times  {\rm Iso}(x)\times  {\rm Iso}(x).
\] 
\end{enumerate}
\end{definition}

The $c$-twisted orbifold cohomology groups are defined to be $H_{orb}^\bullet(\X, c, \com):=H^\bullet(I{\X}, \sL_c, \com)$,  the cohomology groups of $I{\X}$ with coefficients in the inner local system $\sL_c$. The groups $H_{orb}^\bullet(\X, c, \com)$ are
equipped with the age-grading that is defined in the same way as
Definition \ref{Horb_groups}.

The definition of a $c$-twisted orbifold Poincar\'e pairing for
$H_{orb}^\bullet(\X, c, \com)$ is exactly parallel to Definition
\ref{orb_pairing}: for $\alpha\in H^\bullet(\X_i,\sL_c)$ and
$\beta\in H^\bullet(\X_{i^I},\sL_c)$, define $$(\alpha,
\beta)_{orb,c}^\X:=\int_{\X_i}\alpha\cup I_\X^*\beta.$$ 
In the above integral, we have applied the nondegenerate pairing 
in Definition \ref{dfn:inner-local} (2). Hence, the pairing takes value in $\complex$. 

The $c$-twisted orbifold cohomology $H_{orb}^\bullet(\X, c, \com)$ also
carries an orbifold cup product $\star_c$ defined using the
obstruction bundle $Ob_\X$. The definition is parallel to Definition
\ref{orb_cup_product}. Define
\begin{equation*}
\<\alpha_1,\alpha_2,\alpha_3\>_c^\X:=\int_{\X_{(2)}}ev_{\X,1}^*\alpha_1\cup
ev_{\X,2}^*\alpha_2\cup ev_{\X,3}^*\alpha_3\cup e(Ob_\X),\quad \alpha_1, \alpha_2, \alpha_3\in
H^\bullet(I{\X}, \sL_c, \com). 
\end{equation*}
The above integral is well defined because of the nondegenerate pairing and multiplication $\theta$ in Definition \ref{dfn:inner-local} (2) and (3). For simplicity we omit $\theta$ in the notation.

Fix an additive basis $\{\phi_i\}$ of $H^\bullet(I{\X}, \sL_c,
\com)$ such that each element $\phi_i$ is homogeneous and is
supported on one connected component of $I{\X}$. Let
$\phi^i:=PD(\phi_i)$ be the class dual to $\phi_i$ under the pairing
$(-,-)_{orb,c}^\X$. The $c$-twisted orbifold cup product of
$\alpha_1, \alpha_2\in H_{orb}^\bullet(\X, c, \com)$ is defined as
\begin{equation}
\alpha_1\star_{c}\alpha_2:=\sum_i \<\alpha_1, \alpha_2,
\phi_i\>_c^\X \phi^i.
\end{equation}

It follows from the definition that $(\alpha_1\star_{orb}\alpha_2,
\phi_i)_{orb, c}^\X=\<\alpha_1,\alpha_2,\phi_i\>_c^\X$. And the $c$-twisted orbifold cohomology $H_{orb}^\bullet(\X,c, \com)$ equipped
with the above structures is a graded (super) commutative $\complex$-algebra.
\subsection{Chen-Ruan cohomology of \'etale gerbes} \label{subsec:cohom-gerbe}
Let $\B$ be a compact connected almost complex orbifold, and $G$ a finite group. Let $\Y\to \B$ be a $G$-gerbe over $\B$ and
$\widehat{\Y}\to \B$ its dual, equipped with the flat
$U(1)$-gerbe $c$. Denote by $\sL_c$ the inner local system
associated with $c$.

The dual $\widehat{\Y}$ is not necessarily connected; let
$\widehat{\Y}=\coprod_{\bi\in \mathbf{I}}\widehat{\Y}_\bi$ be the
decomposition into connected components. Here, $\mathbf{I}$ is an
index set. Let $c_\bi$ be the $U(1)$-gerbe on $\widehat{\Y}_\bi$
obtained by restricting $c$ to $\widehat{\Y}_\bi$, and let
$\sL_{c_\bi}$ be the inner local system associated to $c_\bi$. The $c$-twisted orbifold cohomology
$H_{orb}^\bullet(\widehat{\Y},c,\com)$ is a direct sum
\begin{equation}\label{decomp_of_Horb_Yhat}
H_{orb}^\bullet(\widehat{\Y},c,\com)=\bigoplus_{\bi\in \mathbf{I}}
H_{orb}^\bullet(\widehat{\Y}_\bi,c_\bi,\com).
\end{equation}
The ring structures and pairings are compatible with this
decomposition: if $\alpha_1, \alpha_2, \alpha_3\in
H_{orb}^\bullet(\widehat{\Y},c,\com)$ are decomposed with respect to
(\ref{decomp_of_Horb_Yhat}) as
\begin{equation}\label{decomp_of_Horb_classes_Yhat}
\alpha_1=\oplus_{\bi}\alpha_{1\bi}, \alpha_2=\oplus_\bi
\alpha_{2\bi}, \alpha_3=\oplus_\bi \alpha_{3\bi}, \quad
\alpha_{1\bi}, \alpha_{2\bi}, \alpha_{3\bi}\in
H_{orb}^\bullet(\widehat{\Y}_\bi, c_\bi,\com),
\end{equation}
then
\begin{equation}\label{decomp_of_invs_Yhat}
\<\alpha_1, \alpha_2, \alpha_3\>_c^{\widehat{\Y}}=\sum_\bi \<\alpha_{1\bi}, \alpha_{2\bi}, \alpha_{3\bi}\>_{c_\bi}^{\widehat{\Y}_\bi},\ 
\alpha_1\star_c \alpha_2=\oplus_\bi (\alpha_{1\bi}\star_{c_\bi}\alpha_{2\bi}),\ 
(\alpha_1,\alpha_2)_{orb,c}^{\widehat{\Y}}=\sum_{\bi}
(\alpha_{1\bi},\alpha_{2\bi})_{orb,c_{\bi}}^{\widehat{\Y}_\bi}.
\end{equation}

Suppose that $\B$ is symplectic and is equipped with a compatible
almost complex structure. We equip $\Y$ and $\widehat{\Y}$ with the
symplectic and compatible almost complex structures induced from
those on $\B$. In Sec.  \ref{subsec:cohomology}, we have
constructed an additive isomorphism
\begin{equation*}
\Psi:= I^{-1}: H_{orb}^\bullet(\widehat{\Y},c, \com)\to
H_{{\rm CR}}^\bullet(\Y, \com),
\end{equation*}
which respects the age-gradings. See Proposition
\ref{prop:periodic-homology}, Theorem \ref{thm:cohomology}, and Remark \ref{rmk:cpx_coeff}. In
this section, we analyze the compatibility of $\Psi$ with the cup
products $\star_{orb}$ and $\star_c$. The main result, Theorem
\ref{thm:iso-coh-ring}, states that $\Psi$ is in fact a ring
isomorphism.

We begin by comparing the obstruction bundles. Consider the natural
maps between $2$-multi-sectors $$\pi_\Y: \Y_{(2)} \to \B_{(2)},
\quad \pi_{\widehat{\Y}}:\widehat{\Y}_{(2)}\to \B_{(2)},$$ which are
induced from the maps $\Y\to \B$ and $\widehat{\Y}\to \B$.

\begin{proposition}\label{prop:obstruction}
The following relations hold among obstruction bundles.
\begin{equation}\label{comparison_Ob_bundle_Y}
Ob_\Y=\pi_\Y^*Ob_\B,
\end{equation}
\begin{equation}\label{comparison_Ob_bundle_Yhat}
Ob_{\widehat{\Y}}=\pi_{\widehat{\Y}}^*Ob_\B.
\end{equation}
\end{proposition}

\begin{proof}
This is an easy application of the constructions of obstruction
bundles. We prove (\ref{comparison_Ob_bundle_Y}). The proof of (\ref{comparison_Ob_bundle_Yhat}) is almost identical and is left to the reader. Denote by $f: \Y\to \B$ the structure map of the $G$-gerbe. The map
$f$ is \'etale and $T\Y=f^*T\B$. Consider the situation of
Construction \ref{construction_Hepworth}. Let $(x,(h_1,h_2))\in
\Y_{(2)}$ and let $(x, (q_1,q_2))\in \B_{(2)}$ be its image under
$\pi_{\Y}$. Then we have a decomposition
$\mathbf{p}_{\B_{(2)}}^*T\B|_{(x,(q_1,q_2))}=\bigoplus V_i\otimes T_i$
of vector bundles, where the $V_i$'s are irreducible representations
of the group $\<q_1,q_2\>$ and the $T_i$'s are vector bundles over
the component of $\B_{(2)}$ containing $(x, (q_1,q_2))$. Since
$f:\Y\to \B$ is a $G$-gerbe, the induced group homomorphism
$\<h_1,h_2\>\to \<q_1, q_2\>$, given by $h_1\mapsto q_1, h_2\mapsto
q_2$, is surjective. Hence, the $V_i$'s can be viewed as irreducible
representations of $\<h_1,h_2\>$. Also, the group $G$ acts trivially
on fibers of $T\Y$. Also note that $\mathbf{p}_{\B_{(2)}}\circ
\pi_{\Y}=f\circ \mathbf{p}_{\Y_{(2)}}$. By this discussion, it
follows that the decomposition of the bundle
$\mathbf{p}_{\Y_{(2)}}^*T\Y$ as in Construction
\ref{construction_Hepworth} is given by
$\mathbf{p}_{\Y_{(2)}}^*T\Y= \bigoplus V_i\otimes \pi_{\Y}^*T_i.$ 
According to Construction \ref{construction_Hepworth}, we have
$Ob_\Y=\bigoplus_i \pi_{\Y}^*T_i^{\oplus h_i(\Y)},$ and 
$Ob_\B=\bigoplus_i T_i^{\oplus h_i(\B)}.$ The numbers $h_i(\Y)$ and
$h_i(\B)$ are given by (\ref{defn:h_i}). It follows easily from the
previous discussion that $h_i(\Y)=h_i(\B)$. This proves
(\ref{comparison_Ob_bundle_Y}).
\end{proof}
We remark that it is possible to use Construction \ref{construction_jkk} to
prove (\ref{comparison_Ob_bundle_Y}) and (\ref{comparison_Ob_bundle_Yhat}) at the level of $K$-theory classes and we will leave the details to the reader. 

Recall that locally the base $\B$ can be presented as a quotient
$[M/Q]$, the gerbe $\Y$ can be presented as $[M/H]$, where the
finite groups $H, Q$ fit into an exact sequence $1\to G\to H\to
Q\to 1.$ Over $[M/Q]$, the dual $\widehat{\Y}$ is presented as the
quotient $[(M\times \widehat{G})/Q]$. We define a function $M\times
\widehat{G}\to \mathbb{Q}$ by $(x, [\rho])\mapsto
(\text{dim}V_\rho/|G|)^2$. Since representations belonging to the
same $Q$ orbit all have the same dimension, this function descends
to a function $w: \widehat{\Y}\to \mathbb{Q}.$ Clearly, $w$ only
depends on connected components of $\widehat{\Y}$. Let
$w(\widehat{\Y}_\bi)$ denote the value of $w$ on the component
$\widehat{\Y}_\bi$.

\begin{theorem}
Let $\alpha_1, \alpha_2, \alpha_2\in H_{orb}^\bullet(\widehat{\Y},c,
\com)$ be classes whose decompositions with respect to
(\ref{decomp_of_Horb_Yhat}) are given by
(\ref{decomp_of_Horb_classes_Yhat}). Then we have
\begin{equation}\label{comparison_3pt_invariants}
\<\Psi(\alpha_1),\Psi(\alpha_2), \Psi(\alpha_3)\>^{\Y}=\sum_\bi
w(\widehat{\Y}_\bi)\<\alpha_{1\bi},\alpha_{2\bi},\alpha_{3\bi}\>_{c_\bi}^{\widehat{\Y}_\bi}.
\end{equation}
\end{theorem}
\begin{proof}
By the definitions of the symbols $\<-,-,-\>^\Y$ and
$\<-,-,-\>_{c_\bi}^{\widehat{\Y}_\bi}$, we see that
(\ref{comparison_3pt_invariants}) can be written as
\begin{equation}\label{3pt_inv_1}
\begin{split}
&\int_{\Y_{(2)}}ev_{\Y, 1}^*\Psi(\alpha_1)\cup ev_{\Y, 2}^*\Psi(\alpha_2)\cup ev_{\Y, 3}^*\Psi(\alpha_3)\cup e(Ob_\Y)\\
 =&\sum_\bi w(\widehat{\Y}_\bi)\int_{\widehat{\Y}_{\bi (2)}}ev_{\widehat{\Y}_\bi, 1}^*\alpha_{1\bi}\cup ev_{\widehat{\Y}_\bi, 2}^*\alpha_{2\bi}\cup ev_{\widehat{\Y}_\bi, 3}^*\alpha_{3\bi}\cup e(Ob_{\widehat{\Y}_\bi}).
 \end{split}
\end{equation}
Let $\pi_{\widehat{\Y}_\bi}: \widehat{\Y}_{\bi (2)}\to \B_{(2)}$
denote the natural map induced by the map $\widehat{\Y}_\bi\to \B$.
Then (\ref{comparison_Ob_bundle_Yhat}) implies that
$Ob_{\widehat{\Y}_\bi}= \pi_{\widehat{\Y}_\bi}^*Ob_\B$. Together
with (\ref{comparison_Ob_bundle_Y}), it implies that
(\ref{3pt_inv_1}) can be rewritten as
\begin{equation}\label{3pt_inv_2}
\begin{split}
&\int_{\B_{(2)}}\pi_{\Y*}\left(ev_{\Y, 1}^*\Psi(\alpha_1)\cup ev_{\Y, 2}^*\Psi(\alpha_2)\cup ev_{\Y, 3}^*\Psi(\alpha_3)\right)\cup e(Ob_\B)\\
 =&\sum_\bi w(\widehat{\Y}_\bi)\int_{\B_{(2)}} \pi_{\widehat{\Y}_\bi*}\left(ev_{\widehat{\Y}_\bi, 1}^*\alpha_{1\bi}\cup ev_{\widehat{\Y}_\bi, 2}^*\alpha_{2\bi}\cup ev_{\widehat{\Y}_\bi, 3}^*\alpha_{3\bi}\right)\cup e(Ob_\B).
 \end{split}
\end{equation}
Thus (\ref{comparison_3pt_invariants}) follows from the following
equality of classes in $H^\bullet(\B_{(2)},\com)$:
\begin{equation}\label{3pt_inv_3}
\pi_{\Y*}\left(ev_{\Y, 1}^*\Psi(\alpha_1)\cup ev_{\Y, 2}^*\Psi(\alpha_2)\cup ev_{\Y, 3}^*\Psi(\alpha_3)\right) =\sum_\bi w(\widehat{\Y}_\bi) \pi_{\widehat{\Y}_\bi*}\left(ev_{\widehat{\Y}_\bi, 1}^*\alpha_{1\bi}\cup ev_{\widehat{\Y}_\bi, 2}^*\alpha_{2\bi}\cup ev_{\widehat{\Y}_\bi, 3}^*\alpha_{3\bi}\right).
\end{equation}
The proof of (\ref{3pt_inv_3}) is a little technical and lengthy, and
will be given in Sec.  \ref{subsec:pairing}.
\end{proof}

\begin{corollary}\label{cor:poin-pairing}
Let $\alpha_1, \alpha_2\in H_{orb}^\bullet(\widehat{\Y}, c,\com)$ be
classes whose decompositions with respect to
(\ref{decomp_of_Horb_Yhat}) are given by
(\ref{decomp_of_Horb_classes_Yhat}). Then
\begin{equation}\label{comparison_pairings}
(\Psi(\alpha_1), \Psi(\alpha_2))_{orb}^\Y=\sum_\bi
w(\widehat{\Y}_\bi)(\alpha_{1\bi},
\alpha_{2\bi})_{orb,c_\bi}^{\widehat{\Y}_\bi}.
\end{equation}
\end{corollary}
\begin{proof}
Let $\alpha_3=\oplus_\bi 1_\bi\in \oplus_\bi
H^0(\widehat{\Y}_\bi,c_\bi,\com)$, where $1_\bi\in
H^0(\widehat{\Y}_\bi,c_\bi,\com) $ is the identity element with
respect to the product $\star_{c_\bi}$. Then by the description of
$I=\Psi^{-1}$ in (\ref{eq:I-iso}), we have that $\Psi(\alpha_3)=1\in
H^0(\Y, \com)$ is the class Poincar\'e dual to the fundamental
class, and the corollary follows from
(\ref{comparison_3pt_invariants}) since $\<\Psi(\alpha_1),
\Psi(\alpha_2), 1\>^\Y=(\Psi(\alpha_1), \Psi(\alpha_2))_{orb}^\Y$ and
$\<\alpha_{1\bi},\alpha_{2\bi},1_\bi\>_{c_\bi}^{\widehat{\Y}_\bi}=(\alpha_{1\bi},
\alpha_{2\bi})_{orb,c_\bi}^{\widehat{\Y}_\bi}$.
\end{proof}

\begin{theorem}\label{thm:iso-coh-ring}
The map $\Psi: H_{orb}^\bullet(\widehat{\Y}, c,\com)\to H_{{\rm CR}}^\bullet(\Y,
\com)$ is an isomorphism of rings.
\end{theorem}
\begin{proof}
Since $\Psi$ is an additive isomorphism, it suffices to prove that
for $\alpha_1, \alpha_2\in H_{orb}^\bullet(\widehat{\Y}, c,\com)$, we have
\begin{equation*}
\Psi(\alpha_1)\star_{orb}\Psi(\alpha_2)=\Psi(\alpha_1\star_c
\alpha_2).
\end{equation*}
By the non-degeneracy of the pairing $(-,-)_{orb}^\Y$, this is
equivalent to
\begin{equation}\label{ring_isom1}
(\Psi(\alpha_1)\star_{orb}\Psi(\alpha_2),
\Psi(\alpha_3))_{orb}^\Y=(\Psi(\alpha_1\star_c \alpha_2),
\Psi(\alpha_3))_{orb}^\Y,\quad \text{for any }\alpha_3\in
H_{orb}^\bullet(\widehat{\Y}, c,\com).
\end{equation}
The left-hand side of (\ref{ring_isom1}) is $\<
\Psi(\alpha_1),\Psi(\alpha_2), \Psi(\alpha_3)\>^\Y$. Suppose that
the classes $\alpha_1, \alpha_2, \alpha_3$ are decomposed as in
(\ref{decomp_of_Horb_classes_Yhat}). Then by
(\ref{comparison_pairings}), the right-hand side of
(\ref{ring_isom1}) is equal to
\begin{equation*}
(\Psi(\alpha_1\star_c \alpha_2), \Psi(\alpha_3))_{orb}^\Y
=\sum_\bi w(\widehat{\Y}_\bi)(\alpha_{1\bi}\star_{c_\bi}\alpha_{2\bi}, \alpha_{3\bi})_{orb, c_\bi}^{\widehat{\Y}_\bi}= \sum_\bi w(\widehat{\Y}_\bi)\<\alpha_{1\bi}, \alpha_{2\bi},
\alpha_{3\bi}\>^{\widehat{\Y}_\bi}.
\end{equation*}
Thus, (\ref{ring_isom1}) is equivalent to
(\ref{comparison_3pt_invariants}). The theorem follows.
\end{proof}

\subsection{Proof of (\ref{3pt_inv_3})}\label{subsec:pairing}
Recall that in the proof of Theorem
\ref{prop:quasi-isomorphism-coh}, we constructed an explicit formula
(\ref{eq:I-alpha})
\[
I(\alpha)([\rho],q)=\sum_{g}\frac{1}{\dim(V_\rho)}\alpha(g,q)\tr(\rho(g){T^{[\rho]}_q}^{-1})
\]
of the isomorphism $I: H^\bullet_{\text{\rm CR}}(\Y, \complex)\to
H^\bullet_{orb}(\widehat{\Y}, c, \complex)$. In particular, a
class in $H^\bullet_{\text{\rm CR}}(\Y, \complex)$ can be
represented by an $\frakH$-invariant differential form on
$\frakH^{(0)}:=\{(g,q)\in \frakH|s(q)=t(q)\}$, and a class in $H^\bullet_{orb}(\widehat{\Y},c, \complex)$ can be
represented by a differential form supported on
$\mathfrak{Q}':=\{([\rho], q): q([\rho])=[\rho], s(q)=t(q)\}\subset
\widehat{G}\rtimes_c\mathfrak {Q}$ that is invariant under the
$c$-twisted conjugation action by $\widehat{G}\rtimes_c\mathfrak
{Q}$. We prove Eq.  (\ref{3pt_inv_3}) by showing
that the isomorphism $I$, which is the inverse of the map $\Psi$,
satisfies Eq.  (\ref{3pt_inv_3}). 

\begin{lemma}\label{lem:center-comm}
At $(g,q)\in \mathfrak {H}^{(0)}$, fix a basis $\{\omega_j\}$ of
$\wedge^\bullet T^*_{g,q}(\frakH^{(0)})((\hbar))$. Write
$\alpha(g,q)=\sum_\sigma \alpha(g,q)^j \omega_j$. Then $\sum_g
\alpha(g,q)^j$ satisfies the following: for any $g_0\in G$, we have
\[
\sum_g \alpha(g,q)^j g_0g =\sum_g \alpha(g,q)^j
g\Ad_{\sigma(q)}(g_0).
\]
\end{lemma}
\begin{proof} Since $\alpha=\sum_{g,q}\alpha(g,q)$ is invariant
under the conjugation action of $H$, for any $g_0\in G$, we have
\[
g_0\sum_{g,q}\alpha(g,q)(g,q)=\sum_{g,q}\alpha(g,q)(g,q)g_0.
\]
Since $G$ acts on $\mathfrak {H}_0$ trivially, this implies that
\[
\sum_{g,q}\alpha(g,q)(g_0g,q)=\sum_{g,q}\alpha(g,q)(g\Ad_{\sigma(q)}(g_0),
q).
\]
We deduce the lemma from the above equation by looking at every $q$
component.
\end{proof}

\begin{lemma}\label{lem:rho-coh}
Let $\alpha$ be a closed differential form on $\mathfrak
{H}^{(0)}=\{(g,q)\in \mathfrak {H}: s(q)=t(q)\}\subset \mathfrak
{H}$ that is invariant under the conjugation action of $\mathfrak
{H}$ on $\mathfrak {H}^{(0)}$. Following the same conventions as in
Lemma \ref{lem:center-comm}, we write $\alpha=\sum_{j,
g,q}\alpha(g,q)^j$. Then
\begin{enumerate}
\item if $q([\rho])=[\rho]$ and $\rho$ is an irreducible representation of $G$, then
$\sum_{g}\alpha(g,q)^\sigma\rho(g){T^{[\rho]}_q}^{-1}$ is a scalar
multiple of the identity operator.
\item if $q([\rho])\ne [\rho]$, then $\sum_g
\alpha(g,q)^\sigma\rho(g)=0$.
\end{enumerate}
\end{lemma}

\begin{proof}
When $q([\rho])=[\rho]$, we apply $\rho$ to $\sum_g \alpha(g,q)^j
g$. By Lemma \ref{lem:center-comm}, we have
\begin{eqnarray*}
&\rho\left(\sum_g \alpha(g,q)^jg_0g\right) &= \rho \Big(
\sum_g \alpha(g,q)^j g\Ad_{\sigma(q)}(g_0)\Big ),\\
&\rho(g_0)\rho\Big(\sum_g \alpha(g,q)^j g\Big)&=\rho\Big(\sum_g
\alpha(g,q)^j
g\Big)\rho(\Ad_{\sigma(q)}(g_0))\\
&&=\rho\Big(\sum_g \alpha(g,q)^j
g\Big)(T^{[\rho]}_q)^{-1}q([\rho])(g_0) T^{[\rho]}_q,
\end{eqnarray*}
where in the last equality, we used the property of
$\rho(\Ad_{\sigma(q)}(g))$ in Sec.  \ref{sec:mackey}. As
$q([\rho])=[\rho]$, we have
\[
\rho(g_0)\rho\Big(\sum_g \alpha(g,q)^j
g\Big){T^{[\rho]}_q}^{-1}=\rho\Big(\sum_g \alpha(g,q)^j
g\Big){T^{[\rho]}_q}^{-1}\rho(g_0).
\]
If $\rho$ is an irreducible representation of $G$, then the above
equation together with Schur's lemma implies that $\rho\left(\sum_g
\alpha(g,q)^j g\right){T^{[\rho]}_q}^{-1}$ is a scalar multiple of
the identity operator. This proves the first claim.

For the second claim, we consider the following equality from Lemma
\ref{lem:center-comm}:
\[
\sum_{g}\alpha(g,q)(g_0g,q)=\sum_{g}\alpha(g,q)(g\Ad_{\sigma(q)}(g_0),
q).
\]
This is equivalent to
\[
\sum_{g}\alpha(g,q)g=\sum_{g}\alpha(g,q)g_0^{-1}g\Ad_{\sigma(q)}(g_0).
\]
We apply $\rho$ to both sides of the above equation and obtain
\[
\begin{split}
\sum_{g}\alpha(g,q)\rho(g)&=\sum_g
\alpha(g,q)\rho(g_0^{-1})\rho(g\Ad_{\sigma(q)}(g_0)g^{-1})\rho(g)\\
=&\frac{1}{|G|}\sum_g\alpha(g,q)\sum_{g_0}\rho(g_0^{-1})\rho(g\Ad_{\sigma(q)}(g_0)g^{-1})\rho(g).
\end{split}
\]
We observe that $\rho(g\Ad_{\sigma(q)}(-)g^{-1})$ is an irreducible
representation of $G$ equivalent to $\rho(\Ad_{\sigma(q)}(-))$, which
is equivalent to $q([\rho])$.

If $q([\rho])\ne [\rho]$, then by \cite[Ch. IX, Theorem 4.2]{fe-do},
\[
\sum_g\alpha(g,q)\rho(g)=\frac{1}{|G|}\sum_g\alpha(g,q)\sum_{g_0}\rho(g_0^{-1})\rho(g\Ad_{\sigma(q)}(g_0)g^{-1})\rho(g)=0.
\]
\end{proof}

Consider $\alpha_1, \alpha_2, \alpha_3\in H^\bullet(I{\Y},
\complex((\hbar)))$. We represent the $\alpha_i$'s as $\mathfrak
{H}$ invariant differential forms over $\mathfrak {H}^{(0)}$. First we
compute $\pi_{\Y*}(ev_{\Y, 1}^*\big(\alpha_1)\cup ev_{\Y, 2}^*(\alpha_2)\cup
ev_{\Y, 3}^*(\alpha_3)\big)$.

For $q_1, q_2\in \mathfrak {Q}^{(0)}=\{q\in \mathfrak {Q}:
s(q)=t(q)\}$ and $q_3=(q_1q_2)^{-1}$, it is equal to
\begin{equation}\label{eq:int-y}
\frac{1}{|G|}\sum_{\tiny \begin{array}{c}g_1, g_2, g_3\\
(g_1,q_1)(g_2,q_2)(g_3, q_3)=1\end{array}} \alpha_1(g_1,
q_1)q_1^*(\alpha_2(g_2, q_2))q_1^*q_2^*(\alpha_3(g_3, q_3)).
\end{equation}

Now we compute 
\begin{equation}\label{RHS_5.13}
\sum_\bi w(\widehat{\Y}_\bi)
\pi_{\widehat{\Y}_\bi*}\left(ev_{\widehat{\Y}_\bi,
1}^*I(\alpha_{1\bi})\cup ev_{\widehat{\Y}_\bi,
2}^*I(\alpha_{2\bi})\cup ev_{\widehat{\Y}_\bi,
3}^*I(\alpha_{3\bi})\right).
\end{equation}
Using Eq.  (\ref{eq:I-iso}), for
$q_1, q_2, q_3=(q_1q_2)^{-1}\in \mathfrak{Q}^{(0)}$, it is equal to
\begin{eqnarray*}
&&\sum_{\rho,
q_1([\rho])=q_2([\rho])=\rho}\frac{\dim(V_\rho)^2}{|G|^2}\frac{c^{[\rho]}(q_1,q_2)c^{[\rho]}(q_1q_2,q_3)}{\dim(V_\rho)^3}\sum_{g_1}\alpha_1(g_1,
q_1)\tr(\rho(g_1){T^{[\rho]}_{q_1}}^{-1})\\
&&\qquad \qquad \qquad \sum_{g_2}q_1^*(\alpha_2(g_2,
q_2))\tr(\rho(g_2){T^{[\rho]}_{q_2}}^{-1})\sum_{g_3}q_1^*q_2^*(\alpha_3(g_3,
q_3))\tr(\rho(g_3){T^{[\rho]}_{q_3}}^{-1})\\
&=&\frac{1}{|G|}\sum_{\rho,
q_1([\rho])=q_2([\rho])=\rho}\frac{c^{[\rho]}(q_1,q_2)c^{[\rho]}(q_1q_2,q_3)}{\dim(V_\rho)|G|}\sum_{g_1,g_2,g_3}\alpha_1(g_1,
q_1)\tr(\rho(g_1){T^{[\rho]}_{q_1}}^{-1})\\
&&\hskip 2cm q_1^*(\alpha_2(g_2,
q_2))\tr(\rho(g_2){T^{[\rho]}_{q_2}}^{-1})q_1^*q_2^*(\alpha_3(g_3,
q_3))\tr(\rho(g_3){T^{[\rho]}_{q_3}}^{-1}).
\end{eqnarray*}

Using the fact that $\sum_{g_i}\alpha_i
(g_i,q_i)\rho(g_i){T^{[\rho]}_{q_i} }^{-1}$ is a scalar operator, we
can rewrite the term
\[
\sum_{g_1,g_2,g_3}\alpha_1(g_1,
q_1)\tr(\rho(g_1){T^{[\rho]}_{q_1}}^{-1}) q_1^*(\alpha_2(g_2,
q_2))\tr(\rho(g_2){T^{[\rho]}_{q_2}}^{-1})q_1^*q_2^*(\alpha_3(g_3,
q_3))\tr(\rho(g_3){T^{[\rho]}_{q_3}}^{-1})
\]
in the above equation as
\[
\sum_{g_1,g_2,g_3}\dim(V_\rho)^2\alpha_1(g_1, q_1)
q_1^*(\alpha_2(g_2, q_2))q_1^*q_2^*(\alpha_3(g_3,
q_3))\tr(\rho(g_1){T^{[\rho]}_{q_1}}^{-1}\rho(g_2){T^{[\rho]}_{q_2}}^{-1}\rho(g_3){T^{[\rho]}_{q_3}}^{-1}).
\]

The operator $\rho(g_1){T^{[\rho]}_{q_1}}^{-1}\rho(g_2){T^{[\rho]}_{q_2}}^{-1}\rho(g_3){T^{[\rho]}_{q_3}}^{-1}$ can be written as
\[
\rho(g_1){T^{[\rho]}_{q_1}}^{-1}\rho(g_2)T^{[\rho]}_{q_1}{T^{[\rho]}_{q_1}}^{-1}
{T^{[\rho]}_{q_2}}^{-1}\rho(g_3)T^{[\rho]}_{q_1q_2}{T^{[\rho]}_{q_1q_2}}^{-1}{T^{[\rho]}_{q_3}}^{-1}.
\]
Using the defining Eq.  (\ref{eq:dfn-c}) of $c^{[\rho]}(q_1,
q_2)$ and $c^{[\rho]}(q_1q_2, q_3)$, we have
\begin{eqnarray*}
c^{[\rho]}(q_1,q_2){T^{[\rho]}_{q_1}}^{-1}
{T^{[\rho]}_{q_2}}^{-1}=
\rho(\tau(q_1,q_2)){T^{[\rho]}_{q_1q_2}}^{-1},\ 
c^{[\rho]}(q_1q_2,q_3){T^{[\rho]}_{q_1q_2}}^{-1}{T^{[\rho]}_{q_3}}^{-1}=\rho(\tau(q_1q_2,q_3)){T^{[\rho]}_{q_1q_2q_3}}^{-1}.
\end{eqnarray*}

With the above considerations, (\ref{RHS_5.13}) can be written as
\[
\begin{split}
&\frac{1}{|G|^2}\sum_{\tiny\begin{array}{c}\rho,
q_1([\rho])=q_2([\rho])=\rho,\\ g_1, g_2,
g_3\end{array}}\dim(V_\rho)
\alpha_1(g_1,q_1)q_1^*{\alpha_2(g_2,q_2)}q_1^*q_2^*(\alpha_3(g_3,q_3))\\
&\hskip 2cm
\tr\left(\rho(g_1){T^{[\rho]}_{q_1}}^{-1}\rho(g_2)T^{[\rho]}_{q_1}\rho(\tau(q_1,q_2))
{T^{[\rho]}_{q_1q_2}}^{-1}\rho(g_3)T^{[\rho]}_{q_1q_2}\rho(\tau(q_1q_2,q_3))\right)\\
=&\frac{1}{|G|^2}\sum_{\tiny\begin{array}{c}\rho,
q_1([\rho])=q_2([\rho])=\rho,\\ g_1, g_2,
g_3\end{array}}\dim(V_\rho)
\alpha_1(g_1,q_1)q_1^*{\alpha_2(g_2,q_2)}q_1^*q_2^*(\alpha_3(g_3,q_3))\\
&\hskip 2cm
\tr\left(\rho(g_1\Ad_{\sigma(q_1)}(g_2)\tau(q_1,q_2)\Ad_{\sigma(q_1q_2)}(g_3)\tau(q_1q_2,q_3))\right).
\end{split}
\]
By Lemma \ref{lem:rho-coh}, (2), we know that the restriction
$q_1([\rho])=q_2([\rho])=[\rho]$ in the above summation can be
dropped since, otherwise, the contribution vanishes. Thus (\ref{RHS_5.13}) is equal to 
\[
\begin{split}
&\frac{1}{|G|^2}\sum_{\rho}\sum_{g_1, g_2, g_3}\dim(V_\rho) \alpha_1(g_1,q_1)q_1^*{\alpha_2(g_2,q_2)}q_1^*q_2^*(\alpha_3(g_3,q_3))\\
&\hskip 2cm
\tr\left(\rho(g_1\Ad_{\sigma(q_1)}(g_2)\tau(q_1,q_2)\Ad_{\sigma(q_1q_2)}(g_3)\tau(q_1q_2,q_3))\right)\\
=&\sum_{g_1, g_2, g_3}\alpha_1(g_1, g_1)q_1^*(\alpha_2(g_2, q_2))q_1^*q_2^*(\alpha_3(g_3,q_3))\\
&\hskip 2cm
\frac{1}{|G|^2}\sum_\rho\dim(V_\rho)\tr\left(\rho(g_1\Ad_{\sigma(q_1)}(g_2)\tau(q_1,q_2)\Ad_{\sigma(q_1q_2)}(g_3)\tau(q_1q_2,q_3))\right).
\end{split}
\]
By the orthogonality relations of characters of $G$ (see, e.g.,
\cite[(2.20)]{fu-ha}), we know that the above sum vanishes unless
$g_1\Ad_{\sigma(q_1)}(g_2)\tau(q_1,q_2)\Ad_{\sigma(q_1q_2)}(g_3)\tau(q_1q_2,q_3)=1$.

We conclude that (\ref{RHS_5.13}) is equal to
$$\frac{1}{|G|}\sum_{(g_1,q_1)(g_2,q_2)(g_3,q_3)=1}\alpha_1(g_1,
g_1)q_1^*(\alpha_2(g_2, q_2))q_1^*q_2^*(\alpha_3(g_3,q_3)),$$
which is exactly the expression (\ref{eq:int-y}) for
$
\pi_{\Y*}\left(ev_{\Y, 1}^*(\alpha_1)\cup ev_{\Y, 2}^*(\alpha_2)\cup
ev_{\Y, 3}^*(\alpha_3)\right).$ This proves (\ref{3pt_inv_3}).
\section{Gromov-Witten theory of $BH$}\label{sec:cqft}
In this section, we discuss how the results in Sec. 
\ref{subsec:center} can be used to deduce Gromov-Witten theoretic
consequences for the $G$-gerbe $BH$ over $BQ$, given by the exact
sequence (\ref{eq:extension}). The Gromov-Witten theory of the
classifying orbifold $BH$ of a finite group has been studied in
great detail in \cite{jk}, to which we refer the readers for basic
definitions and discussions.
\subsection{Quantum cohomology}
The orbifold quantum cohomology ring $QH_{orb}^\bullet(\X)$ of a
compact symplectic orbifold $\X$ is a deformation of the cohomology
$H^\bullet(I\X, \complex)$ of the inertia orbifold $I\X$ constructed
using genus $0$ Gromov-Witten invariants of $\X$. Details of the
construction can be found in \cite{cr2} (see \cite{agv1} for the
construction in the algebro-geometric context). In the special case
of $BH$, detailed discussions on $QH_{orb}^\bullet(BH)$ can be found
in \cite{jk}.

 It is known (see \cite{cr} and \cite{jk}) that the orbifold quantum cohomology
ring of $BH$ is simply the center of the group ring $\mathbb{C}H$,
i.e.,
\begin{equation}\label{QH=center}
QH_{orb}^\bullet(BH)\simeq Z(\mathbb{C}H).
\end{equation}

Consider the $Q$-action on $\widehat{G}$ as discussed in Sec.  \ref{subsec:groupalgebra}. The set $\widehat{G}$ may be divided into a disjoint union of
$Q$-orbits, $\widehat{G}=\cup_{O_i\in \text{Orb}^Q(\widehat{G})} O_i.$
For each $O_i$, pick $[\rho_i]\in O_i$ and let
$Q_i=\text{Stab}([\rho_i])\subset Q$ denote the stabilizer subgroup
of $[\rho_i]$. Then as orbifolds, we have $[\widehat{G}/Q]\simeq \cup_i BQ_i. $
Each $BQ_i$ admits a flat $U(1)$-gerbe $c_i$ obtained from the flat $U(1)$-gerbe $c$ on $\widehat{BH}=[\widehat{G}/Q]$. By Theorem \ref{thm:orb-morita} we
have

\begin{proposition}
The twisted group(oid) algebras $C(\widehat{G}\rtimes Q, c)$ and $\oplus_i C(Q_i, c_i)$ are Morita
equivalent.
\end{proposition}

As reviewed in Sec.  \ref{subsec:cr}, given a compact symplectic
orbifold $\X$ and a flat $U(1)$-gerbe $c$ on $\X$, one can consider
the cohomology $H^\bullet(I\X, \sL_c)$ with coefficients in the
inner local system $\sL_c\to I\X$ associated with $c$. The work
\cite{pry} constructs a deformation, $QH_{orb,c}^\bullet(\X)$, of
$H^\bullet(I\X, \sL_c)$ using {\em gerbe-twisted Gromov-Witten
invariants} of $\X$. The basics of gerbe-twisted Gromov-Witten
invariants in the case of $(BQ_i, c_i)$ will be reviewed below. The
construction in the general case can be found in \cite{pry}.

It is known (see \cite[Example 6.4]{ru1} and \cite{pry}) that the
$c_i$-twisted orbifold quantum cohomology of $BQ_i$ coincides with
the center of the twisted group algebra, i.e.,
\begin{equation}\label{QH=twisted_center}
QH_{orb, c_i}^\bullet(BQ_i)\simeq Z(C(Q_i, c_i)).
\end{equation}
Using the map $I$ defined in (\ref{eqn:orbit_decomp_of_center}), we get
\begin{equation}\label{QH_isom}
QH_{orb}^\bullet(BH)\overset{(\ref{QH=center})}{\simeq}
Z(\mathbb{C}H)\overset{I}{\simeq} \oplus_i Z(C(Q_i,
c_i))\overset{(\ref{QH=twisted_center})}{\simeq} \oplus_i QH_{orb,
c_i}^\bullet(BQ_i)=QH_{orb,c}^\bullet([\widehat{G}/Q]).
\end{equation}

As stated in \cite[Corollary 3.3]{jk}, the map (\ref{QH=center})
identifies the orbifold Poincar\'e pairing $(-,-)_{orb}^{BH}$ on
$BH$ with the pairing defined by the trace $\text{tr}_H$. Similarly,
one can deduce from \cite{ru1} that the map
(\ref{QH=twisted_center}) identifies the orbifold Poincar\'e pairing
on $QH_{orb, c}^\bullet(BQ_i)$ with the one defined by the trace
$\text{tr}_{[\rho_i]}$. Now Proposition \ref{prop:trace} implies
that

\begin{lemma}\label{equality_pairing}
The isomorphism in (\ref{QH_isom}) identifies the orbifold
Poincar\'e pairing $(-,-)_{orb}^{BH}$ with the following rescaled
orbifold Poincar\'e pairing:
$$\bigoplus_{i} \Big(\frac{\text{dim}\,V_{\rho_i}}{|G|}\Big)^2(-,-)_{orb}^{BQ_i},$$
where $\text{dim}\, V_{\rho_i}$ is the dimension of the irreducible $G$-representation $\rho_i: G\to \End(V_{\rho_i})$.
\end{lemma}

\subsection{Gromov-Witten invariants}
We may view (\ref{QH_isom}) as a decomposition of
$QH_{orb}^\bullet(BH)$ into a direct sum. In this section, we
discuss how to extend this decomposition to the full Gromov-Witten
theory. In this subsection, the letters $g$ and $g_i$ denote genera
of curves.

We begin with a brief review of the gerbe-twisted orbifold
Gromov-Witten theory for $(BQ_i,c_i)$. Details can be found in
\cite{pry} and \cite{ru1}. Let
$$IBQ_i=\bigcup_{\<q\>\subset Q_i} BC_{Q_i}(q)$$ be the
decomposition of the inertia orbifold of $BQ_i$, where the union is
taken over the conjugacy classes $\<q\>$ of $Q_i$, and
$C_{Q_i}(q)\subset Q_i$ is the centralizer subgroup of $q\in Q$. Let
$\sL_{c_i}\to IBQ_i$ be the inner local system associated
with the $U(1)$-gerbe $c_i$ (see \cite{ru1} and \cite{pry} for its
construction). Given classes $\alpha_j\in H^\bullet(BC_{Q_i}(q_j),
\sL_{c_i}), 1\leq j\leq n$ and non-negative integer $a_1,...,a_n$,
by the construction of gerbe-twisted Gromov-Witten invariants, there
exists an isomorphism of line bundles (see \cite[Sec.  5.2]{pry})
$$\theta_{\<q_1\>,...,\<q_n\>}: \otimes_{j=1}^n (\sL_{c_i}|_{BC_{Q_i}(q_j)})\to \mathbb{C},$$
such that the invariant is defined to be (see \cite{pry}, Definition
5.4)
\begin{equation}\label{eq:n-function-Q}
\begin{split}
\langle\tau_{a_1}(\alpha_1),...,\tau_{a_n}(\alpha_n)\rangle_{g,n}^{BQ_i,
c} =\int_{\Mbar_{g,n}(BQ_i,
\<q_1\>,..., \<q_n\>)}(\theta_{\<q_1\>,...,\<q_n\>})_*(\prod_{j=1}^n
ev_j^*\alpha_j)\prod_{j=1}^n\pi^*\psi_j^{a_j}.
\end{split}
\end{equation}
Here,
$\theta(\alpha_1,...,\alpha_n):=(\theta_{\<q_1\>,...,\<q_n\>})_*(\prod_{j=1}^n
ev_j^*\alpha_j)\in \mathbb{C}$. The integral is taken over the
moduli space  $\Mbar_{g,n}(BQ_i, \<q_1\>,...,\<q_n\>)$ of
$n$-pointed genus $g$ orbifold stable maps to $BQ_i$ such that the
orbifold structures at marked points are determined by the conjugacy
classes $\<q_1\>,...,\<q_n\>$ of $Q_i$. And denote by
$$\pi:\Mbar_{g,n}(BQ_i, \<q_1\>,...,\<q_n\>)\to \Mbar_{g,n}$$ the
natural map from forgetting the orbifold structures and by
$\psi_j\in H^2(\Mbar_{g,n})$ the descendant classes. We refer the
reader to \cite{jk} for a detailed discussion of orbifold stable
maps to $BQ_i$ and their moduli spaces.

Consider descendant integrals over $\Mbar_{g,n}$, 
$\langle\tau_{a_1},...,\tau_{a_n}\rangle_{g,n}:=\int_{\Mbar_{g,n}}\prod_{j=1}^n\psi_j^{a_j}.$ 
By a discussion similar to \cite[Proposition 3.4]{jk}, the projection formula implies
\begin{equation}\label{eqn:GW_inv_BQ_i1}
\begin{split}
\langle\tau_{a_1}(\alpha_1),...,\tau_{a_n}(\alpha_n)\rangle_{g,n}^{BQ_i,
c}
=&(\text{deg}\,\pi)\theta(\alpha_1,...,\alpha_n)\int_{\Mbar_{g,n}}\prod_{j=1}^n\psi_j^{a_j}
=(\text{deg}\,\pi)\theta(\alpha_1,...,\alpha_n)\langle\tau_{a_1},...,\tau_{a_n}\rangle_{g,n}.
\end{split}
\end{equation}
According to the proof of \cite[Proposition 3.4]{jk}, the degree
$\text{deg}\,\pi$ is equal to
\begin{equation}
\Omega_g^{Q_i}(\<q_1\>,...,\<q_n\>)
:=\frac{1}{|Q_i|}\#\Big\{\alpha_1,...,\alpha_g, \beta_1,...,\beta_g, \sigma_1,...,\sigma_n|\prod_{j=1}^g[\alpha_j, \beta_j]=\prod_{k=1}^n\sigma_k, \sigma_k\in \<q_k\> \text{ for all }k\Big\}.
\end{equation}

Recall that, according to \cite[Example 6.4]{ru1}, the cohomology
vector space $H^\bullet(BC_{Q_i}(q), \sL_{c_i})$ is $1$-dimensional
if $\<q\>$ is a {\em $c_i$-regular conjugacy class} of $Q_i$ and
$0$-dimensional otherwise. Recall that a conjugacy class $\<q\>$ of
$Q_i$ is $c_i$-regular if $c_i(q_1,q)c_i(q,q_1)^{-1}=1$ for every
$q_1\in C_{Q_i}(q)$. For a $c_i$-regular conjugacy class $\<q\>$, we
denote by $e_{\<q\>}$ a generator of $H^\bullet(BC_{Q_i}(q),
\sL_{c_i})$.

\begin{proposition}
The assignment $\Lambda_{g,n}^{BQ_i, c_i}: H_{orb,
c_i}^\bullet(BQ_i)^{\otimes n}\to \com$ given by  $$\alpha_1\otimes
...\otimes \alpha_n\mapsto\Lambda_{g,n}^{BQ_i,
c_i}(\alpha_1,...,\alpha_n):=\Omega_g^{Q_i}(\<q_1\>,...,\<q_n\>)\theta(\alpha_1,...,\alpha_n)$$
satisfies the following properties:
\begin{enumerate}
\item (Forgetting tails)
\begin{equation}\label{forget_tail}
\Lambda_{g,n}^{BQ_i, c_i}(e_{\<q_1\>},...,e_{\<q_n\>})=
\Lambda_{g,n+1}^{BQ_i, c_i}(e_{\<1\>}, e_{\<q_1\>},...,e_{\<q_n\>}).
\end{equation}
\item (Cutting loops)
\begin{equation}\label{cut_loop}
\Lambda_{g,n}^{BQ_i,
c_i}(e_{\<q_1\>},...,e_{\<q_n\>})=\sum_{\<q\>}|C_{Q_i}(q)|\Lambda_{g-1,
n+2}^{BQ_i, c_i}(e_{\<q\>}, e_{\<q^{-1}\>}, e_{\<q_1\>},...,e_{\<q_n\>}).
\end{equation}
\item (Cutting trees)
For $g=g_1+g_2$ and $\{1,...,n\}=P_1\coprod P_2$, we have
\begin{equation}\label{cut_tree}
\Lambda_{g,n}^{BQ_i,c_i}(e_{\<q_1\>},...,e_{\<q_n\>})=\sum_{\<q\>}|C_{Q_i}(q)|\Lambda_{g_1,
|P_1|+1}^{BQ_i, c_i}(\{e_{\<q_i\>}\}_{i\in P_1}, e_{\<q\>})\Lambda_{g_2, |P_2|+1}^{BQ_i,c_i}(e_{\<q^{-1}\>}, \{e_{\<q_i\>}\}_{i\in P_2}).
\end{equation}
\end{enumerate}
\end{proposition}
\begin{remark}
\hfill
\begin{enumerate}
\item
It is easy to see that $\Lambda_{g,n}^{BQ_i, c_i}(e_{\<q_1\>},...,e_{\<q_n\>})$ is independent of the ordering of
conjugacy classes 
$\<q_1\>,...,\<q_n\>$.
\item
The collection of maps $\{\Lambda_{g,n}^{BQ_i, c_i} \}$
are determined by their values on $e_{\<q_1\>}\otimes ...\otimes
e_{\<q_n\>}$, and therefore the gerbe-twisted Gromov-Witten invariants of $BQ_i$ form a
cohomological field theory\footnote{The construction of
\cite{pry} implies that gerbe-twisted Gromov-Witten invariants form
a cohomological field theory in general.} (see, for example,
\cite[Chapter III]{manin} for a comprehensive introduction to
cohomological field theory).
\end{enumerate}
\end{remark}

\begin{proof}
The proof of this proposition amounts to a repeat of the arguments
in the proof of \cite[Lemma 3.5]{jk}. Handling the factor
$\theta(\alpha_1,...,\alpha_n)$ requires some properties of the
inner local system (see \cite{ru1}).

To prove the {\em forgetting tails} property (\ref{forget_tail}),
first note that for  the inner local system $\sL_{c_i}$, the
restriction $\sL_{c_i}|_{BC_{Q_i}(1)}$ is a trivial line bundle.
Thus, $\theta_{\<1\>, \<q_1\>,...\<q_n\>}=id\otimes
\theta_{\<q_1\>,...,\<q_n\>}$ and $\theta(e_{\<1\>},
e_{\<q_1\>},...,e_{\<q_n\>})=\theta(e_{\<q_1\>},...,e_{\<q_n\>})$.
With the relevant property of $\Omega_g^{Q_i}$ (see
\cite{jk}, Proposition 3.5 (3)), this implies (\ref{forget_tail}).

To prove the {\em cutting loop} property (\ref{cut_loop}), first
note that by \cite[Proposition 3.5 (2)]{jk}, we have
$$\Omega_{g}^{Q_i}(\<q_1\>,...,\<q_n\>)=\sum_{\<q\>}|C_{Q_i}(q)|\Omega_{g-1}^{Q_i}(\<q\>,
\<q^{-1}\>, \<q_1\>,...,\<q_n\>).
$$
Thus, it suffices to prove
\begin{equation}\label{splitting_theta}
\theta(e_{\<q_1\>},...,e_{\<q_n\>})=\theta(e_{\<q\>}, e_{\<q^{-1}\>}, e_{\<q_1\>},...,e_{\<q_n\>}).
\end{equation}
This follows from the following two facts:
\begin{enumerate}
\item
$\theta_{\<q\>, \<q^{-1}\>,\<q_1\>,...,\<q_n\>}=\theta_{\<q\>,
\<q^{-1}\>}\otimes \theta_{\<q_1\>,...,\<q_n\>}$, by 
the gluing law of $\theta$ (see \cite[Sec.  5.1]{pry}).
\item
$\theta_{\<q\>, \<q^{-1}\>}(e_{\<q\>}, e_{\<q^{-1}\>})=1$, 
by a direct calculation (see \cite[Example 6.4]{ru1}).
\end{enumerate}
The {\em cutting tree} property (\ref{cut_tree}) is proved by a
similar argument and we omit the details.
\end{proof}

The Gromov-Witten theory of $BH$ is completely solved by \cite{jk}.
It is not hard to see that the methods of \cite{jk} can be used
to solve the $c_i$-twisted Gromov-Witten theory of $BQ_i$. Instead of 
pursuing this here, we will compare the
Gromov-Witten theory of $BH$ with the $c_i$-twisted Gromov-Witten
theory of $BQ_i$.

For $\phi_1,...,\phi_n\in H^\bullet(IBH,\complex)$ and
integers $a_1,...,a_n\geq 0$, denote by
$\<\tau_{a_1}(\phi_1),...,\tau_{a_n}(\phi_n)\>_{g,n}^{BH}$ the
corresponding descendant Gromov-Witten invariant of $BH$. See \cite[Sec.  3]{jk} for its definition. Let
$IBH=\bigcup_{\<h\>\subset H} BC_H(h)$ be the decomposition of the
inertia orbifold of $BH$, where the union is taken over conjugacy
classes $\<h\>$ of $H$, and $C_H(h)\subset H$ is the centralizer
subgroup of $h\in H$. For each conjugacy class $\<h\>$, let
$1_{\<h\>}:=1\in H^0(BC_H(h),\complex)$ be the natural generator. By
\cite[Proposition 3.4]{jk}, we have
\begin{equation}\label{eqn:GW_inv_BH1}
\<\tau_{a_1}(1_{\<h_1\>}),...,\tau_{a_n}(1_{\<h_n\>})\>_{g,n}^{BH}=\langle\tau_{a_1},...,\tau_{a_n}\rangle_{g,n}\Omega_g^H(\<h_1\>,...,\<h_n\>),
\end{equation}
where the quantity $\Omega_g^H(\<h_1\>,...,\<h_n\>)$ is given by
\begin{equation}
\Omega_g^H(\<h_1\>,...,\<h_n\>)
=\frac{1}{|H|}\#\Big\{\alpha_1,...,\alpha_g, \beta_1,...,\beta_g,
\sigma_1,...,\sigma_n|\prod_{j=1}^g[\alpha_j, \beta_j]=\prod_{k=1}^n\sigma_k, \sigma_k\in \<h_k\> \text{ for all }k\Big\}.
\end{equation}
By \cite[Lemma 3.5]{jk}, the quantity $\Omega_g^H(\<h_1\>,...,\<h_n\>)$ satisfies the following properties:
\begin{eqnarray}\label{forget_tail_BH}
&\Omega_{g}^{H}(1_{\<h_1\>},...,1_{\<h_n\>})&=
\Omega_{g}^{H}(1_{\<1\>}, 1_{\<h_1\>},...,1_{\<h_n\>}),\\
\label{cut_loop_BH}
&\Omega_{g}^{H}(1_{\<h_1\>},...,1_{\<h_n\>})&=\sum_{\<h\>}|C_{H}(h)|\Omega_{g-1}^{H}(1_{\<h\>}, 1_{\<h^{-1}\>}, 1_{\<h_1\>},...,1_{\<h_n\>}),\\
\label{cut_tree_BH}
&\Omega_{g}^{H}(1_{\<h_1\>},...,1_{\<h_n\>})&=\sum_{\<h\>}|C_{H}(h)|\Omega_{g_1}^{H}(\{1_{\<h_i\>}\}_{i\in P_1}, 1_{\<h\>})\Omega_{g_2}^{H}(1_{\<h^{-1}\>}, \{1_{\<h_i\>}\}_{i\in P_2}),
\end{eqnarray}
where $g=g_1+g_2$ and $\{1,...,n\}=P_1\coprod P_2$.

\begin{theorem}\label{thm:GW_inv_decomp}
Let $\{ \phi_{ij}\}_{1\leq j\leq \text{dim}H_{orb,c_i}^\bullet(BQ_i)}\subset H_{orb,c_i}^\bullet(BQ_i)$
be an additive basis. We also view $\phi_{ij}$ as an element in
$QH_{orb}^\bullet(BH)$ via (\ref{QH_isom}). Then
\begin{equation}\label{decomp_vanishing}
\langle\tau_{a_1}(\phi_{i_1j_1}),...,\tau_{a_n}(\phi_{i_nj_n})\rangle_{g,n}^{BH}=0
\end{equation}
 unless $i_1=i_2=...=i_n=:i$, in which case,
\begin{equation}\label{decomp_formula}
\langle\tau_{a_1}(\phi_{ij_1}),...,\tau_{a_n}(\phi_{ij_n})\rangle_{g,n}^{BH}=\Big(\frac{\text{dim}\, V_{\rho_i}}{|G|}\Big)^{2-2g}\langle\tau_{a_1}(\phi_{ij_1}),...,\tau_{a_n}(\phi_{ij_n})\rangle_{g,n}^{BQ_i,
c_i}.
\end{equation}
\end{theorem}
\begin{proof}
We apply the argument in the proof of \cite[Proposition 4.2]{jk}. For $3$-pointed genus $0$ invariants,
(\ref{decomp_vanishing})-(\ref{decomp_formula}) follow easily from
the isomorphism (\ref{QH_isom}) and Lemma \ref{equality_pairing}.
Then we proceed by induction on the genus $g$ and the number of
insertions $n$. Using (\ref{eqn:GW_inv_BQ_i1}) and
(\ref{eqn:GW_inv_BH1}), we can rephrase
(\ref{decomp_vanishing})-(\ref{decomp_formula}) in terms of the
quantities $\Lambda_{g,n}^{BQ_i,c_i}(-)$ and $\Omega_g^H(-)$. The
induction step can then be carried out using properties
(\ref{forget_tail})-(\ref{cut_tree}) of
$\Lambda_{g,n}^{BQ_i,c_i}(-)$ and properties
(\ref{forget_tail_BH})-(\ref{cut_tree_BH}) of $\Omega_g^H(-)$. The
details are left to the reader.
\end{proof}

We now use Theorem \ref{thm:GW_inv_decomp} to prove Conjecture
\ref{intro:dual_GW} for the gerbe $BH\to BQ$. Let $$\{t_{ij, a}|
O_i\in \text{Orb}^Q(\widehat{G}), 1\leq j\leq
\text{dim}H_{orb,c_i}^\bullet(BQ_i), a\in \mathbb{N}_{\geq 0}\}$$ be
a set of variables. The genus $g$ descendant Gromov-Witten
potentials are defined as follows:
{\small
\begin{eqnarray*}
\F^g_{BQ_i,c_i}(\{t_{ij,a}\}_{1\leq j\leq\text{dim}H_{orb,c_i}^\bullet(BQ_i), a\in \mathbb{N}_{\geq 0}})
:=\sum_{n\geq 0}\sum_{a_1,...,a_n\geq 0}\frac{\prod_{k=1}^nt_{ij_k}}{n!}\<\tau_{a_1}(\phi_{ij_1}),...,\tau_{a_n}(\phi_{ij_n})\>_{g,n}^{BQ_i, c_i},\\
\F^g_{BH}(\{t_{ij,a}\}_{O_i\in \text{Orb}^Q(\widehat{G}), 1\leq j\leq \text{dim}H_{orb,c_i}^\bullet(BQ_i), a\in \mathbb{N}_{\geq 0}}):=\sum_{n\geq 0}\sum_{a_1,...,a_n\geq 0}\frac{\prod_{k=1}^nt_{i_kj_k}}{n!}\<\tau_{a_1}(\phi_{i_1j_1}),...,\tau_{a_n}(\phi_{i_nj_n})\>_{g,n}^{BH}.
\end{eqnarray*}}
The total descendant potentials are defined as follows:
\begin{eqnarray*}
\D_{BQ_i, c_i}(\{t_{ij,a}\}_{1\leq j\leq \text{dim}H_{orb,c_i}^\bullet(BQ_i), a\in \mathbb{N}_{\geq 0}}; \epsilon):=\exp\Big(\sum_{g\geq 0} \epsilon^{2g-2}\F^g_{BQ_i,c_i}\Big),\\
\D_{BH}(\{t_{ij,a}\}_{O_i\in \text{Orb}^Q(\widehat{G}), 1\leq j\leq \text{dim}H_{orb,c_i}^\bullet(BQ_i), a\in \mathbb{N}_{\geq 0}}; \epsilon):= \exp\Big(\sum_{g\geq 0} \epsilon^{2g-2}\F^g_{BH}\Big).
\end{eqnarray*}
By Theorem \ref{thm:GW_inv_decomp}, we confirms Conjecture \ref{intro:dual_GW} in this case with the following equation
\begin{equation*}
\begin{split}
&\D_{BH}(\{t_{ij,a}\}_{O_i\in \text{Orb}^Q(\widehat{G}), 1\leq j\leq \text{dim}H_{orb,c_i}^\bullet(BQ_i), a\in \mathbb{N}_{\geq 0}}; \epsilon)\\
=&\sum_{O_i\in \text{Orb}^Q(\widehat{G})}\D_{BQ_i, c_i}(\{t_{ij,a}\}_{1\leq j\leq \text{dim}H_{orb,c_i}^\bullet(BQ_i), a\in \mathbb{N}_{\geq 0}}; \epsilon\frac{|G|}{\text{dim}\, V_{\rho_i}}).
\end{split}
\end{equation*}

\section{Sheaves on gerbes and twisted sheaves}\label{sec:sheaf}
In this section, we discuss sheaf theoretic aspects of the duality
of gerbes. We prove that the category of sheaves on a $G$-gerbe $\Y$ over
$\B$ is equivalent to the category of $c$-{\em twisted} sheaves on
its dual $\widehat{\Y}$.

To illustrate our approach to sheaf theory on $G$-gerbes, we begin
with considering (complex) vector bundles on the simplest example of
$G$-gerbes, namely, $BG\to pt$. In this case, the dual space is
$\widehat{G}$ and the $U(1)$-gerbe $c$ on $\widehat{G}$ is trivial.
We aim at relating vector bundles on $BG$ to vector bundles on
$\widehat{G}$.

By definition, a (complex) vector bundle on $BG$ is a
$\complex$-linear representation $V$ of $G$. There is a
decomposition of $V$ into a direct sum of irreducible
$G$-representations, i.e.,
\begin{equation}\label{G_rep_decomp}
V=\bigoplus_{[\rho]\in \widehat{G}} Hom_G(V_\rho, V)\otimes V_\rho,
\end{equation}
where $Hom_G(V_\rho, V)$ is the $\complex$-vector space of
$G$-equivariant linear maps from $V_\rho$ to $V$. View the collection
$\{Hom_G(V_\rho,V)| [\rho]\in \widehat{G}\}$ as a
vector bundle over the disconnected space $\widehat{G}$ by assigning
the vector space $Hom_G(V_\rho, V)$ to the point $[\rho]\in
\widehat{G}$. Thus, the assignment $V\mapsto \{Hom_G(V_\rho,V)|
[\rho]\in \widehat{G}\}$ defines a functor from the category of
complex vector bundles on $BG$ to the category of complex vector
bundles on $\widehat{G}$. It is easy to see that this functor is an
equivalence of abelian categories.

It is clear from the above discussion that  the decomposition (\ref{G_rep_decomp})
of a $G$-representation into a direct sum of irreducible
representations is the key in the
construction of the equivalence $V\mapsto \{Hom_G(V_\rho,V)|
[\rho]\in \widehat{G}\}$. Such a decomposition will also be the
key to our study of the sheaf theory on an arbitrary $G$-gerbe.

\subsection{Global quotient}\label{sheaf_on_global_quotient}
In this subsection, we study the sheaf theory for gerbes arising
from global quotients. Our approach is based on \cite[Chapter
XII]{fe-do}, which can be understood as the case of the $G$-gerbe
$BH\to BQ$ arising from the extension (\ref{eq:extension}).

Consider the exact sequence
(\ref{eq:extension}).
As in Sec.  \ref{subsec:groupalgebra}, we choose a section $s:
Q\to H$. Recall that with such a section we can define a cocycle
$\tau: Q\times Q\to G$ via
$s(q_1)s(q_2)=\tau(q_1, q_2)s(q_1q_2),$ for $ q_1, q_2\in Q.$

Let $M$ be a smooth manifold\footnote{Depending on the context, we work with Euclidean, analytic, or \'etale topology. Our
arguments work in all these settings.} with an $H$-action. We assume
that this $H$-action restricts to a trivial $G$-action on $M$.
Consequently, a $Q$-action on $M$ is naturally defined. The natural
map $[M/H]\to [M/Q]$
defines a $G$-gerbe. Note that the $H$ and $Q$ actions on $M$
agree in the following sense: let $q\in Q$ and $h\in
j^{-1}(q)\subset H$, then for any point $m\in M$, we have
$q(m)=h(m)$, as $G=\text{ker}\, (j)$ acts trivially
on $M$.

Given a $G$-representation $\rho$ and $h\in H$ we may consider
another $G$-representation,
\begin{equation}\label{action_on_rep}
G\ni g\mapsto \rho(hg h^{-1}).
\end{equation}
It is easy to see that (\ref{action_on_rep}) defines a right action
of $H$ on the set $\widehat{G}$ of {\em isomorphism classes} of
unitary irreducible representations of $G$. If $h\in G$, then
(\ref{action_on_rep}) is equivalent to $\rho$ via the intertwining
operator $\rho(h)$. Hence, by choosing a section $s:Q\to H$, one
sees that (\ref{action_on_rep}) defines a right action\footnote{Note
that this $Q$-action is independent of the choice of the section.}
of $Q$ on $\widehat{G}$ as well. This is the action\footnote{Again,
we write this action as a left action.} we have seen in Sec. 
\ref{subsec:groupalgebra}. Let $H$ and $Q$ act on the product
$M\times \widehat{G}$ by the given actions on the factors. The two actions agree in the following sense: let $q\in Q$ and $h\in j^{-1}(q)\subset H$; as $G=\text{ker}\, (j)$ acts trivially on $M\times
\widehat{G}$, for any $m\in M$ and $[\rho]\in
\widehat{G}$, we have $q((m, [\rho]))=h((m, [\rho]))$. 
\subsubsection{A twisted sheaf}
As in Sec. \ref{subsec:groupalgebra}, for each isomorphism class
$[\rho]\in \widehat{G}$, we {\em fix} a choice of a representative,
denoted by $\rho: G\to \End(V_\rho)$. Let $\calv_G$ be the sheaf over $M\times \widehat{G}$ defined by
requiring that its restriction to $M\times [\rho]$ is the trivial
sheaf with fiber $V_\rho$.  There is a natural $G$-sheaf structure
on $\calv_G$. We will construct a twisted action by $H$, so
that $\calv_G$ descends to a twisted sheaf on the orbifold
$[(M\times\widehat{G})/H]$. We refer to
\cite{cal} for a discussion on the category of twisted sheaves.

Let $T_q^{[\rho]}: V_{\rho}\to V_{q([\rho])}$ be the intertwining
operator introduced in Sec.  \ref{subsec:groupalgebra} with the following property
\begin{equation}\label{defining_property_intertwining_op}
\rho(s(q)gs(q)^{-1})=T_q^{[\rho]^{-1}}\circ q([\rho])(g)\circ
T_q^{[\rho]}.
\end{equation}
For $h\in H$, there exist unique $q\in Q$ and $g\in G$ such that
$h=gs(q)$. Define
$$E_{h, [\rho]}:=\rho(g)\circ T_q^{[\rho]^{-1}}: V_{q([\rho])}\to
V_{\rho}.$$ For $h\in H$ we may view $E_{h, [\rho]}$ as an
isomorphism $h^*\calv_G|_{M\times [\rho]} \to \calv_G|_{M\times
[\rho]}$, where $h: M\times \widehat{G}\to M\times \widehat{G}$ is
the map defined by the action of $h\in H$.

Similarly, for $q\in Q$, we define $E_{q, [\rho]}:=E_{s(q), [\rho]}$
and view it as an isomorphism $q^*\calv_G|_{M\times [\rho]} \to
\calv_G|_{M\times [\rho]}$, where $q: M\times \widehat{G}\to M\times
\widehat{G}$ is the map defined by the action by $q\in Q$.

Recall that in Sec. \ref{subsec:groupalgebra}, a cocycle $c:
\widehat{G}\times Q\times Q\to U(1)$ is defined. We extend this to a
cocycle $c: \widehat{G}\times H\times H\to U(1)$ as follows. For
$h_1, h_2\in H$, we may uniquely write $h_1=g_1s(q_1)$ and
$h_2=g_2s(q_2)$ with $q_1, q_2\in Q, g_1, g_2\in G$. Set $c^{[\rho]}(h_1, h_2):= c^{[\rho]}(q_1, q_2)$.

\begin{lemma}\label{lem:Q_str_on_V}
The collection $\{E_{h, [\rho]}| h\in H, [\rho]\in\widehat{G}\}$ of
isomorphisms defines a $c^{-1}$-twisted $H$-equivariant structure on
the sheaf $\calv_G$.
\end{lemma}
\begin{proof}
It suffices to check that the isomorphisms are
compatible with the group actions.
Let $h_1, h_2\in H$. We need to show that the composition $$E_{h_1,
[\rho]}\circ E_{h_2, h_1([\rho])}: h_2^*h_1^*\calv_G|_{M\times
[\rho]}\to h_1^*\calv_G|_{M\times [\rho]}\to \calv_G|_{M\times
[\rho]}$$ coincides, up to a twist, with $E_{h_1h_2, [\rho]}:
(h_1h_2)^*\calv_G|_{M\times [\rho]}\to \calv_G|_{M\times [\rho]}.$
Write $h_1=g_1s(q_1)$ and $h_2=g_2s(q_2)$ with $g_1, g_2\in G$ and
$q_1, q_2\in Q$ as above. We compute
\begin{equation*}
h_1h_2=g_1s(q_1)g_2s(q_2)=g_1s(q_1)g_2s(q_1)^{-1}s(q_1)s(q_2)=g_1s(q_1)g_2s(q_1)^{-1}\tau(q_1,q_2) s(q_1q_2).
\end{equation*}
Set $\tilde{g}:=g_1s(q_1)g_2s(q_1)^{-1}\tau(q_1,q_2)$, so
$\tilde{g}\in G$. We compute
\begin{equation*}
\begin{split}
E_{h_1h_2, [\rho]}&=\rho(\tilde{g})\circ T_{q_1q_2}^{[\rho]^{-1}}=\rho(g_1)\circ \rho(s(q_1)g_2 s(q_1)^{-1})\circ \rho(\tau(q_1, q_2))\circ T_{q_1q_2}^{[\rho]^{-1}}\\
&= \rho(g_1)\circ T_{q_1}^{[\rho]^{-1}}\circ q_1([\rho])(g_2)\circ T_{q_1}^{[\rho]}\circ \rho(\tau(q_1,q_2))\circ T_{q_1q_2}^{[\rho]^{-1}}       \quad \text{ by } (\ref{defining_property_intertwining_op})\\
&=\rho(g_1)\circ T_{q_1}^{[\rho]^{-1}}\circ q_1([\rho])(g_2)\circ T_{q_1}^{[\rho]}\circ T_{q_1}^{[\rho]^{-1}}\circ T_{q_2}^{q_1([\rho])^{-1}}c^{[\rho]}(q_1,q_2)  \quad \text{ by } (\ref{eq:dfn-c})\\
&= \rho(g_1)\circ T_{q_1}^{[\rho]^{-1}}\circ q_1([\rho])(g_2)\circ T_{q_2}^{q_1([\rho])^{-1}}c^{[\rho]}(q_1,q_2)\\
&= c^{[\rho]}(q_1, q_2) E_{h_1, [\rho]}\circ E_{h_2, h_1([\rho])},
\end{split}
\end{equation*}
as desired.
\end{proof}
\subsubsection{Equivalence}
We consider $H$-equivariant sheaves $\widetilde{\calw}$ on $M\times
\widehat{G}$ satisfying the following:
\begin{assump}\label{restriction-decomposition}
The restriction of $\widetilde{\calw}$ to $M\times \{[\rho]\}$ is
isomorphic, as a $G$-sheaf, to the tensor product of an ordinary
sheaf on $M$ and the trivial $G$-sheaf $V_\rho$.
\end{assump}
\begin{proposition}\label{equivalence2}
Tensoring with $\calv_G$ yields an equivalence between the category
of $c$-twisted $Q$-equivariant sheaves on $M\times \widehat{G}$ and
the category of $H$-equivariant sheaves on $M\times \widehat{G}$
satisfying Assumption \ref{restriction-decomposition}.
\end{proposition}
\begin{proof}
Let $\calw'$ be a $c$-twisted $Q$-equivariant sheaf on $M\times
\widehat{G}$. Let $$\Gamma_{q, [\rho]}: q^*\calw'|_{M\times
[\rho]}\to \calw'|_{M\times [\rho]},\quad q\in Q, [\rho]\in
\widehat{G}$$ be the $c$-twisted $Q$-action on $\calw'$. Let
$\calw'\otimes \calv_G$ be the sheaf on $M\times \widehat{G}$
defined by
$$(\calw'\otimes \calv_G)|_{M\times [\rho]}:=\calw'|_{M\times [\rho]}\otimes
\calv_G|_{M\times [\rho]}.$$ For $h\in H$, we write $h=gs(q)$ with
$g\in G, q\in Q$. We fix such an expression for each $h\in H$. Set
$$\gamma_{h, [\rho]}:= \Gamma_{q, [\rho]}\otimes E_{h, [\rho]}:
h^*(\calw'\otimes \calv_G)|_{M\times [\rho]}\to \calw'\otimes
\calv_G|_{M\times [\rho]}.$$ For $h_1, h_2\in H$, we calculate
\[
\begin{split}
\gamma_{h_1, [\rho]}\circ \gamma_{h_2, h_1([\rho])}&=(\Gamma_{q_1, [\rho]}\otimes E_{h_1, [\rho]})\circ (\Gamma_{q_2, q_1([\rho])}\otimes E_{h_2, h_1([\rho])})\\
&=(c^{[\rho]}(q_1, q_2)^{-1}\Gamma_{q_1q_2, [\rho]})\otimes (c^{[\rho]}(q_1, q_2)E_{h_1h_2, [\rho]})=\Gamma_{q_1q_2, [\rho]}\otimes E_{h_1h_2, [\rho]}=\gamma_{h_1h_2, [\rho]}.
\end{split}
\]
Therefore, the following collection defines an $H$-equivariant structure on $\calw'\otimes \calv_G$
\begin{equation}\label{H-action_on_sheaf}
\{\gamma_{h,[\rho]}\}, \quad h\in H, [\rho]\in \widehat{G}. 
\end{equation}

Let $\calu'$ and $\calw'$ be two $c$-twisted $Q$-equivariant sheaves
on $M\times \widehat{G}$ with the equivariant structures given,
respectively, by $\{\Gamma^\calu_{q, [\rho]}\}$ and
$\{\Gamma^\calw_{q, [\rho]}\}$, and let $\calu'\otimes \calv_G$ and
$\calw'\otimes \calv_G$ be the sheaves with $H$-equivariant
structures defined, respectively, by
$$\gamma^\calu_{h, [\rho]}=\Gamma^\calu_{q, [\rho]}\otimes E_{h, [\rho]},\qquad \gamma^\calw_{h, [\rho]}= \Gamma^\calw_{q, [\rho]}\otimes E_{h, [\rho]},$$
as in (\ref{H-action_on_sheaf}). By Schur's lemma, we have $Hom_H
(\calu'\otimes \calv_G, \calw'\otimes \calv_G)=Hom_{c, Q}(\calu',
\calw').$ Indeed, for $\phi\in Hom_H(\calu'\otimes \calv_G,
\calw'\otimes \calv_G)$, the map $\phi$ is $H$-equivariant. Hence,
it is also $G$-equivariant. It follows from Schur's lemma that the
restriction $\phi|_{M\times [\rho]}: \calu'\otimes \calv_G|_{M\times
[\rho]}\to \calw'\otimes \calv_G|_{M\times [\rho]}$ is of the form
$\bar{\phi}_{[\rho]}\otimes id$. The $H$-equivariance of $\phi$
reads $(\bar{\phi}_{[\rho]}\otimes id)\circ \gamma^U_{h,
[\rho]}=\gamma^W_{h,[\rho]}\circ (\bar{\phi}_{q([\rho])}\otimes
id).$ This implies that $\bar{\phi}_{[\rho]}\circ \Gamma^U_{q,
[\rho]}=\Gamma^W_{q, [\rho]}\circ \bar{\phi}_{q([\rho])}$. Thus, the
collection $\{\bar{\phi}_{[\rho]}| [\rho]\in \widehat{G}\}$ of maps
defines a $c$-twisted $Q$-equivariant map $\bar{\phi}: \calu'\to
\calw'$ of sheaves on $M\times \widehat{G}$.

Conversely, given a $c$-twisted $Q$-equivariant map $\bar{\phi}:
\calu'\to \calw'$, the map $\bar{\phi}\otimes id: \calu'\otimes
\calv_G\to \calw'\otimes \calv_G$ is equivariant with respect to the
$H$-actions defined in (\ref{H-action_on_sheaf}),
\begin{equation*}
\begin{split}
(\bar{\phi}_{[\rho]}\otimes id)\circ \gamma^\calu_{h, [\rho]}&=(\bar{\phi}_{[\rho]}\otimes id)\circ (\Gamma^\calu_{q, [\rho]}\otimes E_{h, [\rho]})=(\bar{\phi}_{[\rho]}\circ \Gamma^\calu_{q, [\rho]})\otimes E_{h, [\rho]}\\
&=(\Gamma^\calw_{q, [\rho]}\circ \bar{\phi}_{q([\rho])})\otimes E_{h, [\rho]}=(\Gamma^\calw_{q, [\rho]}\otimes E_{h, [\rho]})\circ (\bar{\phi}_{q([\rho])}\otimes id)=\gamma^\calw_{h, [\rho]}\circ (\bar{\phi}_{q([\rho])}\otimes id),
\end{split}
\end{equation*}
where we used the $Q$-equivariance of $\bar{\phi}$ in the middle
equality.

We have proved that $(-)\otimes \calv_G$ is a fully faithful functor
from the category of $c$-twisted $Q$-equivariant sheaves on $M\times
\widehat{G}$ to the category of $H$-equivariant sheaves on $M\times
\widehat{G}$ satisfying Assumption \ref{restriction-decomposition}.
To prove that this functor is an equivalence, we construct the
inverse functor. Let $\calw$ be an $H$-equivariant sheaf on $M\times
\widehat{G}$ satisfying Assumption \ref{restriction-decomposition}.
By Assumption \ref{restriction-decomposition}, we can write
\begin{equation}\label{splitting_as_G-sheaves}
\calw|_{M\times [\rho]}=\mathcal{H}om_G(V_{\rho}, \calw)\otimes
V_{\rho},
\end{equation}
as $G$-sheaves. Let $\widehat{\calw}$ be the sheaf over $M\times
\widehat{G}$ defined by $\widehat{\calw}|_{M\times [\rho]}:=
\mathcal{H}om_G(V_{\rho}, \calw).$ We will show that
$\widehat{\calw}$ is naturally a $c$-twisted $Q$-equivariant sheaf.

Let $\gamma_{h, [\rho]}: h^*\calw|_{M\times [\rho]}\to
\calw|_{M\times [\rho]}$ denote the $H$-equivariant structure on
$\calw$. In view of (\ref{splitting_as_G-sheaves}), we may assume
that
\begin{equation}\label{G-action}
\gamma_{g, [\rho]}=id\otimes \rho(g)\quad \text{ for } g\in G.
\end{equation}
Note that $h^*\calw|_{M\times [\rho]}=\mathcal{H}om_G(V_{h([\rho])},
\calw)\otimes V_{h([\rho])}$. Also, note that if we write $h=gs(q)$
with $g\in G$ and $q\in Q$, then $q([\rho])=h([\rho])$. Precomposing
with $T_q^{[\rho]}: V_{\rho}\to V_{q([\rho])}=V_{h([\rho])}$ defines
an isomorphism $\mathcal{H}om_G(V_{h([\rho])}, \calw)\to
\mathcal{H}om_G (V_{\rho}, \calw)$, which we denote by $T_{q,
[\rho]}^\vee$. Consider the composition
\[
\mathcal{H}om_G(V_{h([\rho])}, \calw)\otimes
V_{h([\rho])}\overset{\gamma_{h, [\rho]}}{\longrightarrow}
\calw|_{M\times [\rho]}=\mathcal{H}om_G(V_{\rho},
\calw)\otimes V_{\rho}\overset{T_{q, [\rho]}^{\vee-1}\otimes
E_{h,[\rho]}^{-1}}{\longrightarrow} \mathcal{H}om_G(V_{h([\rho])},
\calw)\otimes V_{h([\rho])}.
\]

\begin{claim}
The map $(T_{q, [\rho]}^{\vee-1}\otimes E_{h,[\rho]}^{-1})\circ
\gamma_{h, [\rho]}$ commutes with $id\otimes h([\rho])(g')$ for any
$g'\in G$.
\end{claim}
\begin{proof}[Proof of Claim]
We first compute
\begin{equation*}
\begin{split}
&(T_{q, [\rho]}^{\vee-1}\otimes E_{h,[\rho]}^{-1})\circ \gamma_{h, [\rho]}\circ (id\otimes h([\rho])(g'))\\
=&(T_{q, [\rho]}^{\vee-1}\otimes E_{h,[\rho]}^{-1})\circ \gamma_{h, [\rho]}\circ \gamma_{g', h([\rho])} \quad \text{ by }(\ref{G-action})\\
=&(T_{q, [\rho]}^{\vee-1}\otimes E_{h,[\rho]}^{-1})\circ \gamma_{hg', [\rho]} \quad \text{ since }\gamma_{h, [\rho]} \text{ is the equivariant structure on }\calw\\
=&(T_{q, [\rho]}^{\vee-1}\otimes E_{h,[\rho]}^{-1})\circ \gamma_{hg'h^{-1}, [\rho]}\circ \gamma_{h, (hg'h^{-1})([\rho])} \quad \text{ by the same reason}\\
=&(T_{q, [\rho]}^{\vee-1}\otimes E_{h,[\rho]}^{-1})\circ (id\otimes \rho(hg'h^{-1}))\circ \gamma_{h, (hg'h^{-1})([\rho])} \quad G \text{ is normal in } H, \text{ so } hg'h^{-1}\in G\\
=&(T_{q, [\rho]}^{\vee-1}\otimes E_{h,[\rho]}^{-1})\circ (id\otimes \rho(hg'h^{-1}))\circ \gamma_{h, [\rho]} \quad \text{ since } G \text{ fixes } [\rho]\\
=&(T_{q, [\rho]}^{\vee-1}\otimes (E_{h, [\rho]}^{-1}\circ
\rho(hg'h^{-1})))\circ \gamma_{h,[\rho]}.
\end{split}
\end{equation*}
Next we compute $E_{h, [\rho]}^{-1}\circ \rho(hg'h^{-1})$. By
definition, $E_{h, [\rho]}=\rho(g)T_q^{[\rho]^{-1}}$, where
$h=gs(q)$ with $g\in G, q\in Q$. As $G$ is normal in $H$, $s(q)gs(q)^{-1}\in G$ and 
$$\rho(hg'h^{-1})=\rho(gs(q)g's(q)^{-1}g^{-1})=\rho(g)\rho(s(q)g's(q)^{-1})\rho(g)^{-1}.$$
The claim follows from the following calculation. 
\begin{equation*}
\begin{split}
E_{h, [\rho]}^{-1}\circ \rho(hg'h^{-1})&=T_q^{[\rho]}\rho(g)^{-1}\rho(g)\rho(s(q)g's(q)^{-1})\rho(g)^{-1}=T_q^{[\rho]}\rho(s(q)g's(q)^{-1})\rho(g)^{-1}\\
&=T_q^{[\rho]}T_q^{[\rho]^{-1}}\circ q([\rho])(g')\circ T_q^{[\rho]}\circ \rho(g)^{-1}\quad \text{ by }(\ref{defining_property_intertwining_op})\\
&=q([\rho])(g')\circ T_q^{[\rho]}\circ \rho(g)^{-1}=h([\rho])(g')\circ E_{h, [\rho]}^{-1} \quad \text{ because }
h([\rho])=q([\rho]).
\end{split}
\end{equation*}
\end{proof}
By the claim and Schur's lemma, we have
$$(T_{q, [\rho]}^{\vee-1}\otimes E_{h,[\rho]}^{-1})\circ \gamma_{h,
[\rho]}=\Gamma'_{h, [\rho]}\otimes id$$ for some sheaf map
$\Gamma'_{h, [\rho]}: \mathcal{H}om_G(V_{h([\rho])}, \calw)\to
\mathcal{H}om_G(V_{h([\rho])}, \calw)$. Let
$$\Gamma_{h, [\rho]}:=T_{q, [\rho]}^{\vee}\circ \Gamma'_{h, [\rho]}: h^*\widehat{\calw}|_{M\times [\rho]}
=\mathcal{H}om_G(V_{h([\rho])}, \calw)\to
\mathcal{H}om_G(V_{\rho}, \calw)=\widehat{\calw}|_{M\times
[\rho]}.$$ Then $\gamma_{h, [\rho]}=\Gamma_{h, [\rho]}\otimes E_{h,
[\rho]}$. We can easily check that the collection $\{\Gamma_{h, [\rho]}\}$
defines $c$-twisted $H$-equivariant and $Q$-equivariant structures
on $\widehat{\calw}$.

Using the properties of $\gamma_{h, [\rho]}$ and $E_{h, [\rho]}$
that were discussed above, we can compute
\[
\Gamma_{h_1h_2, [\rho]}\otimes id =c^{[\rho]}(h_1, h_2)^{-1}(\Gamma_{h_1,[\rho]}\circ \Gamma_{h_2,
h_1([\rho])}\otimes id).
\]
Thus,
\begin{equation}\label{H-equivariant_structure_on_hatW}
\Gamma_{h_1h_2, [\rho]}=c^{[\rho]}(h_1,
h_2)^{-1}\Gamma_{h_1,[\rho]}\circ \Gamma_{h_2, h_1([\rho])}.
\end{equation}
In other words, $\{\Gamma_{h, [\rho]}| h\in H, [\rho]\in
\widehat{G}\}$ defines a $c$-twisted $H$-equivariant structure on
$\widehat{\calw}$.

Note that $\Gamma_{g, [\rho]}=id$ for $g\in G$, by (\ref{G-action}).
A special case of (\ref{H-equivariant_structure_on_hatW}) reads
$\Gamma_{g, [\rho]}\circ \Gamma_{h,
[\rho]}=c^{[\rho]}(g,h)\Gamma_{gh, [\rho]},$ where $g\in G, h\in H$
and note that $g([\rho])=[\rho]$. We claim that $c^{[\rho]}(g,h)=1$
for all $g\in G, h\in H$. To see this, note that by Proposition
\ref{prop:u(1)-cocycle}, $c^{[\rho]}(g,h)=c^{[\rho]}(1,q)=1$ for
$q\in Q$ such that $hs(q)^{-1}\in G$.

By the discussion above, we find that $\Gamma_{gh,
[\rho]}=\Gamma_{h, [\rho]}$ for all $g\in G$. Therefore, $\Gamma_{h,
[\rho]}$ depends only on the $G$-coset of $h$, not the element $h$
itself. Moreover, for $q\in Q$, the definition $\Gamma_{q, [\rho]}:=\Gamma_{s(q), [\rho]}$ is independent of the
choice of the section $s: Q\to H$. It follows from
(\ref{H-equivariant_structure_on_hatW}) that $\{\Gamma_{q, [\rho]}|
q\in Q, [\rho]\in \widehat{G}\}$ defines a $c$-twisted
$Q$-equivariant structure on $\widehat{\calw}$.

It is straightforward to check that the functor $\calw\mapsto
\widehat{\calw}$ is the inverse of the functor $\calw'\mapsto
\calw'\otimes \calv_G$. The proposition is proved.
\end{proof}

\begin{lemma}\label{equivalence1}
The category of $H$-equivariant sheaves on $M\times \widehat{G}$ satisfying
Assumption \ref{restriction-decomposition} is equivalent to the category of $H$-equivariant sheaves on $M$.
\end{lemma}
\begin{proof}
Let $\widetilde{\calw}$ be an $H$-equivariant sheaf on $M\times
\widehat{G}$. Then the direct sum $\bigoplus_{[\rho]\in
\widehat{G}}\widetilde{\calw}|_{M\times \{[\rho]\}}$ is a sheaf on
$M$ with a natural $H$-equivariant structure induced from that of
$\widetilde{\calw}$. Clearly,
\[
Hom_H(\widetilde{\calu},
\widetilde{\calw})=Hom_H(\bigoplus_{[\rho]\in\widehat{G}}\widetilde{\calu}|_{M\times
\{[\rho]\}},\bigoplus_{[\rho]\in\widehat{G}}\widetilde{\calw}|_{M\times
\{[\rho]\}}).
\]
Hence, the assignment $\widetilde{\calw}\mapsto
\bigoplus_{[\rho]\in \widehat{G}}\widetilde{\calw}|_{M\times
\{[\rho]\}}$ is a covariant fully faithful functor from the category
of $H$-equivariant sheaves on $M\times \widehat{G}$ satisfying
Assumption \ref{restriction-decomposition} to the category of
$H$-equivariant sheaves on $M$. It remains to construct an inverse
functor.

Let $\calw$ be an $H$-equivariant sheaf on $M$. Since $G$ acts
trivially on $M$, we have the following canonical decomposition as
$G$-equivariant sheaves (see, e.g., \cite{bkr}, Sec.  4.2):
\begin{equation}\label{canonical_decomp}
\calw=\bigoplus_{[\rho]\in \widehat{G}} \mathcal{H}om_G(V_\rho,
\calw)\otimes V_\rho,
\end{equation}
where $V_\rho$ is again the trivial vector bundle over $M$ with a
$G$-action given by $\rho$, and $\mathcal{H}om_G(V_\rho, \calw)$ is
just an ordinary sheaf on $M$. Define a sheaf $\widetilde{\calw}$ on
$M\times \widehat{G}$ by
$\widetilde{\calw}|_{M\times \{[\rho]\}}:= \mathcal{H}om_G(V_\rho, \calw)\otimes V_\rho.$
Clearly, $\widetilde{\calw}$ satisfies Assumption
\ref{restriction-decomposition}. We claim that $\widetilde{\calw}$
has the structure of an $H$-equivariant sheaf.

Since the projection $\mathbf{p}: M\times \widehat{G}\to M$ onto the
first factor is $H$-equivariant, the pull-back $\mathbf{p}^*\calw$
is an $H$-equivariant sheaf on $M\times \widehat{G}$. Clearly, for
any $[\rho]\in \widehat{G}$, we have $\mathbf{p}^*\calw|_{M\times
\{[\rho]\}}=\calw$. Also,
$$
\mathcal{H}om_G(\calv_G, \mathbf{p}^*\calw)|_{M\times
\{[\rho]\}}=\mathcal{H}om_G(V_\rho, \calw).
$$
It follows that $\mathcal{H}om_G(V_\rho, \calw)\otimes V_\rho=
\mathcal{H}om_G(\calv_G, \mathbf{p}^*\calw)|_{M\times
\{[\rho]\}}\otimes \calv_G|_{M\times \{[\rho]\}}$, i.e.,
$$
\widetilde{\calw}=\mathcal{H}om_G(\calv_G, \mathbf{p}^*\calw)\otimes \calv_G.
$$
By Lemma \ref{lem:Q_str_on_V}, $\calv_G$ is a $c^{-1}$-twisted
$H$-equivariant sheaf. As $\mathcal{H}om_G(\calv_G,
\mathbf{p}^*\calw)$ is a $c$-twisted $H$-equivariant sheaf,
$\widetilde{\calw}$, being the tensor product of
$\mathcal{H}om_G(\calv_G, \mathbf{p}^*\calw)$ and $\calv_G$, is an
$H$-equivariant sheaf.

For two $H$-equivariant sheaves $\calu$ and $\calw$ on $M$, it is
easy to see that $Hom_H(\calu,\calw)=Hom_H(\widetilde{\calu},
\widetilde{\calw})$. The functor $\calw\mapsto \widetilde{\calw}$
provides the needed inverse functor.
\end{proof}

It is known that sheaves on the orbifold $[M/H]$ are equivalent to
$H$-equivariant sheaves on $M$. The cocycle $c$ defines a cocycle,
which we still denote by $c$, on the underlying orbifold $[(M\times
\widehat{G})/Q]$. It is also known that $c$-twisted sheaves on the
orbifold $[(M\times \widehat{G})/Q]$ are equivalent to $c$-twisted
$Q$-equivariant sheaves on $M\times \widehat{G}$. We may
combine Proposition \ref{equivalence2} with Lemma \ref{equivalence1}
to deduce the following theorem:
\begin{theorem}\label{equivalence3}
The categories of ({\em $c$-twisted}) sheaves on the gerbe $[M/H]$ and $[(M\times \widehat{G})/Q]$ are equivalent.
\end{theorem}

\subsection{General case}
In this section we discuss sheaf theory on a general $G$-gerbe. Let
$\Y\to \B$ be a $G$-gerbe over an orbifold $\B$. If
$\frakH\rightrightarrows \frakH_0$ is an \'etale groupoid presenting
the gerbe $\Y$, then sheaves on $\Y$ are equivalent to
$\frakH$-sheaves on $\frakH_0$. Therefore, in order to study sheaves
on the gerbe $\Y$, we may pick a suitable \'etale groupoid
presentation $\frakH$ of $\Y$ and work with $\frakH$-sheaves. As discussed in Sec.  \ref{subsec:general}, we may choose proper
\'etale groupoids $\frakH$ and $\frakQ$ so that the $G$-gerbe $\Y\to
\B$ is presented by the groupoid extension
\begin{equation}\label{grp_extension_presentation_of_gerbe}
(M\times G\rightrightarrows
M)\overset{i}{\longrightarrow}\frakH\overset{j}{\longrightarrow}\frakQ,
\end{equation}
such that
\begin{enumerate}
\item
$M\times G\rightrightarrows M$ is the groupoid for the trivial
action of $G$ on $M$.
\item
$\frakH\rightrightarrows\frakH_0$ with $\frakH_0=M$ is a
presentation of $\Y$.

\item
$\frakQ\rightrightarrows\frakQ_0$ with $\frakQ_0=M$ is a
presentation of $\B$.

\item
$i|_M=j|_M=\text{identity}$.

\item
there is a section of $j$, i.e., a map $\sigma: \frakQ\to \frakH$,
such that $j\circ \sigma=id$ and $\sigma|_M=\text{identity}$.
\end{enumerate}

\begin{remark}
In the algebraic context, i.e. when $\Y$ and $\B$ are
Deligne-Mumford $\mathbb{C}$-stacks, a presentation of $\Y\to
\B$ as in (\ref{grp_extension_presentation_of_gerbe}) can be
obtained as follows. By \cite{av}, Lemma 2.2.3, we may find an
\'etale cover $U:=\coprod_i U_i\to \Y$ such that $\Y$ is locally
isomorphic to a quotient $[U_i/H_i]$ by some finite group $H_i$
acting on $U_i$. Since $\Y$ is a $G$-gerbe, the group $H_i$ contains
$G$ as a normal subgroup, and the induced $G$-action on $U_i$ is
trivial. Set $Q_i:= H_i/G$. Then $\B$ is locally isomorphic to the
quotient $[U_i/Q_i]$, and the map $\Y\to \B$ is locally presented as
$[U_i/H_i]\to [U_i/Q_i]$. We may take $M=U$, $\frakH:=(U\times_\Y
U\rightrightarrows U)$, and $\frakQ:=(U\times_\B U\rightrightarrows
U)$. Choosing a section $\sigma: \frakQ\to \frakH$ amounts to
choosing sections $Q_i\to H_i$ and $Q_{ij}\to H_{ij}$, where
$Q_{ij}$ and $H_{ij}$ are finite groups, so that over $U_i\times_\Y
U_j$, the map $\Y\to \B$ is presented as $[V_{ij}/H_{ij}]\to
[V_{ij}/Q_{ij}]$.
\end{remark}

We now proceed to study sheaves on the gerbe $\Y$ by studying
$\frakH$-sheaves on $M=\frakH_0$, in a way similar to the treatment
in Sec.  \ref{sheaf_on_global_quotient}.

As discussed in Sec.  \ref{subsec:general}, the groupoid $\frakQ$
acts on $\widehat{G}$. Similarly, $\frakH$ acts on $\widehat{G}$ as
well. Indeed, the $\frakH$-action on $\widehat{G}$ is obtained from
the $\frakQ$-action by the map $j:\frakH\to \frakQ$. Consider the
two transformation groupoids $\widehat{\frakH}:=\widehat{G}\rtimes
\frakH$ and $\widehat{\frakQ}:=\widehat{G}\rtimes \frakQ$. There is
a groupoid cocycle $c$ on $\widehat{\frakQ}$. Note that
$\widehat{\frakH}_0=\widehat{\frakQ}_0=M\times \widehat{G}$.

Let $\calv_G$ be the sheaf over $M\times \widehat{G}$ defined by
requiring that its restriction to $M\times [\rho]$ be the trivial
sheaf with fiber $V_\rho$.  There is a natural $G$-sheaf structure
on $\calv_G$. Similar to the method in Sec. 
\ref{sheaf_on_global_quotient}, we can construct a $c^{-1}$-twisted
$\wfrakH$-equivariant structure on $\calv_G$. Here, $c$ is the cocycle in Proposition
\ref{prop:u(1)-cocycle-gpd}. Evidently, $c$ can be extended to a
cocycle $\wfrakH\times_{\wfrakH_0} \wfrakH\to U(1)$ via the map $j:
\frakH\to \frakQ$.

We generalize Proposition \ref{equivalence2} to $\frakH$-sheaves $\widetilde{\calw}$ on $M\times
\widehat{G}$ satisfying the following:
\begin{assump}\label{restriction-decomposition_general}
The restriction of $\widetilde{\calw}$ to $M\times \{[\rho]\}$ is
isomorphic, as $G$-sheaves, to the tensor product of an ordinary
sheaf on $M$ and the trivial $G$-sheaf $V_\rho$.
\end{assump}
\begin{proposition}\label{equivalence2_general}
Tensoring with $\calv_G$ yields an equivalence between the category
of $c$-twisted $\wfrakQ$-sheaves on $M\times \widehat{G}$ and the
category of $\wfrakH$-sheaves on $M\times \widehat{G}$ satisfying
Assumption \ref{restriction-decomposition_general}.
\end{proposition}
\begin{proof}
The proof is a straightforward modification of the proof of
Proposition \ref{equivalence2}. We omit the details.
\end{proof}

Let $\calw$ be a $\frakH$-sheaf on $M$. Since $G$ acts trivially on
$M$, we have the following canonical decomposition as
$G$-equivariant sheaves:
\begin{equation}\label{canonical_decomp_general}
\calw=\bigoplus_{[\rho]\in \widehat{G}} \mathcal{H}om_G(V_\rho,
\calw)\otimes V_\rho,
\end{equation}
where $V_\rho$ is, again, the trivial vector bundle over $M$ with a
$G$-action given by $\rho$, and $\mathcal{H}om_G(V_\rho, \calw)$ is
just an ordinary sheaf on $M$. Given $\calw$ as above, define a sheaf $\widetilde{\calw}$ on
$M\times \widehat{G}$ by
$\widetilde{\calw}|_{M\times \{[\rho]\}}:= \mathcal{H}om_G(V_\rho, \calw)\otimes V_\rho.$
A generalization of Lemma \ref{equivalence1} is immediate. We omit the proof.
\begin{lemma}\label{equivalence1_general}
The assignment $\calw\mapsto \widetilde{\calw}$ defines an
equivalence between the category of $\frakH$-sheaves on $M$ and the
category of $\wfrakH$-sheaves on $M\times \widehat{G}$ satisfying
Assumption \ref{restriction-decomposition_general}.
\end{lemma}
Combining Proposition \ref{equivalence2_general} with Lemma
\ref{equivalence1_general}, we obtain the following:
\begin{theorem}
The category of $\frakH$-sheaves is equivalent to the
category of $c$-twisted $\wfrakQ$-sheaves.
\end{theorem}

Observe that the groupoid $\wfrakQ$ is a presentation of the
orbifold $\widehat{\Y}$, which is dual to the gerbe $\Y\to \B$, and
the cocycle $c$ defines a flat $U(1)$-gerbe on $\widehat{\Y}$ (which
we still denote by $c$). It is known that sheaves on $\Y$ are
equivalent to $\frakH$-sheaves on $M$, and $c$-twisted sheaves on
$\widehat{\Y}$ are equivalent to $c$-twisted $\wfrakQ$-sheaves on
$M\times \widehat{G}$. Therefore, we may rephrase the above theorem
as follows:
\begin{theorem}\label{equivalence3_general}
The category of sheaves on $\Y$ is equivalent to the category of
$c$-twisted sheaves on $\widehat{\Y}$.
\end{theorem}

\begin{remark}
\hfill
\begin{enumerate}
\item
Our arguments in this section are valid in the
algebro-geometric context. Hence, the main results of this section,
as well as their counterparts for (quasi-)coherent sheaves, hold for
$G$-gerbes over Deligne-Mumford stacks as well. For example, the
abelian category of (quasi-)coherent sheaves on $\Y$ is equivalent
to the abelian category of $c$-twisted (quasi-)coherent sheaves on
$\widehat{\Y}$.
\item
The equivalences in Theorems \ref{equivalence3_general} and
\ref{equivalence3} {\em do not} preserve tensor product structure.
The main reason is that the tensor product of two $c$-twisted
sheaves is not a $c$-twisted sheaf. However, the equivalence is
compatible with tensor products by ``invariant'' vector bundles.
More precisely, let $F: \Sh(\Y)\to \Sh(\widehat{\Y})$ be
the equivalence in Theorem \ref{equivalence3_general}. Let $\pi_\Y:
\Y\to \B$ and $\pi_{\widehat{\Y}}: \widehat{\Y}\to \B$ denote the
natural maps. Suppose that $V\to \B$ is a vector bundle such that at
any point $x\in \B$, the action of the isotropy group
$\text{Iso}(x)$ on the fiber $V_x$ is trivial\footnote{If $\B$ is a
Deligne-Mumford stack, such a vector bundle $V$ is pulled back from
the coarse moduli space of $\B$.}. Then for any sheaf $\F$ on $\Y$, it
follows easily from the construction of the functor $F$ that $F(\F\otimes \pi_\Y^*V)=F(\F)\otimes \pi_{\widehat{\Y}}^*V.$
\end{enumerate}
\end{remark}


\begin{thebibliography}{12}
\bibitem{abramovich} D. Abramovich, {\em Lectures on Gromov-Witten invariants of orbifolds}, In: Enumerative invariants in algebraic geometry and string theory, 1--48, Lecture Notes in Math., \textbf{1947}, Springer, Berlin, (2008).

\bibitem{agv1} D. Abramovich, T. Graber, A. Vistoli, {\em Algebraic orbifold quantum products}, In:  \cite{Adem_Morava_Ruan}, 1--24.

\bibitem{agv2} D. Abramovich, T. Graber, A. Vistoli, {\em Gromov-Witten theory of Deligne-Mumford stacks}, Amer. J. Math., \textbf{130} (2008), no. 5, 1337--1398.

\bibitem{av} D. Abramovich, A. Vistoli, {\em Compactifying the space of stable maps}, J. Amer. Math. Soc.  \textbf{15}  (2002),  no. 1, 27--75.

\bibitem{Adem_Leida_Ruan} A. Adem, J. Leida, Y. Ruan, {\em Orbifolds and stringy topology}, Cambridge Tracts in Mathematics, \textbf{171}. Cambridge University Press, Cambridge, (2007).

\bibitem{Adem_Morava_Ruan} A. Adem, J. Morava, Y. Ruan (Ed.), {\em Orbifolds in mathematics and physics (Madison, WI, 2001)}, Contemp. Math., \textbf{310}, Amer. Math. Soc. Providence, RI, (2002).

\bibitem{adem_ruan} A. Adem, Y. Ruan, {\em Twisted orbifold $K$-theory}, Comm. Math. Phys. \textbf{237} (2003), no. 3, 533--556.

\bibitem{an-fu}F. Anderson, K. Fuller, {\em Rings and Categories of Modules}, Graduate Texts in Mathematics, Vol. \textbf{13}, 2nd Ed., Springer-Verlag, New York, (1992).

\bibitem{ajt1} E. Andreini, Y. Jiang, H.-H. Tseng, {\em Gromov-Witten theory of product stacks}, arXiv:0905.2258.

\bibitem{ajt2} E. Andreini, Y. Jiang, H.-H. Tseng, {\em Gromov-Witten theory of root gerbes, I: structure of genus $0$ moduli spaces}, to appear in J. Differential Geom., arXiv:0907.2087.

\bibitem{ajt2.5} E. Andreini, Y. Jiang, H.-H. Tseng, {\em Gromov-Witten theory of banded gerbes over schemes}, arXiv:1101.5996.

\bibitem{ajt3} E. Andreini, Y. Jiang, H.-H. Tseng, in preparation.

\bibitem{bn} K. Behrend, B. Noohi, {\em Uniformization of Deligne-Mumford curves},  J. Reine Angew. Math. \textbf{599} (2006), 111--153.

\bibitem{be-xu} K. Behrend, P. Xu, {\em Differentiable stacks and gerbes}, J. Symplectic Geom. \textbf{9} (2011), no. 3, 285--341, arXiv:math/0605694.

\bibitem{br} L. Breen, {\em Notes on $1$- and $2$-gerbes}, In: Towards Higher Categories, 193--235, IMA Vol. Math. Appl., \textbf{152}, Springer, New York (2010), arXiv:math/0611317.

\bibitem{bgnt}P. Bressler, A. Gorokhovsky, R. Nest, B. Tsygan, {\em Deformation quantization of gerbes}, Adv. Math. \textbf{214} (2007), no. 1, 230--266.

\bibitem{bkr} T. Bridgeland, A. King, M. Reid, {\em The McKay correspondence as an equivalence of derived categories}, J. Amer. Math. Soc. {\textbf 14} (2001), no. 3, 535--554.

\bibitem{bry} J. Brylinski, {\em Loop spaces, characteristic classes and geometric quantization}, Reprint of the 1993 edition, Modern Birkh\"auser Classics, Birkh\"auser Boston, Inc., Boston, MA, (2008).

\bibitem{br-ni} J. Brylinski, V. Nistor, {\em Cyclic cohomology of \'etale groupoids}, $K$-Theory \textbf{8} (1994), no. 4, 341--365.

\bibitem{cal} A. Caldararu, {\em Derived Categories of Twisted Sheaves on Calabi-Yau Manifolds}, Ph.D. Thesis, Cornell University, (2000).

\bibitem{cr2} W. Chen, Y. Ruan, {\em Orbifold Gromov-Witten theory}, In: \cite{Adem_Morava_Ruan}, 25--85.

\bibitem{cr}W. Chen, Y. Ruan, {\em A new cohomology theory of orbifold}, Comm. Math. Phys. \textbf{248} (2004), no. 1, 1--31.

\bibitem{cl}A. Clifford, {\em Representations induced in an invariant subgroup}, Ann. of Math. (2) \textbf{38} (1937), no. 3, 533--550.

\bibitem{c}A. Connes,  {\em Noncommutative Geometry},  Academic Press (San Diego), (1994).

\bibitem{crainic}M. Crainic, {\em Cyclic cohomology of \'etale groupoids: the general case}, $K$-theory, \textbf{17} (1999), no. 4, 319--362.

\bibitem{deligne-mumford} P. Deligne, D. Mumford, {\em The irreducibility of the space of curves of given genus}, Inst. Hautes \'Etudes Sci. Publ. Math. No. \textbf{36} (1969), 75--109.

\bibitem{EHKV} D. Edidin, B. Hassett, A. Kresch, A. Vistoli, {\em Brauer groups and quotient stacks}, Amer. J. Math. \textbf{123} (2001), no. 4, 761--777.

\bibitem{fe-do} J. Fell, R. Doran, {\em Representations of $\sp *$-algebras, locally compact groups, and Banach $*$-algebraic bundles}, Vol. 2. Banach $*$-algebraic bundles, induced representations, and the generalized Mackey analysis. Pure and Applied Mathematics \textbf{126}, Academic Press, Inc., Boston, MA, (1988).

\bibitem{fu-ha} W. Fulton, J. Harris, {\em Representation theory}, Graduate Texts in Mathematics, \textbf{129}, Readings in Mathematics, Springer-Verlag, New York, (1991).

\bibitem{gt} A. Gholampour, H.-H. Tseng, {\em On Donaldson-Thomas invariants of threefold stacks and gerbes}, Proc. Amer. Math. Soc. \textbf{141} (2013), no. 1, 191--203. 

\bibitem{gi} J. Giraud, {\em Cohomologie non ab\'elienne}, Die Grundlehren der mathematischen Wissenschaften \textbf{179}, Springer-Verlag, (1971).

\bibitem{givental_quantization} A. Givental, {\em Gromov-Witten invariants and quantization of quadratic Hamiltonians}, Mosc. Math. J. \textbf{1} (2001), no. 4, 551--568.

\bibitem{ha} A. Haefliger, {\em Groupo\" ides d'holonomie et classifiants}, In: Transversal structure of foliations (Toulouse, 1982). Ast\'erisque No. \textbf{116}, (1984), 70--97.

\bibitem{hel-hen-pan-sh} S. Hellerman, A. Henriques, T. Pantev, E. Sharpe, M. Ando, {\em Cluster decomposition, $T$-duality, and gerby CFTs}, Adv.
Theor. Math. Phys. \textbf{11} (2007), no. 5, 751--818.

\bibitem{hepworth} R. Hepworth, {\em The age grading and the Chen-Ruan cup product},  Bull. Lond. Math. Soc. \textbf{42} (2010), no. 5, 868--878,

\bibitem{jk} T. Jarvis, T. Kimura, {\em Orbifold quantum cohomology of the classifying space of a finite group}, In: \cite{Adem_Morava_Ruan}, 123--134.

\bibitem{jkk} T. Jarvis, R. Kaufmann, T. Kimura, {\em Stringy $K$-theory and the Chern character}, Invent. Math. \textbf{168} (2007), no. 1, 23--81.

\bibitem{johnson} P. Johnson, {\em Equivariant Gromov-Witten theory of one dimensional stacks}, arXiv:0903.1068.

\bibitem{kar} G. Karpilovsky, {\em Projective Representations of Finite Groups}, Monographs in Pure and Applied Mathematics \textbf{94}, New York: Marcel Dekker Inc. (1985).

\bibitem{lm-b} G. Laumon, L. Moret-Bailly, {\em Champs alg\'ebriques}, Ergebnisse der Mathematik und ihrer Grenzgebiete. 3. Folge. A Series of Modern Surveys in Mathematics, \textbf{39}, Springer-Verlag, Berlin, (2000).

\bibitem{la-st-xu-1} C. Laurent-Gengoux, J. Tu, P. Xu,  {\em Chern-Weil map for principal bundles over groupoids}, Math. Z. \textbf{255} (2007), no. 3, 451--491.

\bibitem{la-st-xu} C. Laurent-Gengoux, M. Sti\'enon, P. Xu, {\em Non-abelian differentiable gerbes}, Adv. Math. \textbf{220} (2009), no. 5, 1357--1427.

\bibitem{Lieblich} M. Lieblich, {\em Moduli of twisted sheaves}, Duke Math. J., \textbf{138} (2007), no. 1, 23--118.

\bibitem{loday} J. Loday, {\em Cyclic homology}, Appendix E by Mar\'ia O. Ronco, Second edition,  Grundlehren der Mathematischen Wissenschaften, \textbf{301}, Springer-Verlag, Berlin, (1998).

\bibitem{manin} Y. Manin, {\em Frobenius manifolds, quantum cohomology, and moduli spaces},  AMS Colloquium Publications \textbf{47}, Amer. Math. Soc., Providence, RI, (1999).

\bibitem{mo} I. Moerdijk,  {\em Orbifolds as groupoids: an introduction},  In: \cite{Adem_Morava_Ruan}, 205--222.

\bibitem{mo-pr} I. Moerdijk, D. Pronk, {\em Orbifolds, sheaves and groupoids}, $K$-Theory \textbf{12} (1997), no. 1, 3--21.

\bibitem{mo-pr-ind} I. Moerdijk, D. Pronk, {\em A. Simplicial cohomology of orbifolds}, Indag. Math. (N.S.) \textbf{10} (1999), no. 2, 269--293.

\bibitem{rmw} P. Muhly, J. Renault, D. Williams, {\em Equivalence and isomorphism for groupoid $C\sp \ast$-algebras},  J. Operator Theory \textbf{17} (1987), no. 1, 3--22.

\bibitem{me} R. Meyer, {\em Analytic Cyclic Cohomology}, Ph.D.~Thesis, M\"unster, arXiv:math.KT/9906205.

\bibitem{nppt} N. Neumaier, M.~Pflaum, H.~Posthuma, X.~Tang,  {\em Homology of formal deformations of proper \'etale Lie groupoids}, J. Reine Angew. Math. \textbf{593}  (2006), 117--168.

\bibitem{pry} J. Pan, Y. Ruan, X. Yin, {\em Gerbes and twisted orbifold quantum cohomology}, Sci. China Ser. A \textbf{51} (2008), no. 6, 995--1016.

\bibitem{pptt} M. Pflaum, H. Posthuma, X. Tang, H.-H. Tseng, {\em Orbifold cup products and ring structures on Hochschild cohomologies},  Commun. Contemp. Math. \textbf{13} (2011), no. 1, 123--182.

\bibitem{ro} D. Robinson, {\em A Course in the Theory of Groups}, 2nd edition, Graduate Texts in Mathematics \textbf{80}, Springer-Verlag, (1996).

\bibitem{ru1} Y.  Ruan, {\em Discrete torsion and twisted orbifold cohomology}, J. Symplectic Geom. \textbf{2} (2003), no. 1, 1--24.

\bibitem{ruan_0011149} Y. Ruan, {\em Stringy Geometry and Topology of Orbifolds}, In: Symposium in Honor of C. H. Clemens (Salt Lake City, UT, (2000), 187--233, Contemp. Math., \textbf{312}, Amer. Math. Soc., Providence, RI, (2002).

\bibitem{ruan_0201123} Y. Ruan, {\em Stringy orbifolds}, In: \cite{Adem_Morava_Ruan}, 259--299.

\bibitem{sa} I. Satake, {\em On a generalization of the notion of manifold}, Proc. Nat. Acad. Sci. U.S.A., \textbf{42} (1956), 359--363.

\bibitem{ta} X. Tang, {\em Deformation quantization of pseudo Poisson groupoids}, Geom. Funct. Anal. \textbf{16} (2006), no. 3, 731--766.

\bibitem{th} W. Thurston, {\em The Geometry and Topology of Three-Manifolds} (Chapter 13), Princeton University lecture notes (1978--1981).

\bibitem{tseng} H.- H. Tseng, {\em Orbifold Quantum Riemann-Roch, Lefschetz and Serre}, Geom. Topol. \textbf{14} (2010) 1--81.

\bibitem{tu-xu} J. Tu, P. Xu, {\em Chern character for twisted $K$-theory of orbifolds}, Adv. Math. \textbf{207} (2006), no. 2, 455--483.

\bibitem{tu-la-xu} J. Tu, P. Xu, C. Laurent-Gengoux, {\em Twisted $K$-theory of differentiable stacks}, Ann. Sci. \'Ecole Norm. Sup. (4) \textbf{37} (2004), no. 6, 841--910.
\end{thebibliography}
\end{document}